\documentclass{amsart}
\RequirePackage{amssymb}

\usepackage[all]{xy}

\usepackage[latin1]{inputenc}
\theoremstyle{plain}
 
\newtheorem{theorem}{Theorem}[section]
\newtheorem{proposition}[theorem]{Proposition}
\newtheorem{lemma}[theorem]{Lemma}
\newtheorem{corollary}[theorem]{Corollary}
 
\theoremstyle{definition}
 
\newtheorem{definition}[theorem]{Definition}

\newtheorem{remark}[theorem]{Remark}

\newcommand{\me}{\mathfrak}
\newcommand{\mf}{\mathfrak}

\newcommand{\C}{\mathbf{C}}
\newcommand{\R}{\mathbf{R}}
\newcommand{\N}{\mathbf{N}}
\newcommand{\Z}{\mathbf{Z}}

\newcommand{\I}{\mathbf{I}}
\newcommand{\E}{\mathcal{E}}
\newcommand{\ERf}{\mathcal{E}^{\R,f}}
\newcommand{\ER}{\mathcal{E}^\R}
\newcommand{\Esa}{\mathcal{E}^{\sha}}
\newcommand{\Eind}{\mathcal{E}^{ind}}
\newcommand{\Eindd}{\mathcal{E}'^{ind}}
\renewcommand{\O}{\mathcal{O}}
\newcommand{\oot}{\mathbb{O}}
\newcommand{\D}{\mathcal{D}}
\newcommand{\shan}{\mathcal{C}^\omega}
\newcommand{\CP}{\mathbf{P}}

\newcommand{\CC}{\mathrm{C}}
\newcommand{\DC}{\mathrm{D}}
\newcommand{\K}{\mathrm{K}}
\newcommand{\BDC}{\mathrm{D}^b}
\newcommand{\Mod}{\mathrm{Mod}}
\newcommand{\Ind}{\mathrm{Ind}}

\newcommand{\newindlim}{\operatornamewithlimits{\text{``}\varinjlim\text{''}}}
\newcommand{\grosfrac}[2]{\frac{\textstyle #1}{\textstyle #2}}

\newcommand{\cl}{\colon}
\renewcommand{\to}{\rightarrow}
\newcommand{\from}{\leftarrow}

\newcommand{\xto}[2][]{\xrightarrow[#1]{#2}}
\newcommand{\xfrom}[2][]{\xleftarrow[#1]{#2}}
\newcommand{\isoto}[1][]{\xrightarrow[#1]{\raisebox{-2pt}[0pt][0pt]{\makebox[0cm]{$\scriptstyle\sim$}}}}
\newcommand{\isofrom}[1][]{\xleftarrow[#1]{\raisebox{-2pt}[0pt][0pt]{\makebox[0cm]{$\scriptstyle\backsim$}}}}
\newcommand{\petittimes}{\makebox{$\times$}}

\newcommand{\sect}{\Gamma}
\newcommand{\rsect}{R\Gamma}

\newcommand{\RHom}{\operatorname{RHom}}
\newcommand{\Hom}{\operatorname{Hom}}

\renewcommand{\hom}{\operatorname{\mathcal{H}om}}
\newcommand{\ihom}{\operatorname{\mathcal{IH}om}}
\newcommand{\Rihom}{\operatorname{R\mathcal{IH}om}}
\newcommand{\Rhom}{\operatorname{R\mathcal{H}om}}
\newcommand{\supp}{\operatorname{supp}}
\newcommand{\tmu}{T\!\!{-}\!\mu}
\newcommand{\For}{\operatorname{For}}

\newcommand{\id}{\mathrm{id}}

\newcommand{\sha}{\mathcal{A}}
\newcommand{\shb}{\mathcal{B}}
\newcommand{\shc}{\mathcal{C}}

\newcommand{\shl}{\mathcal{L}}
\newcommand{\shm}{\mathcal{M}}
\newcommand{\shn}{\mathcal{N}}

\newcommand{\shp}{\mathcal{P}}

\newcommand{\cinf}{\mathcal{C}^\infty}
\newcommand{\cinft}{\mathcal{C}^{\infty, t}}
\newcommand{\Omegat}{\Omega}

\newcommand{\Tnz}{\smash{\overset{{\text{\Large .}}}{T}}}
\newcommand{\TnzNS}{\overset{{\text{\Large .}}}{T^*}}

\begin{document}
\title{dg-methods for microlocalization}


\date{}
\author{St\'ephane Guillermou}
\address{Universit\'e de Grenoble I\\
D\'epartement de Math\'ematiques\\
Institut Fourier, UMR 5582 du CNRS\\
38402 Saint-Martin d'H\`eres Cedex, France}
\email{Stephane.Guillermou@ujf-grenoble.fr}

\begin{abstract}
  For a complex manifold $X$ the ring of microdifferential operators
  $\E_X$ acts on the microlocalization $\mu hom(F,\O_X)$, for $F$ in
  the derived category of sheaves on $X$.  Kashiwara, Schapira,
  Ivorra, Waschkies proved, as a byproduct of their new
  microlocalization functor for ind-sheaves, $\mu_X$, that $\mu
  hom(F,\O_X)$ can in fact be defined as an object of $\DC(\E_X)$:
  this follows from the fact that $\mu_X \O_X$ is concentrated in one
  degree.
  
  In this paper we prove that the tempered microlocalization $\tmu
  hom(F,\O_X)$, or $\mu_X \O_X^t$, also are objects of $\DC(\E_X)$.
  Since we don't know whether $\mu_X \O_X^t$ is concentrated in one
  degree, we built resolutions, of $\E_X$ and $\mu_X \O_X^t$, such
  that the action of $\E_X$ is realized in the category of complexes
  (and not only up to homotopy).  To define these resolutions we
  introduce a version of the de Rham algebra on the subanalytic site
  which is quasi-injective.  We prove that some standard operations in
  the derived category of sheaves can be lifted to the (non-derived)
  category of dg-modules over this de Rham algebra. Then we built the
  microlocalization in this framework, together with a convolution
  product.
\end{abstract}

\maketitle

\section{Introduction}
For a complex analytic manifold the sheaf of microlocal differential
operators on its cotangent bundle was introduced by Sato, Kashiwara
and Kawai using Sato's microlocalization functor. Let us recall
briefly the definition, in the framework of~\cite{KS90}.  Let $X$ be a
manifold and let $\BDC(\C_X)$ be the bounded derived category of
sheaves of $\C$-vector spaces on $X$. For objects $F,G \in
\BDC(\C_X)$, a generalization of Sato's microlocalization functor
gives $\mu hom(F,G) \in \BDC(\C_{T^*X})$, and a convolution product is
defined in~\cite{KS90} for this functor $\mu hom$. When $X$ is a
complex analytic manifold of complex dimension $d_X$, one version of
the ring of microlocal operators is defined by $\ER_X = \mu
hom(\C_{\Delta}, \O_{X\times X}^{(0,d_X)})[d_X]$, where $\Delta$ is
the diagonal of $X\times X$ and $\O_{X\times X}^{(0,d_X)}$ denotes the
holomorphic forms of degree $0$ on the first factor and degree $d_X$
on the second factor.  It has support on the conormal bundle of
$\Delta$, which may be identified with $T^*X$.  The product of $\ER_X$
is given by the convolution product of $\mu hom$.

The convolution product also induces an action of $\ER_X$ on $\mu
hom(F,\O_X)$, for any $F \in \BDC(\C_X)$, i.e.  a morphism in
$\BDC(\C_{T^*X})$, $\ER_X \otimes\mu hom(F,\O_X) \to \mu hom(F,\O_X)$,
satisfying commutative diagrams which express the properties of an
action.

A natural question is then whether $\mu hom(F,\O_X)$ has a natural
construction as an object of $\BDC(\ER_X)$.  It was answered
positively in~\cite{KSIW06} as a byproduct of the construction of a
microlocalization functor for ind-sheaves.  The category of
ind-sheaves on $X$, $\I(\C_X)$, is introduced and studied
in~\cite{KS01}.  It comes with an internal $\Hom$ functor, $\ihom$,
and contains $\Mod(\C_X)$ as a full subcategory; the embedding of
$\Mod(\C_X)$ in $\I(\C_X)$ admits a left adjoint (which corresponds to
taking the limit) $\alpha_X \cl \I(\C_X) \to \Mod(\C_X)$ which is
exact. In this framework the construction of~\cite{KSIW06} yields a
new microlocalization functor $\mu_X \cl \BDC(\I(\C_X)) \to
\BDC(\I(\C_{T^*X}))$ such that
\begin{equation}
  \label{eq:muhom_mu}
\mu hom(F,G) \simeq \alpha_{T^*X} \Rihom(\mu_X F, \mu_X G ).
\end{equation}
In particular $\mu_X$ applies to a single object of $\BDC(\I(\C_X))$
and $\mu hom(F,G)$ takes the form of a usual $\hom$ functor between
objects on $T^*X$.

The convolution product is also defined in this context, and now it
gives an action of $\ER_X$ on $\mu_X(\O_X)$.  Through
isomorphism~\eqref{eq:muhom_mu} this action on $\mu_X(\O_X)$ induces
the action on $\mu hom(F,\O_X)$.  Hence it is enough to define
$\mu_X(\O_X)$ as an object of $\BDC(\ER_X)$ to have the answer for all
$\mu hom(F,\O_X)$.  It turns out that, outside the zero section of
$T^*X$, $\mu_X(\O_X)$ is concentrated in degree $-d_X$. Thus
$\mu_X(\O_X) \simeq H^{-d_X}\mu_X(\O_X)[d_X]$ and, since the action of
$\ER_X$ gives an $\ER_X$-module structure on $H^{-d_X}\mu_X(\O_X)$, we
see that $\mu_X(\O_X)$ naturally belongs to $\BDC(\ER_X)$, as
required.

\medskip

However in many situations differential operators of finite order are
more appropriate.  In this paper we solve the same problem in the
tempered situation, i.e. for the sheaf $\ERf_X$ of differential
operators with bounded degree and for the tempered version of $\mu
hom(F,\O_X)$.  This tempered microlocalization $\tmu hom(F,\O_X)$ is
introduced in~\cite{A94} and also has a reformulation in terms of
ind-sheaves. Namely it makes sense to consider the ind-sheaf of
tempered $\cinf$-functions and the corresponding Dolbeault complex
$\O_X^t$ (it is actually a motivation for the theory of ind-sheaves).
Then
$$
\tmu hom(F,\O_X) \simeq \alpha_{T^*X} R\ihom(\mu_X F,\mu_X \O_X^t). 
$$
We have as above a natural action of $\ERf_X$ on $\mu_X(\O_X^t)$.
Unfortunately this last complex is a priori not concentrated in one
degree and we cannot conclude directly that $\mu_X(\O_X^t)$ is an
object of $\BDC(\ERf_X)$.

\medskip

We will in fact find resolutions of $\ERf_X$ and $\mu_X(\O_X^t)$ such
that the action corresponds to a dg-module structure over a
dg-algebra.  More precisely we will define an ind-sheaf of dg-algebras
$\E^\sha_X$ on $T^*X$ (outside the zero section) with cohomology only
in degree $0$ and such that $H^0(\E^\sha_X) =\ERf_X$. We will also
find a dg-$\E^\sha_X$-module, say $M$, such that $M \simeq
\mu_X(\O_X^t)$ in $\BDC(\I(\C_{T^*X}))$ and such that the morphism of
complexes $\E^\sha_X \otimes M \to M$ given by the
dg-$\E^\sha_X$-module structure coincides with the action $\ERf_X
\otimes \mu_X(\O_X^t) \to \mu_X(\O_X^t)$.  Then, as recalled in
section~\ref{sec:dg_alg}, extension and restriction of scalars yield
an object $M' \in \BDC(\ERf_X)$ which represents $\mu_X(\O_X^t)$ with
its $\ERf_X$-action. So we conclude as in the non tempered case.

Now we explain how we construct $\E^\sha_X$ and $M$.  The main step in
the definition of $\ERf_X$, as well as its action on $\mu
hom(F,\O_X)$, is the microlocal convolution product
\begin{equation}
  \label{eq:intro_micro_prod}
\mu_{X\times X} \O_{X\times X}^{t(0,d_X)}
\overset{a}{\circ} \mu_{X\times X}
\O_{X\times X}^{t(0,d_X)}[d_X] 
\to \mu_{X\times X} \O_{X\times X}^{t(0,d_X)},  
\end{equation}
where $\overset{a}{\circ}$ denotes the composition of kernels.  This
is a morphism in the derived category. It is obtained from the
integration morphism for the Dolbeault complex and the commutation of
the functor $\mu_{X\times X}$ with the convolution of sheaves.  In
order to obtain a true dg-algebra at the end, and not a complex with a
product up to homotopy, we will represent the functor $\mu$ by a
functor between categories of complexes, which satisfies enough
functorial properties so that the convolution also corresponds to a
morphism of complexes.

\smallskip

Let us be more precise.  The first step is the construction of
injective resolutions, with some functorial properties.  For this we
introduce a quasi-injective de Rham algebra, $\sha$, below
(quasi-injectivity is a property of ind-sheaves weaker than
injectivity but sufficient to derive the usual functors).  We use the
construction of ind-sheaves from sheaves on the ``subanalytic site''
explained in~\cite{KS01}.  For a real analytic manifold $X$, the
subanalytic site, $X_{sa}$, has for open subsets the subanalytic open
subsets of $X$ and for coverings the locally finite coverings.  On
$X_{sa}$ it makes sense to consider the sheaf of tempered $\cinf$
functions, $\cinft_X$.

We consider the embedding $i_X\cl X = X\times \{0\} \to X\times \R$
and define a sheaf of $i$-forms on $X_{sa}$, $\sha_X^i = i_X^{-1}
\sect_{X\times \R_{>0}}(\mathcal{C}^{\infty, t (i)}_{X\times \R})$.
This gives a de Rham algebra $\sha_X$ and it yields a quasi-injective
resolution of $\C_{X_{sa}}$.  For a morphism of manifolds $f\cl X \to
Y$ we have an inverse image $f^*\cl f^{-1} \sha_Y \to \sha_X$.  If $f$
is smooth, with fibers of dimension $d$, we also have an integration
morphism $\int_f\cl f_{!!} \sha_X \otimes or_{X|Y}[d] \to \sha_Y$,
which represents the integration morphism $Rf_{!!}  or_{X|Y}[d] \to
\C_Y$.

We denote by $\Mod(\sha_X)$ the category of sheaves of
dg-$\sha_X$-modules. We have an obvious forgetful functor $\For'_X\cl
\Mod(\sha_X) \to \DC(\C_{X_{sa}})$. We will prove that the operations
needed in the construction of~\eqref{eq:intro_micro_prod} are defined
in $\Mod(\sha_X)$ and commute with $\For'_X$.  For example, for a
morphism of manifolds $f\cl X \to Y$ we have functors, $f^*$ and $f_*,
f_{!!}$, of inverse and direct images of dg-$\sha$-modules.  In some
cases this gives a way to represent the derived functors $f^{-1}$ and
$Rf_*, Rf_{!!}$. For example, since $\sha_X^0$ is quasi-injective we
can prove, for $F\in \Mod(\sha_X)$, $\For'_Y(f_{!!}(F)) \simeq Rf_{!!}
(\For'_X(F))$.  If $f$ is smooth we also prove, for $G\in
\Mod(\sha_Y)$, $\For'_X (f^*G) \simeq f^{-1}(\For'_Y(G))$.  When $X$
is a complex manifold, we also have a resolution, $\oot_X$, of
$\O_X^t$ by a dg-$\sha_X$-module which is locally free over $\sha^0_X$
(it is deduced from the Dolbeault resolution).

Once we have these operations we define a microlocalization functor
for dg-$\sha$-modules.  Let us recall that the functor $\mu_X$ is
given by composition with a kernel $L_X \in \BDC(\C_{(X \times T^*
  X)_{sa}})$: for $F\in \BDC(\C_{X_{sa}})$ we have $\mu_X(F) = L_X
\circ F = Rp_{2!!} ( L_X \otimes p_1^{-1}F)$.  We define a
corresponding dg-$\sha$-module, $L^\sha_X$, which is quasi-isomorphic
to $L_X$ outside the zero section of $T^*X$, i.e. over $X\times \TnzNS
X$, and we set, for a dg-$\sha_X$-module $F$:
$$
\mu^\sha_X(F) = L^\sha_X \circ F = p_{2!!} ( L^\sha_X
\otimes_{\sha} p_1^*F) .
$$
This functor is defined on the categories of complexes, i.e. it is
a functor from $\Mod(\sha_{X})$ to $\Mod(\sha_{T^* X})$.  If $F$ is
locally free over $\sha^0_X$, we show that $\mu^\sha_X(F)$ is
quasi-injective and represents $\mu_X(F)$ over $\Tnz^*X$: we have
$\For'_{T^* X} (\mu^\sha_X(F)) \simeq \mu_X(\For'_X(F))$. In
particular, when $X$ is a complex manifold we obtain the
dg-$\sha_{T^*X}$-module $\mu^\sha_X(\oot_X)$ which represents
$\mu_X(\O_X^t)$ and can be used to compute $\Rhom(\cdot, \mu_X(\O_X^t)
)$.

With these tools in hand we define the sheaf $\E^\sha_X$ mentioned
above from $\mu^\sha$, the same way $\ERf_X$ was defined from $\mu$.
The definition of the product involves a convolution product for
$\mu^\sha$. The kernel $L^\sha_X$ has indeed the same functorial
behavior as $L_X$, not with respect to all operations, but at least
those needed in the composition of kernels. We end up with a
dg-$\sha_{T^*X}$-module $\E^\sha_X$, which is a ring object in the
category of dg-$\sha_{T^*X}$-modules and which represents $\ERf_X$. In
the same way we obtain a structure of $\E^\sha_X$-module on
$\mu^\sha_X(\oot_X)$, as desired.  As said above this
$\E^\sha_X$-module gives a $\beta_{T^*X} (\ERf_X)$-module by extension
and restriction of scalars (here $\beta$ is the functor from sheaves
to ind-sheaves which is left adjoint to $\alpha$).  Our result is more
precisely stated in Theorem~\ref{thm:tmuhom=Emod}:
\begin{theorem}
  There exists $\O_X^\mu \in \DC(\beta_{T^*X} (\ERf_X)) )$, defined
  over $\Tnz^*X$, which is send to $\mu_X \O^t_X$ in
  $\DC(\I(\C_{\Tnz^*X}))$ by the forgetful functor and satisfies: for
  $F\in \DC^-(\I(\C_X))$ the complex
$$
\alpha_{T^*X} \Rihom(\pi^{-1}F, \O_X^\mu )
$$
which is naturally defined in $\DC(\ERf_X)$, over $\Tnz^*X$, is
isomorphic in $\DC(\C_{\Tnz^*X})$ to $\tmu hom(F, \O_X)$ endowed with
its action of $\ERf_X$.
\end{theorem}

\subsubsection*{Acknowledgements}
The starting point of this paper is a discussion with Raphaël Rouquier
and Pierre Schapira. The author also thanks Luca Prelli for his
comments, especially about soft sheaves on the analytic site.

\section{Notations}
If $X$ is a manifold or a site and $R$ a sheaf of rings on $X$, we
denote by $\Mod(R)$ the category of sheaves of $R$-modules on $X$.
The corresponding category of complexes is $\CC(R)$, and the derived
category $\DC(R)$; we use superscripts $b,+,-$ for the categories of
complexes which are bounded, bounded from below, bounded from above.
More generally, if $R$ is a sheaf of dg-algebras on $X$, $\Mod(R)$ is
the category of sheaves of dg-$R$-modules on $X$, $D(R)$ its derived
category (see section~\ref{sec:dg_alg}).  In particular, if $X$ is a
real analytic manifold, this applies to the subanalytic site $X_{sa}$
whose definition is recalled in section~\ref{sec:ind-sheaves_sasite}.
We denote by $\rho_X$ or $\rho$ the natural morphism of sites $X\to
X_{sa}$.  We denote by $\C_X$ and $\C_{X_{sa}}$ the constant sheaves
with coefficients $\C$ on $X$ and $X_{sa}$.

If $X$ is a manifold we denote by $\I(\C_X)$ the category of
ind-sheaves of $\C_X$-vector spaces on $X$ (see
section~\ref{sec:ind-sheaves_sasite}), and $\DC(\I(\C_X))$ its derived
category. This category comes with a natural functor $\alpha_X$, or
$\alpha\cl \I(\C_X) \to \Mod(\C_X)$ which corresponds to taking the
limit. Its left adjoint is denoted $\beta_X$, or $\beta$.

The dimension of a (real) manifold $X$ is denoted $d_X$; if $X$ is a
complex manifold, its complex dimension is $d^c_X$.

For a morphism of manifolds $f\cl X\to Y$, we let $\omega_{X|Y} = f^!
\C_Y$ be the relative dualizing complex.  Hence $\omega_{X|Y}$ is an
object of $\BDC(\C_X)$. If $Y$ is a point we simply write $\omega_X$;
then $\omega_X \simeq or_X[d_X]$, where $or_X$ is the orientation
sheaf of $X$. In fact, for $X$ connected, $\omega_{X|Y}$ is always
concentrated in one degree (since $X$ and $Y$ are manifolds), say $i$,
and we will use the notation $\omega'_{X|Y} = H^i\omega_{X|Y} [-i]$;
hence $\omega'_{X|Y}$ is a well-defined object of $\CC^b(\C_X)$.  For
an embedding of manifolds $i_Z\cl Z \hookrightarrow X$ we will often
abuse notations and write $\omega_{Z|X}$ for $i_{Z*} \omega_{Z|X}$.

For a manifold $X$, we let $TX$ and $T^*X$ be the tangent and
cotangent bundles. For a submanifold $Z\subset X$ we denote by $T_ZX$
and $T^*_ZX$ the normal and conormal bundle to $Z$.  In particular
$T^*_XX \simeq X$ is the zero section of $T^*X$ and we set $\Tnz^*X =
T^*X \setminus T^*_XX$.  We denote by $\tilde X_Z$ the normal
deformation of $Z$ in $X$ (see for example~\cite{KS90}). We recall
that it contains $T_ZX$ and comes with a map $\tau\cl \tilde X_Z \to
\R$ such that $\tau^{-1}(0) = T_ZX$ and $\tau^{-1}(r) \simeq X$ for
$r\not= 0$.  We also have another map $p\cl \tilde X_Z \to X$ such
that $p^{-1}(z) = (T_ZX)_z \cup \{z\} \times \R$ for $z\in Z$ and
$p^{-1}(x) \simeq \R\setminus \{0\}$ for $x\in X \setminus Z$.  We set
$\Omega = \tau^{-1}(\R_{>0})$.

For a morphism of manifolds $f\cl X\to Y$, the derivative of $f$ gives
the morphisms:
$$
T^*X \xfrom{f_d} X\times_Y T^*Y \xto{f_\pi} T^*Y.
$$
For two manifolds $X,Y$, $F\in \DC^+(\C_X)$, $G \in \DC^+(\C_Y)$,
we set $F \boxtimes G = p_1^{-1}F \otimes p_2^{-1}G$, where $p_i$ is
the projection from $X\times Y$ to the $i^{th}$ factor.  For three
manifolds $X,Y,Z$, and ``kernels'' $K \in \DC^+(\C_{X\times Y})$, $L
\in \DC^+(\C_{Y\times Z})$, we denote the ``composition of kernels''
by $K \circ L = Rp_{23!}(p_{12}^{-1}K \otimes p_{23}^{-1}L)$, where
$p_{ij}$ is the projection from $X\times Y \times Z$ to the
$i^{th}\times j^{th}$ factors.  We use the same notations for the
variants on subanalytic sites or using ind-sheaves.

\section{dg-algebras}
\label{sec:dg_alg}
In this section we recall some facts about (sheaves of) dg-algebras
and their derived categories. We refer the reader to~\cite{BL94}.

A dg-algebra $A$ is a $\Z$-graded algebra with a differential $d_A$ of
degree $+1$. A dg-$A$-module $M$ is a graded $A$-module with a
differential $d_M$ such that, for homogeneous elements $a\in A^i$,
$m\in M^j$, $d_M(a\cdot m) = d_A(a) \cdot m + (-1)^{i}a\cdot d_Mm$.

We consider a site $X$ and a sheaf of dg-algebras $A_X$ on $X$.  We
denote by $\Mod(A_X)$ the category of (left) dg-$A_X$-modules.  We let
$\widetilde A_X$ be the graded algebra underlying $A_X$ (i.e.
forgetting the differential). A morphism $f\cl M \to N$ in $\Mod(A_X)$
is said to be null homotopic if there exists an $\widetilde
A_X$-linear morphism $s\cl M \to N[-1]$ such that $f = sd_M + d_Ns$.
The homotopy category, $\K(A_X)$, has for objects those of
$\Mod(A_X)$, and for sets of morphisms those of $\Mod(A_X)$ quotiented
by null homotopic morphisms. A morphism in $\Mod(A_X)$ (or $\K(A_X)$)
is a quasi-isomorphism if it induces isomorphisms on the cohomology
groups.  Finally, the derived category $\DC(A_X)$ is the localization
of $\K(A_X)$ by quasi-isomorphisms.

Derived functors can be defined in this setting, in particular the
tensor product $\cdot \otimes_{A_X}^L \cdot$.  If $\phi\cl A_X \to
B_X$ is a morphism of sheaves of dg-algebras, we obtain the extension
of scalars $\phi^*\cl \DC(A_X) \to \DC(B_X)$, $M \mapsto
B_X\otimes_{A_X}^L M$, which is left adjoint to the natural
restriction of scalars $\phi_*\cl \DC(B_X) \to \DC(A_X)$.
By~\cite{BL94} (Theorem 10.12.5.1), if $\phi$ induces an isomorphism
$H(A) \isoto H(B)$, then these functors of restriction and extension
of scalars are mutually inverse equivalences of categories $\DC(A_X)
\simeq \DC(B_X)$.

\medskip

Some dg-algebras considered in this paper will appear as ring objects
in categories of complexes. We recall briefly what it means.  We let
$\shc$ be a tensor category with unit $\underline\C$ ($\shc$ will be
$\DC(\C_Y)$, $\DC(\I(\C_Y))$ or $\Mod(\sha_Y)$ for some manifold $Y$
and the unit is $\underline\C =\C_Y$).
\begin{definition}
  \label{def:anneau_module}
  A ring in $\shc$ is a triplet $(A,m,\varepsilon)$ where $A \in
  \shc$, $m\cl A \otimes A \to A$ and $\varepsilon \cl \underline\C
  \to A$ are morphisms in $\shc$ such that the following diagrams
  commute:
$$
\xymatrix{
A \otimes \underline\C \ar[r]^{A\otimes \varepsilon} \ar[dr] 
& A\otimes A  \ar[d]^m  \\
&A}
\quad
\xymatrix{
\underline\C \otimes A  \ar[r]^{\varepsilon \otimes A} \ar[dr] 
& A\otimes A  \ar[d]^m  \\
&A}
\quad
\xymatrix{
A\otimes A \otimes A \ar[r]^-{m \otimes A} \ar[d]^{A\otimes m}
 & A\otimes A  \ar[d]^m  \\
A\otimes A  \ar[r]^m  & A }
$$
In the same way, for such a ``ring'' $(A,m,\varepsilon)$, an action
of $A$ on $M\in \shc$ is a morphism, $\alpha\cl A\otimes M \to M$,
compatible with $m$ and $\varepsilon$. The pairs $(M,\alpha)$ of this
type form a category, where the morphisms from $(M,\alpha)$ to
$(M',\alpha')$ are the morphisms from $M$ to $M'$ commuting with the
action.
\end{definition}
If $E_X$ is a sheaf of (usual) algebras on $X$, we may consider $E_X$
as a ring object in $\DC(\C_X)$ and we denote by $\DC_{E_X}(\C_X)$ the
category of ``objects of $\DC(\C_X)$ with $E_X$-action'' as above.

We consider again a sheaf $A_X$ of dg-algebras on $X$. We assume that
its cohomology sheaves are $0$ except in degree $0$ and we set $E_X =
H^0(A_X)$.  Hence, if we forget the structures and view $A_X$, $E_X$
as objects of $\DC(\C_X)$, we have isomorphisms $A_X \isofrom
\tau_{\leq 0} A_X \isoto E_X$ (where $\tau_{\leq 0}$, $\tau_{\geq 0}$
denote the truncation functors).  We note that $\tau_{\leq 0} A_X =
\cdots\to A_X^{-1} \to \ker d_0 \to 0$ is sub-dg-algebra of $A_X$
(whereas $\tau_{\geq 0} A_X$ has no obvious structure of dg-algebra).
The multiplications of $A_X$ and $E_X$ induce morphisms in
$\DC(\C_X)$: $A_X \otimes A_X \to A_X$, $E_X\otimes E_X \to E_X$.
These morphisms coincide under the identification $A_X \simeq E_X$.
Hence $A_X$ and $E_X$ are isomorphic as ring objects in $\DC(\C_X)$.

For $M\in \DC(A_X)$, the structure of $A_X$-module induces a morphism
in $\DC(\C_X)$: $\alpha\cl E_X \otimes M \simeq A_X \otimes M \to M$.
Then $\alpha$ is an action of $E_X$ on $M$.  In this way we obtain a
forgetful functor $F_{A_X} \cl \DC(A_X) \to \DC_{E_X}(\C_X)$.
\begin{lemma}
  Let $A_X$ be a sheaf of dg-algebras, with cohomology sheaves
  concentrated in degree $0$ and $E_X = H^0(A_X)$.  Let $\phi\cl A_X
  \to B_X$ be a morphism of sheaves of dg-algebras such that $\phi$
  induces an isomorphism $H(A) \isoto H(B)$.  Then we have
  isomorphisms of functors $F_{A_X} \circ \phi_* \simeq F_{B_X}$ and
  $F_{B_X} \circ \phi^* \simeq F_{A_X}$.
\end{lemma}
\begin{proof}
  The first isomorphism is obvious and the second one follows because
  $\phi_*$ and $\phi^*$ are inverse equivalences of categories.
\end{proof}
Applying this lemma to the morphisms $A_X \xfrom{\phi_{\leq 0}}
\tau_{\leq 0} A_X \xto{\phi_0} E_X$, we obtain:
\begin{corollary}
\label{cor:chgt_dg-alg}
  With the hypothesis of the above lemma, we have the commutative
  diagram:
$$
\xymatrix@R=2mm@C=2cm{
\DC(A_X) \ar[dd]_{\phi_0^* \circ \phi_{\leq 0 *}}
\ar[dr]^{F_{A_X}} \\
& \DC_{E_X}(\C_X) . \\
\DC(E_X)  \ar[ur]_{F_{E_X}}  }
$$
\end{corollary}
In particular, for $M \in \DC_{E_X}(\C_X)$, if there exists $N \in
\DC(A_X)$ such that $F_{A_X}(N) \simeq M$ then there exists $N' \in
\DC(E_X)$ such that $F_{E_X}(N') \simeq M$

\section{Ind-sheaves and subanalytic site}
\label{sec:ind-sheaves_sasite}
We recall briefly some definitions and results of~\cite{KS01} about
ind-sheaves.  To define the ind-sheaves we are interested in we will
use the ``subanalytic site'' as in~\cite{KS01}, where it is introduced
to deal with tempered $\cinf$ functions. It is studied in more details
in~\cite{P07a}.

\subsection{Ind-sheaves}
For a category $\shc$ we denote by $\shc^\wedge$ the category of
functors from $\shc^{op}$ to the category of sets.  It comes with the
``Yoneda embedding'', $h\cl \shc \to \shc^\wedge$, $X \to
\Hom_\shc(\cdot,X)$.  The category $\shc^\wedge$ admits small
inductive limits but, in general, even if $\shc$ also admits such
limits, the functor $h$ may not commute with inductive limits.  We
denote by $\newindlim$ the inductive limit taken in the category
$\shc^\wedge$.

An ind-object in $\shc$ is an object of $\shc^\wedge$ which is
isomorphic to $\newindlim i$ for some functor $i \cl I \to \shc$, with
$I$ a small filtrant category.  We denote by $\Ind(\shc)$ the full
subcategory of $\shc^\wedge$ of ind-objects.

We are interested in two cases. Let $X$ be a real analytic manifold,
$\Mod(\C_X)$ the category of sheaves of $\C$-vector spaces on $X$,
$\Mod_{\R-c}(\C_X)$ the subcategory of $\R$-cons\-tructible sheaves,
$\Mod^c(\C_X)$ and $\Mod^c_{\R-c}(\C_X)$ their respective full
subcategories of objects with compact support.  We define as
in~\cite{KS01}:
$$
\I(\C_X) = \Ind(\Mod^c(\C_X))
\quad \text{ and } \quad
\I_{\R-c}(\C_X) = \Ind(\Mod^c_{\R-c}(\C_X)).
$$
There are natural exact embeddings $I_\tau\cl \I_{\R-c}(\C_X) \to
\I(\C_X)$ and $\iota_X \cl \Mod(\C_X) \to \I(\C_X)$, $F \mapsto
\newindlim F_U$, $U$ running over relatively compact open sets.  Then
$\iota_X$ sends $\Mod_{\R-c}(\C_X)$ into $\I_{\R-c}(\C_X)$.

The functor $\iota_X$ admits an exact left adjoint functor
$\alpha_X\cl \I(\C_X) \to \Mod(\C_X)$, $\newindlim_{i\in
  I}{{F}_i}\mapsto \varinjlim_{i\in I}{{F}_i}$.  Since $\iota_X$ is
fully faithful, we have $\alpha_X\circ\iota_X\simeq \id$.

The functor $\alpha_X$ admits an exact fully faithful left adjoint
$\beta_X \cl \Mod(\C_X)\to \I(\C_X)$.  Since $\beta_X$ is fully
faithful, we have $\alpha_X\circ\beta_X\simeq \id$.  For $Z \subset X$
a closed subset, we have
\begin{equation}
  \label{eq:beta_C_Z}
\beta_X (\C_Z) \simeq \newindlim_{W,\, Z \subset W} \C_{\overline W},
\qquad
\text{$W\subset X$ open subset.} 
\end{equation}
We write $\alpha$, $\beta$ for $\alpha_X$, $\beta_X$ when the context
is clear.  The machinery of Gro\-then\-dieck's six operations also
applies to this context.  There are not enough injectives in
$\I(\C_X)$, but enough ``quasi-injectives'' (see~\cite{KS01}
and~\cite{KS06}): $F \in \I(\C_X)$ is quasi-injective if the functor
$\Hom(\cdot,F)$ is exact on $\Mod^c(\C_X)$. The quasi-injective
objects are sufficient to derive the usual functors. In particular,
for a morphism of manifolds $f\cl X \to Y$ we have the functors:
\begin{align*}
f^{-1},\,f^! &\cl \BDC(\I(\C_Y)) \to \BDC(\I(\C_X)),\\
Rf_*,\, Rf_{!!}  &\cl \BDC(\I(\C_X)) \to \BDC(\I(\C_Y)),\\
\Rihom &\cl  \BDC(\I(\C_X))^{op} \times \BDC(\I(\C_X)) \to
\DC^+(\I(\C_X)), \\
\otimes &\cl  \BDC(\I(\C_X)) \times \BDC(\I(\C_X)) \to
\BDC(\I(\C_X)),
\end{align*}
and also $\Rhom = \alpha \Rihom \cl \BDC(\I(\C_X))^{op} \times
\BDC(\I(\C_X)) \to \DC^+(\C_X)$.

It will be convenient for us to use the equivalence of categories
given in~\cite{KS01} between $\I_{\R-c}(\C_X)$ and sheaves on the
subanalytic site, defined below.

\subsection{Subanalytic site}
In this paragraph $X$ is a real analytic manifold.  The open sets of
the site $X_{sa}$ are the subanalytic open subsets of $X$. A family
$\bigcup_{i\in I} U_i$ of such open sets is a covering of $U$ if and
only if, for any compact subset $K$, there exists a finite subfamily
of $J \subset I$ with $K \cap\bigcup_{i\in J} U_i = K \cap U$.  We
denote by $\Mod(\C_{X_{sa}})$ the category of sheaves of $\C$-vector
on $X_{sa}$.

We have a morphism of sites $\rho_X \cl X \to X_{sa}$ (where $X$ also
denotes the site naturally associated to the topological space $X$).
We just write $\rho$ if there is no risk of confusion.  In particular
we have adjoint functors $\rho_* \cl \Mod(\C_X) \to \Mod(\C_{X_{sa}})$
and $\rho^{-1} \cl \Mod(\C_{X_{sa}}) \to \Mod(\C_X)$.

The functor $\rho^{-1}$ is exact and $\rho_*$ is left exact and fully
faithful (hence $\rho^{-1} \circ \rho_* = \id$).  We denote by
$\rho_{c*}$ the restriction of $\rho_*$ to $\Mod_{\R-c}(\C_X)$.  Then
$\rho_{c*}$ is exact and, for $F\in \Mod_{\R-c}(\C_X)$, we usually
write $F$ instead of $\rho_{c*}F$.  The functor $\rho_{c*}$ induces an
equivalence of categories (see~\cite{KS01}, Theorem~6.3.5):
\begin{align*}
  \lambda \cl \I_{\R-c}(\C_X) &\to \Mod(\C_{X_{sa}}) \\
\newindlim_i F_i &\mapsto \varinjlim_i \rho_{c*}(F_i).
\end{align*}
Through this equivalence, the functor $\rho^{-1}$ corresponds to
$\alpha$ and it also admits an exact left adjoint functor,
corresponding to $\beta$. When dealing with the analytic site we will
use the notation $\rho_! \cl \Mod(\C_X) \to \Mod(\C_{X_{sa}})$ for
this functor.  For example~\eqref{eq:beta_C_Z} becomes $\rho_! \C_Z
\simeq \varinjlim_{Z \subset W} \C_{\overline W}$, where $W$ runs over
the subanalytic open subsets of $X$.  We note the commutative
diagrams:
$$
\xymatrix@C=5mm{
& \Mod(\C_X) \ar[dl]_{\rho_!} \ar[d]^{\beta_X} \\
\I_{\R-c}(\C_X) \simeq \Mod(\C_{X_{sa}}) \ar[r]_-{I_\tau}
& \I(\C_X) }
\!
\xymatrix@C=5mm{
\Mod_{\R-c}(\C_X) \ar[r] \ar[d]_{\rho_{c*}}
&\Mod(\C_X) \ar[d]^{\iota_X} \\
\I_{\R-c}(\C_X) \simeq \Mod(\C_{X_{sa}}) \ar[r]_-{I_\tau}
& \I(\C_X) }
$$
The functors appearing in these diagrams are exact and induce
similar commutative diagrams at the level of derived categories.

The functor $\hom$ is defined on $\Mod(\C_{X_{sa}})$ as on every site
and we set, for $Z\subset X$ a locally closed subanalytic subset:
\begin{equation}
  \label{eq:def_Gamma_Z}
\sect_Z(F) = \hom( \rho_* \C_Z, F),
\qquad
F_Z = F \otimes \rho_* \C_Z.
\end{equation}
The functors $\rho_*$ and $\hom$ commute, hence $\rho_*$ and
$\sect_Z$ also commute.  For subanalytic open subsets $U,V \subset X$
we have $\sect_U(F)(V) = F(U\cap V)$.

By analogy with ind-sheaves, a notion weaker than injective is
introduced in~\cite{P07a}: $F \in \Mod(\C_{X_{sa}})$ is
quasi-injective if $\Hom(\cdot, F)$ is exact on $\rho_*
\Mod^c_{\R-c}(\C_X)$.  In fact, since we consider coefficients in a
field, it is equivalent to ask that for any subanalytic open subsets
$U \subset V$ with compact closure $\sect(V;F) \to \sect(U;F)$ is
surjective.  Quasi-injective sheaves are sufficient to derive usual
left exact functors. In particular we obtain $\Rhom$, $\rsect_Z$, and
they commute with $R\rho_*$.  We note the following identity (which
has no equivalent on the classical site): for $F \in
\mathrm{D^b_{\R-c}}(\C_X)$, $H \in \DC^+(\C_X)$, $G \in
\DC^+(\C_{X_{sa}})$,
\begin{equation}
  \label{eq:formulaire1}
  \Rhom( R\rho_* F, G)\otimes \rho_! H \simeq 
\Rhom( R\rho_* F, G\otimes \rho_! H)
\qquad \text{in $\DC^+(\C_{X_{sa}})$}.
\end{equation}
We also have another related result (see~\cite{P07a},
Proposition~1.1.3): for $\{F_i\}_{i\in I}$ a filtrant inductive system
in $\Mod(\C_{X_{sa}})$ and $U\subset X$ an analytic open subset
\begin{equation}
  \label{eq:section_limit}
\varinjlim_{i} \rsect_U(F_i) \isoto  \rsect_U(\varinjlim_{i} F_i) .
\end{equation}
For a morphism $f\cl X \to Y$ there are the usual direct and inverse
image functors on the analytic sites $f_*$, $f^{-1}$, but also, as in
the case of ind-sheaves, a notion of proper direct image $f_{!!}$,
with a behavior slightly different from the behavior of $f_!$ on the
classical site.  The functor $f^{-1}$ and $f_*$, $f_{!!}$ admit
derived functors.  We quote in particular: for $F\in
\DC^+(\C_{X_{sa}})$, $G\in \BDC_{\R-c}(\C_Y)$ (we identify $G$ with
$\rho_*G$)
\begin{gather}
f_{!!}F = \varinjlim_{U} \, f_* (F_U),
\quad\text{$U\subset X$ relatively compact open subanalytic}, \\
f_{!!}F = \varinjlim_{K} \, f_* (\sect_K F),
\quad\text{$K\subset X$ compact subanalytic}, \\
  \label{eq:adj_im_dir_propre}
  Rf_{!!} \Rhom(f^{-1}G, F) \isoto \Rhom(G, Rf_{!!}F), \\
 \label{eq:adj_im_dir_propre_bis}
  Rf_{!!} \rsect_{f^{-1} U} F \isoto \rsect_U Rf_{!!} F.
\end{gather}
The derived functor $Rf_{!!} \cl \DC^+(\C_{X_{sa}}) \to
\DC^+(\C_{Y_{sa}})$ admits a right adjoint $f^!$. The notation is the
same as in the classical case because of the commutation relation $f^!
\circ R\rho_* \simeq R\rho_* \circ f^!$.  Hence $f^! \C_{Y_{sa}}
\simeq \rho_* \omega_{X|Y}$ and we will usually write $\omega_{X|Y}$
for $\rho_* \omega_{X|Y}$. The adjunction morphism between $f_{!!}$
and $f^!$ induces the integration morphism
\begin{equation}
  \label{eq:int_topologique}
  \operatorname{int}_f \cl Rf_{!!} (\omega_{X|Y}) \to \C_{Y_{sa}}.
\end{equation}

\subsection{``Soft'' sheaves}
\label{sec:soft_sheaves}
In this paragraph $X$ is a real analytic manifold and $X_{sa}$ is the
corresponding subanalytic site.  Though we are not in a framework of
sheaves on a locally compact space, we may introduce a notion of soft
sheaves on the subanalytic site which are acyclic for the direct image
functors.
\begin{definition}
\label{def:soft}
  A sheaf $F \in \Mod(\C_{X_{sa}})$ is soft if for any closed
  subanalytic subset $Z \subset X$ and any open subanalytic subset $U
  \subset X$ the natural morphism $\sect(U; F) \to \sect(U; F_Z)$ is
  surjective.
\end{definition}
We note the following isomorphism, as in the case of sheaves on a
reasonable topological space:
\begin{equation}
  \label{eq:section_F_K}
\sect(U;F_Z) \simeq \varinjlim_{U\cap Z \, \subset \, W \,\subset\, U} 
\sect(W; F),
\qquad
\text{$W\subset X$ subanalytic open set.} 
\end{equation}
From this description of sections it follows that quasi-injective
sheaves are soft.  We also note that if $F$ is soft and $Z \subset X$
is a closed subanalytic subset then $F_Z$ is soft.

Before we prove that soft sheaves are acyclic for functors of direct
image we need a lemma on coverings.
\begin{lemma}
  \label{lem:recouv_ouvert_ferme}
  Let $U = \bigcup_{i\in \N} U_i$ be a locally finite covering by
  subanalytic open subsets of $X$. There exist subanalytic open
  subsets of $X$, $V_i \subset U_i$, $i\in \N$, such that $U =
  \bigcup_{i\in \N} V_i$ and $(U\cap \overline{V_i}) \subset U_i$.
\end{lemma}
\begin{proof}
  We choose an analytic distance $d$ on $X$ and we define $V_n$
  inductively as follows.  If $V_i$, $i<n$, is built we set $W_n = U_n
  \setminus ( \bigcup_{i<n} V_i \cup \bigcup_{j>n} U_j)$ and
$$
V_n = \{x\in U_n; \; d(x,W_n) < d(x,\partial U_n)\}.
$$
We note that $W_n$ is subanalytic because the covering is locally
finite.  Since $d$ is analytic the functions $d(\cdot,Z)$, $Z\subset
X$ subanalytic, are continuous subanalytic functions (see~\cite{BM88}
for the notion of subanalytic function). It follows that $V_n$ is a
subanalytic open subset of $X$ and $V_n \subset U_n$.

By construction $W_n \subset V_n$ and we deduce by induction that $U =
\bigcup_{i\leq n} V_i \cup \bigcup_{j>n} U_j$.  Since the covering is
locally finite this gives $U = \bigcup_{i\in \N} V_i$.

It remains to prove that $(U\cap \overline{V_n}) \subset U_n$.  If
this is false there exists $x_0 \in U\cap \overline{V_n} \cap \partial
U_n$.  Since $W_n$ is closed in $U$, we have $\delta = d(x_0,W_n) >0$,
and the ball $B(x_0,\delta/2)$ doesn't meet $V_n$.  In particular $x_0
\not\in \overline{V_n}$ which is a contradiction.
\end{proof}
\begin{proposition}
\label{prop:soft_section_exact}
  Let $0 \to F' \xto{\alpha} F \xto{\beta} F'' \to 0$ be an exact
  sequence in $\Mod(\C_{X_{sa}})$ with $F'$ soft. Then for any open
  subanalytic subset $U \subset X$ the morphisms
$$
\sect(U;F) \to \sect(U;F'')
\quad \text{and}\quad
\varinjlim_K \sect_K(U;F) \to \varinjlim_K \sect_K(U;F''),
$$
where $K$ runs over the compact subanalytic subsets of $X$, are
surjective.
\end{proposition}
\begin{proof}
  We first consider a section $s\in \sect(U;F'')$.  We may find a
  locally finite covering $U = \bigcup_{i\in \N} U_i$ and $s_i \in
  \sect(U_i;F)$ such that $\alpha(s_i) = s|_{U_i}$.  By
  Lemma~\ref{lem:recouv_ouvert_ferme} there exists a subcovering $U =
  \bigcup_{i\in \N} V_i$ with $(U\cap \overline{V_i}) \subset U_i$.
  
  We set $Z_n = \bigcup_{i=0}^n \overline{V_i}$ and prove by induction
  on $n$ that there exists a section $\tilde s_n \in \sect(U;F_{Z_n})$
  such that $\beta(\tilde s_n) = s|_{Z_n}$ and $\tilde s_n |_{Z_{n-1}}
  = \tilde s_{n-1}$.
  
  This is clear for $n=0$ and we assume it is proved for $n$.  We set
  $t_n = (\tilde s_n - s_{n+1})|_{Z_n \cap \overline{V_{n+1}}}$.  Then
  $\beta(t_n)=0$ so that $t_n$ belongs to $\sect(U;F'_{Z_n \cap
    \overline{V_{n+1}}})$ and by hypothesis we may extend it to $t \in
  \sect(U;F')$. Now we define $\tilde s_{n+1} \in
  \sect(U;F_{Z_{n+1}})$ by $\tilde s_{n+1}|_{Z_n} = \tilde s_n$ and
  $\tilde s_{n+1}|_{\overline{V_{n+1}}} = s_{n+1} + \alpha(t)$.  The
  $\tilde s_n$ glue together into a section $\tilde s \in \sect(U;F)$
  such that $\beta(\tilde s) = s$, which proves the surjectivity of
  the first morphism.

\medskip  

Now we consider a compact $K$ and $s\in \sect_K(U;F'')$.  We choose an
open subanalytic subset $V$ such that $K \subset V$ and $K' =
\overline{V}$ is compact.  We set $Z = X \setminus V$.  We just have
seen that we may find $\tilde s \in \sect(U;F)$ such that
$\beta(\tilde s) = s$.  Hence $\beta(\tilde s|_Z) = 0$ so that $\tilde
s|_Z \in \sect(U;F'_Z)$ and we may extend $\tilde s|_Z$ to $t\in
\sect(U;F')$. Then $\hat s = \tilde s - \alpha(t)$ satisfies $\supp
\hat s \subset K'$ and $\beta(\hat s) = s$.
\end{proof}
\begin{corollary}
\label{cor:soft_conoyau}
  If $0\to F' \to F \to F'' \to 0$ is an exact sequence in
  $\Mod(\C_{X_{sa}})$ with $F'$ and $F$ soft, then $F''$ also is soft.
\end{corollary}
\begin{proof}
  For $Z\subset X$ a subanalytic closed subset we have the exact
  sequence $0\to F'_Z \to F_Z \to F''_Z \to 0$ and $F'_Z$, $F_Z$ still
  are soft.  Hence Proposition~\ref{prop:soft_section_exact} implies
  that, for any subanalytic open subset $U\subset X$, the morphisms
  $\sect(U;F) \to \sect(U;F'')$ and $\sect(U;F_Z) \to \sect(U;F''_Z)$
  are surjective.  Now it follows from the definition that $F''$ is
  soft.
\end{proof}
\begin{corollary}
\label{cor:soft_acyclique}
Let $f\cl X\to Y$ be a morphism of analytic manifolds, $U\subset X$ an
open subanalytic subset. Then soft sheaves in $\Mod(\C_{X_{sa}})$ are
acyclic for the functors $\sect(U;\cdot)$, $\varinjlim_K
\sect_K(U;\cdot)$, $K$ running over the compact subsets of $X$,
$\sect_U$, $f_*$ and $f_{!!}$.
\end{corollary}
\begin{proof}
  For the first two functors this follows from
  Proposition~\ref{prop:soft_section_exact} and
  Corollary~\ref{cor:soft_conoyau} by usual homological algebra
  arguments. This implies the result for the other functors.
\end{proof}

\subsection{Tempered functions}
Here we recall the definition of tempered $\cinf$ functions.  We also
state a tempered de Rham lemma on the subanalytic site, which is
actually a reformulation of results of~\cite{K84}.  In this paragraph,
$X$ is a real analytic manifold.
\begin{definition}
  A $\cinf$ function $f$ defined on an open set $U$ has ``polynomial
  growth at $p \in X$'' if there exist a compact neighborhood $K$ of
  $p$ and $C,N >0$ such that $\forall x\in K\cap U$, $|f(x)| < C \,
  d(x, K\setminus U)^{-N}$, for a distance $d$ defined through some
  coordinate system around $p$.
  
  We say that $f$ is tempered if all its derivatives have polynomial
  growth at any point.  In~\cite{KS01} it is proved, using results of
  \L ojasiewicz, that these functions define a subsheaf $\cinft_X$ of
  $\rho_* \cinf_X$ on $X_{sa}$.
  
  We denote by $\Omegat_X^{t,i}$ the sheaf on $X_{sa}$ of forms of degree
  $i$ with tempered coefficients. We obtain as usual a sheaf of
  dg-algebras on $X_{sa}$, the ``tempered de Rham algebra''
  $\Omegat_{X}^t = 0 \to \Omegat_{X}^{t,0} \to \cdots \to
  \Omegat_{X}^{t,n} \to 0$.
\end{definition}
\begin{lemma}
\label{lem:de_Rham_tempere}
The tempered de Rham algebra is a resolution of the constant sheaf on
the subanalytic site, i.e. we have an exact sequence on $X_{sa}$:
$$
0 \to \C_{X_{sa}} \to \Omegat_{X}^{t,0} \to \cdots
\to \Omegat_{X}^{t,n} \to 0.
$$
\end{lemma}
\begin{proof}
  In other words we have to prove that the morphism $\C_{X_{sa}} \to
  \Omegat_{X}^t$ in $\BDC(\C_{X_{sa}})$ is an isomorphism.  For this it
  is enough to see that, for any $F \in \BDC_{\R-c}(\C_X)$ we have
  \begin{equation}
    \label{eq:de_Rham_tempere}
\RHom(\rho_*F, \C_{X_{sa}}) \simeq \RHom(\rho_*F,\Omegat_{X}^t).
  \end{equation}
  Indeed for any $G \in \DC^+(\C_{X_{sa}})$, $H^k(G)$ is the sheaf
  associated to the presheaf $U \mapsto R^k\sect(U;G) = H^k
  \RHom(\rho_*\C_U, G)$; hence~\eqref{eq:de_Rham_tempere} applied to
  $F = \C_U$ gives the result.
  
  Now we prove~\eqref{eq:de_Rham_tempere}.  Actually this is
  Proposition~4.6 of~\cite{K84}, except that it is not stated in this
  language, and that it is given for tempered distributions instead of
  tempered $\cinf$ functions.  We let $\shan_X$ be the sheaf of real
  analytic functions and $\D_X$ the sheaf of linear differential
  operators with coefficients in $\shan_X$. Using a Koszul resolution
  of $\shan_X$ we have the standard isomorphism $\Rhom_{\rho_!
    \D_X}(\rho_!\shan_X, {\cinf_X}^t) \simeq \Omegat_{X}^t$.
  In~\cite{K84} a functor $RTH_X(F)$ is defined (now denoted
  $T\hom(F,\D b_X)$) and Proposition~4.6 reads:
$$
\RHom(F, \C_X) \simeq \RHom_{\D_X}(\shan_X, T\hom(F,\D b_X) ).
$$
To replace distributions by $\cinf$ functions, we have an analog of
$T\hom(F,\D b_X)$ for $\cinf$ functions, introduced in~\cite{KS96}
and~\cite{KS01}.  By~\cite{KS96}, Theorem~10.5, we have the comparison
isomorphism
$$
\Rhom_{\D_X}(\shan_X, T\hom(F,\cinf_X))
\simeq 
\Rhom_{\D_X}(\shan_X,T\hom(F,\D b_X) ).
$$
Actually, in~\cite{KS96} $X$ is a complex manifold and the result
is stated for the sheaf of anti-holomorphic functions instead of
$\shan_X$, but the proof also works in our case.  Following
~\cite{KS01}, Proposition~7.2.6 or~\cite{P07a}, Proposition~3.3.5, we
may express the functor $T\hom$ using the analytic site:
$T\hom(F,\cinf_X) \simeq \rho^{-1} \Rhom(\rho_* F, {\cinf_X}^t)$.

Putting these isomorphisms together we
obtain~\eqref{eq:de_Rham_tempere}:
\begin{align*}
  \RHom(\rho_*F,\Omegat_{X}^t)
&\simeq
\RHom(\rho_*F, \Rhom_{\rho_! \D_X}(\rho_!\shan_X, {\cinf_X}^t) )   \\
\displaybreak[1]
&\simeq
\RHom_{\rho_! \D_X}(\rho_!\shan_X, 
 \Rhom(\rho_*F,  {\cinf_X}^t) )  \\
\displaybreak[1]
&\simeq
\RHom_{\D_X}(\shan_X, T\hom(F,\D b_X) )  \\
\label{eq:de_Rham_tempD}
&\simeq
\RHom(F, \C_X),
\end{align*}
where we have used adjunction morphisms between $\otimes$, $\hom$ and
$\rho_!, \rho^{-1}$.
\end{proof}
The integration of forms also makes sense in the tempered case: we let
$f\cl X\to Y$ be a submersion with fibers of dimension $d$, $V\subset
Y$ a constructible open subset and we consider a form $\omega \in
\sect(f^{-1}(V); \Omegat_X^{t,i+d} \otimes or_{X|Y})$ such that the
closure (in $X$) of $\supp \omega$ is compact. Then $\int_f \omega \in
\sect(V; \Omegat_Y^{t,i})$.  We deduce the morphism of complexes
\begin{equation}
  \label{eq:int_cinf}
\int_f \cl  f_{!!}(\Omegat_X^t \otimes \omega'_{X|Y})
\to \Omegat_Y^t  .
\end{equation}
Its image in $\BDC(\C_{Y_{sa}})$ coincides with the morphism
$\operatorname{int}_f$ of~\eqref{eq:int_topologique}.

\section{Resolution}
\label{sec:resolution}
In this section we consider real analytic manifolds and sheaves on
their associated subanalytic sites.

\begin{definition}
\label{def:resolution}
For a real manifold $X$ we introduce the notations, $\widehat X =
X\times \R$, $i_X\cl X \to \widehat X$, $x\mapsto (x,0)$ and $X^+ =
X\times \R_{>0}$.  We consider the tempered de Rham algebra on the
site ${\widehat X}_{sa}$,
$$
\Omegat_{\widehat X}^t = 0 \to \Omegat_{\widehat X}^{t,0} \to \cdots
\to \Omegat_{\widehat X}^{t,n+1} \to 0,
$$
and we define a sheaf of anti-commutative dg-algebras on $X_{sa}$:
$\sha_X = i_X^{-1} \sect_{X^+}(\Omegat_{\widehat X}^t)$.

We denote by $\tau_{X,1}\cl \widehat X \to X$ and $\tau_{X,2}\cl
\widehat X \to \R$ the projections, and by $t$ the coordinate on $\R$.
This gives a canonical element $dt \in \sha_X^1$. The decomposition
$\widehat X = X\times \R$ induces a decomposition of the differential
$d= d_1 +d_2$ in anti-commuting differentials, where we set
$d_2(\omega) = (\partial \omega / \partial t)dt$.
\end{definition}

The algebra $\sha_X$ comes with natural morphisms related to inverse
image and direct image by a smooth map.  Let $f\cl X \to Y$ be a
morphism of manifolds. It induces $\widehat f = f\times \id$ and
$f^+$ in the following diagram, whose squares are Cartesian:
$$
\xymatrix@R=5mm{
X \ar[r]^{i_X} \ar[d]_f \ar@{}[dr]|{\Box}
& \widehat X \ar[d]^{\widehat f} \ar@{}[dr]|{\Box}
& X^+ \ar@{_{(}->}[l] \ar[d]^{f^+} \\
Y \ar[r]^{i_Y} &  \widehat Y  & Y^+ \ar@{_{(}->}[l] 
}
$$
We note that $X^+ = {\widehat f}^{-1}(Y^+)$ and this gives a
morphism of functors ${\widehat f}^{-1} \sect_{Y^+} \to \sect_{X^+}
{\widehat f}^{-1}$. Thus we obtain a morphism of dg-algebras:
$$
{\widehat f}^{-1} \sect_{Y^+} (\Omegat_{\widehat Y}^t) 
\to \sect_{X^+} {\widehat f}^{-1} (\Omegat_{\widehat Y}^t)
\to \sect_{X^+}  (\Omegat_{\widehat X}^t) .
$$
\begin{definition}
  \label{def:image_inv_sha}
  We denote by $f^\sharp \cl f^{-1} \sha_Y \to \sha_X$ the image of
  the above morphism by the restriction functor $i_X^{-1}$.  It is a
  morphism of dg-algebras.
\end{definition}
Now we assume that $f$ is smooth. Hence $\widehat f$ is also smooth
and we have the integration morphism~\eqref{eq:int_cinf}
$\int_{\widehat f} \cl \widehat f_{!!}(\Omegat_{\widehat X}^t \otimes
\omega'_{X|Y}) \to \Omegat_{\widehat Y}^t$.  We apply the functor
$i_Y^{-1} \sect_{Y^+}$ to this morphism. We note the base change
isomorphism $f_{!!} i_X^{-1} \simeq i_Y^{-1} \widehat f_{!!}$ and the
morphism $\widehat f_{!!}  \sect_{X^+} \to \sect_{Y^+} \widehat
f_{!!}$. They give the sequence of morphisms:
\begin{equation}
\label{eq:integration}
  \begin{split}
   f_{!!} (\sha_X\otimes \omega'_{X|Y})
&= f_{!!} i_X^{-1} \sect_{X^+} 
(\Omegat_{\widehat X}^{t} \otimes \omega'_{\widehat X|\widehat Y})
\simeq  i_Y^{-1} \widehat f_{!!}  \sect_{X^+} 
(\Omegat_{\widehat X}^{t} \otimes \omega'_{\widehat X|\widehat Y} ) \\
&\to i_Y^{-1} \sect_{Y^+} \widehat f_{!!} 
(\Omegat_{\widehat X}^{t} \otimes \omega'_{\widehat X|\widehat Y})
\to i_Y^{-1} \sect_{Y^+} \Omegat_{\widehat Y}^{t}
= \sha_Y
  \end{split}
\end{equation}
\begin{definition}
\label{def:integration_sha}
For a smooth map $f\cl X \to Y$, we call
morphism~\eqref{eq:integration} the integration morphism and denote it
$\int_f\cl f_{!!}  (\sha_X\otimes \omega'_{X|Y}) \to \sha_Y$.
\end{definition}
The main result of this section is the following theorem.  It is
proved in the remaining part of the section: the quasi-injectivity of
the $\sha_X^i$ is proved in Proposition~\ref{prop:qis-injectivite} and
the fact that $\sha_X$ is a resolution is
Corollary~\ref{cor:resolution}.
\begin{theorem}
\label{thm:resolution}
Let $X$ be a real analytic manifold.  The sheaf of dg-algebras
$\sha_X$ is a quasi-injective resolution of $\C_{X_{sa}}$.
\end{theorem}
\begin{remark}
  By this theorem we have $f_{!!}  (\sha_X\otimes \omega'_{X|Y})
  \simeq Rf_{!!}(\omega_{X|Y})$. Hence the morphism $\int_f$ of
  Definition~\ref{def:integration_sha} induces a morphism in the
  derived category $Rf_{!!}  \omega_{X|Y} \to \C_{Y_{sa}}$.  It
  coincides with the usual integration morphism $\operatorname{int}_f$
  of~\eqref{eq:int_topologique} because this holds for the de Rham
  complex (morphism~\eqref{eq:int_cinf} applied to $\widehat f$), and
  we have the commutative diagram:
$$
\xymatrix{
Rf_{!!} \omega_{X|Y}  \ar@{-}[r]^-{\sim}  \ar[d]_{\operatorname{int}_f}
& Rf_{!!}(i_X^{-1}  \rsect_{X^+} 
\omega_{\widehat X|\widehat Y} ) \ar[r]
& i_Y^{-1} \rsect_{Y^+} R{\widehat f}_{!!} 
\omega_{\widehat X|\widehat Y}  \ar[d]^{\operatorname{int}_{\widehat f}}   \\
\C_{Y_{sa}}  \ar@{-}[rr]^-{\sim}
&& i_Y^{-1} \rsect_{Y^+} \C_{\widehat Y_{sa}} .
}
$$
\end{remark}

For the proof of the theorem we need some lemmas on tempered
functions.  We refer to~\cite{BM88} for results on subanalytic sets.
We recall that a function is subanalytic if its graph is a subanalytic
set.  We introduce the following notation, for $U \subset X$ an open
subset, and $\varphi \cl U \to \R$ a positive continuous function on
$U$:
$$
U_\varphi = \{ (x,t) \in \widehat X; \; x\in U, \, |t| <\varphi(x)\},
\qquad
U_\varphi^+ = U_\varphi \cap X^+.
$$
\begin{lemma}
\label{lem:base_vois}
  Let $U \subset X$ be a subanalytic open subset and $V \subset
  \widehat X$ be a subanalytic open neighborhood of $U$ in $\widehat
  X$.  Then there exists a subanalytic continuous function $\varphi$
  defined on $\overline U$ such that $\varphi = 0$ on the boundary of
  $U$ and $U_\varphi \subset V$.
\end{lemma}
\begin{proof}
  We set $V' = V \cap (U\times \R)$, $Z = \widehat X \setminus V'$ and
  let $\varphi$ be the distance function to $Z$: $\varphi(x) =
  d(x,Z)$.  By~\cite{BM88}, Remark~3.11, this is a subanalytic
  function on $\widehat X$ and its restriction to $\overline U$
  satisfies the required property.
\end{proof}
The following result is similar to a division property for flat
$\cinf$ functions, which can be found for example in~\cite{T72},
Lemma~V.2.4.
\begin{lemma}
\label{lem:fonction_exhaustion}
Let $U \subset X$ be a subanalytic open subset and $\varphi \cl
\overline U \to \R$ a subanalytic continuous function on $U$, such
that $\varphi = 0$ on the boundary of $U$ and $\varphi >0$ on $U$.
Then there exists a $\cinf$ function $\psi \cl U \to \R$ such that
\begin{itemize}
\item [(i)] $\forall x \in U$, $0 < \psi(x) < \varphi(x)$,
\item [(ii)] $\psi$ and $1/\psi$ are tempered.
\end{itemize}
\end{lemma}
\begin{proof}
  We first note that it is enough to find a $\psi$ such that $\psi$ is
  tempered, $0 < \psi < \varphi$ and $1/\psi$ has polynomial growth
  along $\partial U$. We may also work locally: assuming the result is
  true on local charts, we choose
  \begin{itemize}
  \item locally finite coverings of $X$ by subanalytic open subsets,
    $(U_i)$, $(V_i)$, together with a partition of unity $\mu_i \cl X
    \to \R$ such that $\overline U_i \subset V_i$, $0\leq \mu_i$,
    $\sum \mu_i =1$, $\mu_i = 1$ on $U_i$ and $\mu_i = 0$ on a
    neighborhood of $X \setminus V_i$,
  \item $\cinf$ functions $\psi_i \cl U\cap V_i \to \R$ such that $0 <
    \psi_i < \varphi$ on $U\cap V_i$, $\psi_i$ is tempered and
    $1/\psi_i$ has polynomial growth along $\partial (U\cap V_i)$
  \end{itemize}
  and we set $\psi = \sum_i \mu_i \psi_i$. Then $\psi$ satisfies the
  conclusion of the lemma.  Indeed, each $\mu_i \psi_i$ is defined and
  tempered on $U$, and so is $\psi$ since the sum is locally finite,
  and, for $x \in \partial U$, $i$ such that $x\in U_i$, $1/ \psi \leq
  1/\psi_i$ has polynomial growth at $x$.

\medskip  

  Hence we assume $X=\R^n$ and $U$ is bounded. By~\cite{T72},
  Lemma~IV.3.3, there exist constants $C_k$, $k\in \N^n$, such that,
  for any compact $K\subset \R^n$ and any $\varepsilon >0$, there
  exists a $\cinf$ function $\alpha$ on $\R^n$ such that 
  \begin{gather*}
0\leq \alpha  \leq 1,
\qquad \alpha(x) = 0 \text{ if } d(x,K) \geq \varepsilon,
\qquad \alpha(x)  = 1 \text{ if } x\in K,  \\
\forall k \in \N^n, \quad |D^k \alpha | \leq C_k \varepsilon^{-|k|}.
  \end{gather*}
  (The function $\alpha$ is the convolution of the characteristic
  function of $\{x; d(x,K) \leq \varepsilon /2\}$ with a suitable test
  function.)
  
  We set $K_i = \{x\in U; 2^{-i-1} \leq d(x,\partial U) \leq 2^{-i}
  \}$ and we let $\alpha_i$ be the function associated to $K=K_i$ and
  $\varepsilon= 2^{-i-2}$ by the above result. In particular
  $\alpha_i=1$ on $K_i$, $\supp \alpha_i \subset S_i$, where we set
  $S_i = K_{i-1} \cup K_i \cup K_{i+1}$, and $|D^k \alpha_i | \leq C'_k
  2^{ik}$, for some $C'_k \in \R$.  This implies: $\forall x\in U$,
  $|D^k \alpha_i(x) | \leq C''_k d(x,\partial U)^{-k}$, for some other
  constants $C''_k \in \R$.
  
  \L ojasiewicz's inequality gives, for $x\in U$, $c\, d(x,\partial
  U)^r \leq \varphi(x) \leq c'd(x,\partial U)^{r'}$, for some
  constants $c,r,c',r' >0$ (see~\cite{BM88}, Theorem~6.4). We set
  $\lambda_i = \min \{ \varphi(x); x\in S_i \}$. We note that for
  $x,x' \in S_i$, we have $1/8\leq d(x,\partial U)/ d(x',\partial U)
  \leq 8$.  Hence, for $x\in S_i$, we have $C d(x,\partial U)^r \leq
  \lambda_i \leq C' d(x,\partial U)^{r'}$, for some $C,C' >0$. Since
  $\supp \alpha_i \subset S_i$, we also have $\forall i$, $\lambda_i
  \alpha_i \leq \varphi$.

  We note that an $x\in U$ belongs to at most three sets $S_i$ and we
  define $\psi = (1/3) \sum_i \lambda_i \alpha_i$.  The above
  inequalities give, for $x\in U$, $0< \psi(x) \leq \varphi(x)$ and
  $$
  |D^k \psi (x)| \leq C''_k \, C' \, d(x,\partial U)^{r'-k},
  \qquad
  \frac{1}{\psi(x)} \leq 3 \, C^{-1} d(x,\partial U)^{-r},
  $$
  so that $\psi$ and $1/\psi$ are tempered.
\end{proof}

\begin{lemma}
  \label{lem:test_function}
  Let $U \subset X$ and $\varphi \cl \overline U \to \R$ be as in
  Lemma~\ref{lem:fonction_exhaustion}.  There exist another
  subanalytic continuous function $\varphi' \cl \overline U \to \R$
  and a tempered fuction $\alpha\in \sect(X^+;\cinft_{\widehat X})$
  such that $\forall x \in U$, $0 < \varphi'(x) < \varphi(x)$ and
$$
\forall (x,t) \in X^+ \quad 0 \leq \alpha(x,t) \leq 1,
\qquad
\alpha(x,t) =
\begin{cases}
1 & \text{for $(x,t) \in U_{\varphi'}^+$}, \\
0 & \text{for $(x,t) \not\in U_{\varphi}^+$}.
\end{cases}
$$
\end{lemma}
\begin{proof}
  We choose a $\cinf$ function $\psi\cl U \to ]0,+\infty[$ satisfying
  the conclusion of Lem\-ma~\ref{lem:fonction_exhaustion} and another
  $\cinf$ function $h\cl \R \to \R$ such that $\forall t \in \R$,
  $0\leq h(t) \leq 1$, $h(t) = 1$ for $t\leq 1/2$ and $h(t) =0$ for
  $t\geq 1$.  We define our function $\alpha$ on $X^+$ by
$$
\alpha(x,t) = 
\begin{cases}
h(\frac{t}{\psi(x)}) & \text{if $x \in U$} \\
0 & \text{if $x \not\in U$}.
\end{cases}
$$
We first see that $\alpha$ is $\cinf$. This is clear except at
points $(x_0,t_0)$ with $x_0\in \partial U$. For such a point, by
continuity of $\varphi$, we may find a neighborhood $V$ of $x_0$ in
$X$ such that $\forall x\in V$, $\varphi(x) < t_0/2$.  Thus, on the
neighborhood $V \times ]t_0/2;+\infty[$ of $(x_0,t_0)$, $\alpha$ is
identically $0$, and certainly $\cinf$.

Let us check that $\alpha$ is tempered.  We only have to check growth
conditions at points $(x,0) \in \partial X^+$.  We note that $d((x,t),
\partial X^+) =t$ so that we have to bound the $D^k \alpha (x,t)$ by
powers of $t$.  Since $D^k \alpha = 0$ outside $U_\varphi^+$, we
assume $(x,t) \in U_\varphi^+$. The $D^k \alpha$ are polynomial
expressions in $t$, the derivatives of $h$ and the derivatives of
$1/\psi$.  The derivatives of $h$ to a given order are bounded, hence
it just remains to bound $D^l(1/\psi) (x)$, with $(x,t) \in
U_\varphi^+$, by a power of $t$.  Since $1/\psi$ is tempered
$D^l(1/\psi) (x)$ has a bound of the type $C\, d(x,\partial U)^{-N}$.
By \L ojasiewicz's inequality we have $\varphi(x) \leq C' d(x,\partial
U)^r$ and, since $(x,t) \in U_\varphi^+$, we have $t\leq \varphi(x)$.
Hence $D^l(1/\psi) (x) \leq C'' t^{-N/r}$, for some $C''>0$, which is
the desired bound.

By definition $\alpha = 1$ on $U_{\psi/2}^+$ and $\alpha = 0$ outside
$U_{\varphi}^+$. Hence we just have to find a subanalytic continuous
function $\varphi'$ such that $\varphi' \leq \psi/2$.  Since $1/\psi$
is tempered, there exist constants $D,M$ such that $\psi^{-1}(x) \leq
D d(x,\partial U)^{-M}$, and we may take $\varphi'(x) = \frac{1}{2 D}
d(x,\partial U)^M$.
\end{proof}

\begin{proposition}
  \label{prop:qis-injectivite}
  Let $F$ be a $\cinft_{\widehat X}$-module and set $G = i_X^{-1}
  \sect_{X^+} F$. Let $U\subset X$ be a subanalytic open subset.
  
  Then the natural map $\sect(X^+; F) \to \sect(U; G)$ is surjective.
  In particular, the sheaf $G$ is quasi-injective.
\end{proposition}
\begin{proof}
  We consider $s\in \sect(U; G)$.  As in the case of sheaves on
  manifolds we have, for $H\in \Mod(\C_{\widehat X_{sa}})$ and
  $U\subset X$, $\sect(U;i_X^{-1} H) \simeq \varinjlim_{V} \sect(V;H)$
  where $V$ runs over the subanalytic open subsets of $\widehat X$
  containing $U$. Hence, by Lemma~\ref{lem:base_vois}, we may
  represent $s$ by a section $\tilde s \in \sect(U_\varphi^+; F)$, for
  some subanalytic continuous function $\varphi$ defined on $\overline
  U$ such that $\varphi = 0$ on the boundary of $U$.
  
  We apply Lemma~\ref{lem:test_function} to the function $\varphi/2
  \cl \overline U \to \R$ and obtain $\varphi' \cl \overline U \to \R$
  and $\alpha\in \sect(X^+;\cinft_{\widehat X})$ such that $0<
  \varphi' < \varphi/2$, $\alpha =1$ on $U_{\varphi'}^+$ and $\alpha
  =0$ outside $U_{\varphi/2}^+$.  We set $\hat s = \alpha \tilde s$.
  Then $\hat s \in \sect(U_\varphi^+;F)$ extends by $0$ to a section
  $\hat s \in \sect(X^+;F)$ and we have $\hat s|_{U_{\varphi'}^+} =
  \tilde s|_{U_{\varphi'}^+}$ so that $\hat s$ also represents $s$.
  This shows the surjectivity of $\sect(X^+; F) \to \sect(U; G)$.
\end{proof}
We have the following resolution of $\cinft_X$ as an $\sha_X$-module.
Let $I_X$ be the ideal of $\sha_X$ generated by $\Omegat_X^{t,1} \subset
\sha_X^1$.  In local coordinates $(x_1,\ldots, x_n,t)$, $I_X$ consists
of the forms involving one of the $dx_i$ and we obtain the isomorphism
$\sha_X / I_X \simeq 0 \to \sha_X^0 \xto{\partial/\partial t} \sha_X^0
\to 0$, where the differential is given by $f(x,t) \mapsto
\frac{\partial f}{\partial t} (x,t)$.  The following result implies
that the complex $\sha_X / I_X$ is a resolution of $\cinft_X$.
\begin{corollary}
\label{cor:res_cinf}
  For any subanalytic open set $U \subset X$ we have the exact
  sequence:
$$
0 \to \sect(U;\cinft_X) \to \sect(U;\sha_X^0)
\xto{\partial/\partial t} \sect(U;\sha_X^0) \to 0.
$$
\end{corollary}
\begin{proof}
  The less obvious point is the surjectivity. We have the restriction
maps
$$
\xymatrix@R=5mm@C=2cm{
\sect(U\times \R; \cinft_{\widehat X} ) 
\ar[r]^{\partial/\partial t} \ar[d]
& \sect(U\times \R; \cinft_{\widehat X} )  \ar[d]  \\
 \sect(U;\sha_X^0)  \ar[r]^{\partial/\partial t}
& \sect(U;\sha_X^0) .
}
$$
The vertical arrows are surjective by
Proposition~\ref{prop:qis-injectivite}, and so is the top horizontal
arrow: we integrate with respect to $t$ with starting points on
$X\times \{1\}$, which insures that the resulting function is
tempered.
\end{proof}
\begin{corollary}
  For any subanalytic open set $U \subset X$, the sheaf $\cinft_X$ is
  acyclic with respect to the functor $\sect_U$.
\end{corollary}
\begin{proof}
  We have to prove that $R^i\sect_U(\cinft_X) =0$ for $i>0$.  By
  Proposition~\ref{prop:qis-injectivite} $\sha_X^0$ is quasi-injective
  and we may use the resolution $\cinft_X \to \sha_X^0
  \xto{\partial/\partial t} \sha_X^0$ to compute
  $R^i\sect_U(\cinft_X)$. We are thus reduced to proving the
  surjectivity of the morphism $\partial/\partial t \cl
  \sect_U(\sha_X^0) \to \sect_U(\sha_X^0)$.  This follows from
  Corollary~\ref{cor:res_cinf} since $\sect_U(\sha_X^0) (V) =
  \sect(U\cap V;\sha_X^0)$.
\end{proof}
\begin{corollary}
\label{cor:resolution}
The sheaf of dg-algebras $\sha_X$ is quasi-isomorphic to
$\C_{X_{sa}}$, i.e. we have the exact sequence:
\begin{equation*}
0 \to \C_{X_{sa}} \to \sha_X^0 \to \sha_X^1 \to \cdots \to \sha_X^{n+1}
\to 0.
\end{equation*}
\end{corollary}
\begin{proof}
  By Lemma~\ref{lem:de_Rham_tempere}, we have the exact sequence
on $\widehat X$:
\begin{equation}
  \label{eq:resol_Xchapeau}
0 \to \C_{\widehat X_{sa}} \to \Omegat_{\widehat X}^{t,0} \to \cdots \to
\Omegat_{\widehat X}^{t,n+1} \to 0.
\end{equation}
By the previous corollary the sheaves $\Omegat_{\widehat X}^{t,i}$ are
$\sect_{X^+}$-acyclic.  The constant sheaf $\C_{\widehat X_{sa}}$ also
is $\sect_{X^+}$-acyclic because $\rsect_{X^+}(\C_{\widehat X_{sa}})
\simeq \C_{\overline{X^+_{sa}}}$ (recall that $\rho_*$ commutes with
$\rsect_{X^+}$).  Hence we still have an exact sequence when we apply
$\sect_{X^+}$ to~\eqref{eq:resol_Xchapeau}, and applying the exact
functor $i_X^{-1}$ gives the corollary.
\end{proof}

\section{$\sha$-modules}
For a real analytic manifold $X$, we denote by $\Mod(\sha_X)$ the
category of sheaves of {\em bounded below} dg-$\sha_X$-modules on
$X_{sa}$. We have an obvious forgetful functor and its composition
with the localization:
\begin{equation}
  \label{eq:notation_for}
  \For_X \cl \Mod(\sha_X) \to \CC^+(\C_{X_{sa}}),
\qquad
\For'_X \cl \Mod(\sha_X) \to \DC^+(\C_{X_{sa}}).
\end{equation}
We will usually write $F$ instead of $\For_X(F)$ or $\For'_X(F)$ when
the context is clear.  We still write $\For_X$, $\For'_X$ for the
compositions of these forgetful functors with the exact functor
$I_\tau\cl \CC(\C_{X_{sa}}) \to \CC(\I(\C_X))$.

In this section we define operations on $\Mod(\sha_X)$ and check usual
formulas in this framework, as well as some compatibility with the
corresponding operations in $\CC(\C_{X_{sa}})$ or $\DC(\C_{X_{sa}})$
(hence also in $\CC(\I(\C_X))$ or $\DC(\I(\C_X))$, because $I_\tau$
commutes with the standard operations).

\subsection{Tensor product}
For $\shm, \shn \in \Mod(\sha_X)$, the tensor product $\shm
\otimes_{\sha_X} \shn \in \Mod(\sha_X)$ is defined as usual by taking
the tensor product of the underlying sheaves of graded modules over
the underlying sheaf of graded algebras and defining the differential
by $d(m\otimes n) = dm \otimes n + (-1)^{\deg m} m \otimes dn$ (for
$m$ homogeneous).  We have an exact sequence in $\CC^+(\C_{X_{sa}})$:
\begin{equation}
  \label{eq:tensor_prod}
  \shm \otimes \sha_X \otimes \shn
\xto{\delta} \shm \otimes \shn
\to \shm \otimes_{\sha_X} \shn \to 0,
\end{equation}
where $\delta(m\otimes a \otimes n) = (-1)^{\deg a \deg m} am \otimes
n - m \otimes an$, for homogeneous $a,m,n$.

For two real analytic manifolds $X,Y$ and $\shm \in \Mod(\sha_X)$,
$\shn \in \Mod(\sha_Y)$, we denote by $\underline{\boxtimes}$ the
external tensor product in the category of $\sha$-modules:
$$
\shm \underline{\boxtimes} \shn
= \sha_{X\times Y} \otimes_{(\sha_X \boxtimes \sha_Y)}
(\shm \boxtimes \shn).
$$

\subsection{Inverse image and direct image}
Let $f\cl X\to Y$ be a morphism of real analytic manifolds.  Recall
the morphism of dg-algebras $f^\sharp \cl f^{-1} \sha_Y \to \sha_X$
introduced in Definition~\ref{def:image_inv_sha}.
\begin{definition}
  For $\shn \in \Mod(\sha_Y)$ we define its inverse image in
  $\Mod(\sha_X)$:
$$
f^* \shn =  \sha_X \otimes_{f^{-1} \sha_Y} f^{-1}\shn .
$$
By adjunction $f^\sharp$ gives a morphism $\sha_Y \to f_* \sha_X$.
Hence, for $\shm \in \Mod(\sha_X)$, $f_* \shm$ has a natural structure
of dg-$\sha_Y$-module, as well as $f_{!!}\shm$, through the natural
morphism $f_* \sha_X \otimes f_{!!}\shm \to f_{!!}(\sha_X \otimes
\shm) \to f_{!!}\shm$.
\end{definition}
We have a natural morphism $f^{-1}\shn \to f^*\shn$ in
$\CC(\C_{X_{sa}})$ (with the notations of
Remark~\ref{eq:notation_for}, it could be written more exactly
$f^{-1}(\For_Y \shn) \to \For_X f^*\shn$).  We show in
Proposition~\ref{prop:im_inv_sha_subm} that it is a quasi-isomorphism
when $f$ is smooth. We first consider a particular case.
\begin{lemma}
  \label{lem:im_inv_sha_proj}
  We set $X = \R^{m+1}$, $Y = \R^m$ and we let $f\cl X \to Y$ be the
  projection. We consider coordinates $(y_1,\ldots,y_m,u)$ on $X$.
  For $\shn\in \Mod(\sha^0_Y)$ we have an exact sequence in
  $\Mod(\C_{X_{sa}})$:
$$
0 \xto{\quad} f^{-1}\shn 
\xto{\quad}   \sha^0_X \otimes_{f^{-1} \sha^0_Y } f^{-1}\shn
\xto{\quad d \quad} \sha^0_X \otimes_{f^{-1} \sha^0_Y } f^{-1}\shn
\xto{\quad} 0,
$$
where $d$ is defined by $d(a\otimes n) = \grosfrac{\partial
  a}{\partial u} \otimes n$, for $a\in \sha^0_X$, $n\in \shn$.
\end{lemma}
\begin{proof}
  We have the exact sequence $0 \to f^{-1} \sha^0_Y \to \sha^0_X
  \xto{d} \sha^0_X \to 0$, where $d(a) = \frac{\partial a}{\partial
    u}$.  The tensor product with $f^{-1}\shn$ gives the exactness of
  the sequence of the lemma except at the first term.  It just remains
  to check that $\iota \cl f^{-1}\shn \to \sha^0_X \otimes_{f^{-1}
    \sha^0_Y } f^{-1}\shn$, $n\mapsto 1\otimes n$, is injective.
  
  We consider a section $n\in \sect(U;f^{-1}\shn)$ such that $\iota(n)
  = 0$. This means that there exist a locally finite covering $U =
  \bigcup_{i\in I} U_i$ and sections, setting $V_i = f(U_i)$,
$$
n_i, n_{ij} \in \sect(V_i; \shn),
\quad
a_{ij} \in \sect(U_i; \sha_X^0),
\quad
b_{ij} \in \sect(V_i; \sha_Y^0),
$$
such that for each $i\in I$, $n|_{U_i} = f^* n_i$, $j$ runs over a
finite set $J_i$, and we have the identity in $\sect(U_i;\sha_X^0)
\otimes \sect(V_i; \shn)$:
\begin{equation}
\label{eq:im_inv_sha_proj1}
1\otimes n_i = \sum_{j\in J_i} (a_{ij} (b_{ij}\circ f) \otimes
n_{ij} - a_{ij} \otimes b_{ij} n_{ij} ).
\end{equation}
We may as well assume that the $\overline{U_i}$ are compact.  We
show in this case that $n_i=0$, which will prove $n=0$, hence the
injectivity of $\iota$.

By Proposition~\ref{prop:qis-injectivite} we may represent the
$a_{ij}$, $b_{ij}$ by tempered $\cinf$ functions defined on $X^+$,
$Y^+$. We choose continuous subanalytic functions $\varphi_i \cl
\overline{U_i} \to \R$, $\varphi_i >0$ on $U_i$, such that the
identities~\eqref{eq:im_inv_sha_proj1} hold in $\sect(U_{\varphi_i}^+;
\cinft_{\widehat X}) \otimes \sect(V_i; \shn)$.

We apply Lemma~\ref{lem:test_function} to the function $\varphi_i/2
\cl \overline{U_i} \to \R$ and obtain $\varphi'_i \cl \overline{U_i}
\to \R$ and $\alpha_i\in \sect(X^+;\cinft_{\widehat X})$ such that $0<
\varphi'_i < \varphi_i/2$, $0\leq \alpha_i \leq 1$, $\alpha_i =1$ in
$U_{\varphi'_i}^+$ and $\alpha_i =0$ outside $U_{\varphi_i/2}^+$.
Multiplying both sides of~\eqref{eq:im_inv_sha_proj1} by $\alpha_i$ we
obtain identities which now hold on $\sect(X^+; \cinft_{\widehat X})
\otimes \sect(V_i; \shn)$. These identities imply:
$$
\alpha_i \otimes n_i = 0
\qquad \text{in} \qquad
\sect(X^+; \cinft_{\widehat X})
\otimes_{\sect(Y^+; \cinft_{\widehat Y})}  \sect(V_i; \shn).
$$
We note that $\alpha_i$ has compact support and we set $\beta_i =
\int \alpha_i du$.  We have $\beta_i \in \sect(Y^+; \cinft_{\widehat
  Y})$ and the last identity gives $\beta_i n_i=0$.  Now $\sect(V_i;
\shn)$ is a $\sect(V_i;\sha_Y^0)$-module and to conclude that $n_i=0$
it just remains to prove that $\beta_i|_{V_i}$ is invertible in
$\sect(V_i;\sha_Y^0)$.

Since $\beta_i$ is a tempered $\cinf$ function on $Y^+$ it is enough
to check that $\beta_i^{-1}$ has polynomial growth along the boundary
of $W_i = f(U_{\varphi'_i}^+)$.  We set $Z_i = X^+ \setminus
U_{\varphi'_i}^+$ and for $(x,t) \in X^+$, $d_i(x,t) = d((x,t),
\partial Z_i)$. We obtain the bound, for $(y,t)\in
W_i$:
$$
\beta_i(y,t) \; \geq\;
\int_{U_{\varphi'_i}^+ \cap (\{(y,t)\} \times \R)} 1 \cdot du
\;\geq\; 2 \max_{u\in \R} d_i(y,u,t)
$$ The function $m_i(y,t) = \max_{u\in \R} d_i(y,u,t)$ is subanalytic
since the $\max$ can be taken for $u$ running on a compact set. We
have $m_i(y,t) >0$ for $(y,t) \in W_i$.  Hence, by \L ojasiewicz's
inequality we have $m_i(y,t) > C' d((y,t),\partial W_i)^{-N'}$ for
some $C', N' \in \R$ and it follows that $\beta_i^{-1}$ has polynomial
growth along $\partial W_i$.
\end{proof}

\begin{proposition}
  \label{prop:im_inv_sha_subm}
  Let $f\cl X \to Y$ be a smooth morphism and $\shn\in \Mod(\sha_Y)$.
  
  (i) The morphism in $\CC(\C_{X_{sa}})$, $f^{-1}\shn \to f^* \shn$,
  is a quasi-isomorphism.
  
  (ii) If $\shn$ is locally free as an $\sha^0_Y$-module, then $f^*
  \shn$ is locally free as an $\sha^0_X$-module.
  
  (iii) If $\shn$ is flat over $\sha^0_Y$ and we have an exact
  sequence in $\Mod(\sha_Y)$, $0 \to \shn'' \to \shn' \to \shn \to 0$,
  then the sequence $0 \to f^*\shn'' \to f^*\shn' \to f^*\shn \to 0$
  is exact.
\end{proposition}
\begin{proof}
  The statements are local on $X$, so that, up to restriction to open
  subsets, we may assume $X = Y \times \R^n$ and $f$ is the
  projection. Then we factorize $f$ as a composition of projections
  with fiber dimension $1$, so that we may even assume $X = Y\times
  \R$ (and $\widehat X = Y\times \R \times \R$).  We take coordinates
  $(y_1,\ldots,y_m,u,t)$ on $\widehat X$ ($u$ is the coordinate in the
  fiber of $f$).
  
  With this decomposition of $X$ we define the $\sha^0_X$-module
  $\sha_{vert} = \sha^0_X \oplus \sha^0_X du$. This is a
  sub-$\sha^0_X$-algebra of $\sha_X$ (not a sub-dg-algebra);
  $f^{-1}\sha_Y$ is another sub-algebra and the multiplication,
  $\sha_{vert} \otimes_{f^{-1} \sha^0_Y } f^{-1}\sha_Y \to \sha_X$, is an
  isomorphism of $\sha^0_X$-algebras.  This shows that we have an
  isomorphism of $\sha^0_X$-modules, for any dg-$\sha_Y$-module
  $\shn'$:
\begin{equation}
  \label{eq:prop:im_inv_sha_subm_1}
  \sha_{vert} \otimes_{f^{-1} \sha^0_Y } f^{-1}\shn' \isoto f^* \shn'.
\end{equation}
Since $\sha_{vert}$ is free over $\sha_X^0$, this implies (ii).  To
check that the sequence in (iii) is exact, we consider it as a
sequence of $\sha_X^0$-modules. Since $\shn$ is flat over $\sha^0_Y$,
isomorphism~\eqref{eq:prop:im_inv_sha_subm_1} gives the exactness.

\smallskip

Now we prove (i). By~\eqref{eq:prop:im_inv_sha_subm_1} again,
$f^*\shn$ is identified with the total complex of the double complex
with two rows:
\begin{equation}
  \label{eq:dble_cplx}
\vcenter{
\def\objectstyle{\scriptstyle}
\xymatrix{
\sha^0_X \otimes_{f^{-1} \sha^0_Y } f^{-1}\shn^{i-1} \ar[d] \ar[r]
& \sha^0_X \otimes_{f^{-1} \sha^0_Y } f^{-1}\shn^i
\ar[d]^{d_v^i} \ar[r]^{d_h^{1,i}}
&\sha^0_X \otimes_{f^{-1} \sha^0_Y } f^{-1}\shn^{i+1} \ar[d]
\\
\sha^0_X \otimes_{f^{-1} \sha^0_Y } f^{-1}\shn^{i-1}  \ar[r]
& \sha^0_X \otimes_{f^{-1} \sha^0_Y } f^{-1}\shn^i
\ar[r]^{d_h^{2,i}}
&\sha^0_X \otimes_{f^{-1} \sha^0_Y } f^{-1}\shn^{i+1}
}}
\end{equation}
where $d_v^i(a\otimes n) = \grosfrac{\partial a}{\partial u} \otimes
n, \quad d_h^{1,i}(a\otimes n) = \sum_k \grosfrac{\partial a}{\partial
  y_k}\otimes dy_k \cdot n + \grosfrac{\partial a}{\partial t} \otimes
dt \cdot n $ and $d_h^{2,i} = -d_h^{1,i}$.  By
Lemma~\ref{lem:im_inv_sha_proj} the $i^{th}$ column is a resolution of
$f^{-1}\shn^i$. The induced differential on the cohomology of the
columns is easily seen to be the differential of $f^{-1}\shn$ and (i)
follows.
\end{proof}

\begin{lemma}
\label{lem:sha_mod_soft}
Any sheaf of $\sha_X^0$-module is soft in the sense of
Definition~\ref{def:soft}
\end{lemma}
\begin{proof}
  Let $U$ and $Z$ be respectively open and closed subanalytic subsets
  of $X$. Let $F$ be an $\sha_X^0$-module and $s\in \sect(U; F_Z)$. We
  may assume $s\in \sect(W;F)$ for a subanalytic open set $W$ with
  $(U\cap Z) \subset W \subset U$. We choose two subanalytic open sets
  $W_1, W_2$ such that $(U\cap Z) \subset W_1 \subset \overline{W_1}
  \subset W_2 \subset \overline{W_2} \subset W$.  Since $\sha_X^0$ is
  quasi-injective we may find $\alpha \in \sect(X;\sha_X^0)$ such that
  $\alpha=1$ on $W_1$ and $\alpha =0$ on $X\setminus \overline{W_2}$.
  Then $\alpha s \in \sect(W;F)$ extends by $0$ on $U$ and $\alpha s =
  s$ in $\sect(U; F_Z)$. It follows that $\sect(U;F) \to \sect(U;
  F_Z)$ is surjective, as required.
\end{proof}
\begin{proposition}
  \label{prop:Amod_f_*acyclique}
  Let $f\cl X\to Y$ be a morphism of real analytic manifolds.  For any
  $\shm \in \Mod(\sha_X)$, $\For(\shm)\in \CC^+(\C_{X_{sa}})$ is
  acyclic with respect to $f_*$ and $f_{!!}$.  In particular we have
  isomorphisms in $\DC^+(\C_{Y_{sa}})$, $\For'( f_* (\shm)) \simeq
  Rf_* (\For'(\shm))$ and $\For'( f_{!!} (\shm)) \simeq Rf_{!!}
  (\For'(\shm))$.
\end{proposition}
\begin{proof}
  This follows from Lemma~\ref{lem:sha_mod_soft} and
  Corollary~\ref{cor:soft_acyclique}.
\end{proof}

\subsection{Projection formula}
\begin{lemma}
  \label{lem:proj_form}
  Let $f\cl X \to Y$ be a morphism of analytic manifolds, $\shm \in
  \Mod(\sha_X)$, $\shn \in\Mod(\sha_Y)$. There exists a natural
  isomorphism in $\Mod(\sha_Y)$:
$$
\shn \otimes_{\sha_Y} f_{!!}\shm
\isoto
f_{!!} (f^*\shn \otimes_{\sha_X} \shm),
$$
whose image in $\CC^+(Y_{sa})$ gives a commutative diagram:
$$
\xymatrix@R=5mm{
\shn \otimes_{\sha_Y} f_{!!}\shm  \ar[r]^-\sim
& f_{!!} (f^*\shn \otimes_{\sha_X} \shm)  \\
\shn \otimes f_{!!}\shm   \ar[r]^-\sim \ar[u]
& f_{!!} (f^{-1}\shn \otimes \shm) ,  \ar[u] 
}
$$
where the bottom arrow is the usual projection formula.
\end{lemma}
\begin{proof}
  Using~\eqref{eq:tensor_prod} and $f^*\shn \otimes_{\sha_X} \shm
  \simeq f^{-1}\shn \otimes_{f^{-1}\sha_Y} \shm$ we have the
  commutative diagram (extending the diagram of the lemma):
$$
\def\objectstyle{\scriptstyle}
\xymatrix{
\shn \otimes \sha_Y \otimes f_{!!}\shm  \ar[d]^a  \ar[r] 
& \shn \otimes f_{!!}\shm  \ar[d]^b  \ar[r] 
& \shn \otimes_{\sha_Y} f_{!!}\shm  \ar[d] \ar[r]   
& 0  \\
f_{!!}(f^{-1}\shn \otimes f^{-1}\sha_Y \otimes \shm)   \ar[r]  
& f_{!!}(f^{-1}\shn \otimes \shm )   \ar[r]  
& f_{!!}(f^{-1}\shn \otimes_{f^{-1}\sha_Y} \shm )  \ar[r]  
& 0}
$$ The top row of this diagram is exact by definition of the tensor
product, as well as the bottom row, before we take the image by
$f_{!!}$.  But any complex of the type $\shp\otimes \shm$ is an
$\sha_X^0$-module, because $\shm$ is; hence it is $f_{!!}$-acyclic by
Lemma~\ref{lem:sha_mod_soft} and Corollary~\ref{cor:soft_acyclique}.
It follows that the bottom row is exact.  Now, the vertical arrows $a$
and $b$ are isomorphisms in view of the classical projection formula.
Hence so is the morphism of the lemma.
\end{proof}

\subsection{Base change}
We consider a Cartesian square of real analytic manifolds
$$
\xymatrix{ X' \ar[r]^{f'}\ar[d]^{g'}  & Y' \ar[d]^g  \\
X \ar[r]^{f}  & Y. }
$$
We have the usual base change formula in $\Mod(\C_{Y'_{sa}})$ or
$\CC^+(\C_{Y'_{sa}})$, $f^{-1} g_{!!}  \simeq g'_{!!} f'^{-1}$ (and
its derived version in $\DC^+(\C_{Y'_{sa}})$, $f^{-1} Rg_{!!} \simeq
Rg'_{!!} f'^{-1}$).
\begin{lemma}
  \label{lem:base_change}
  Let $\shn$ a dg-$\sha_{Y'}$-module. There exists a natural morphism
\begin{equation}
  \label{eq:base_change_sha}
f^* g_{!!} \shn \to g'_{!!} f'^* \shn
\end{equation}
of dg-$\sha_{X}$-modules, whose image in the category of complexes
$\CC^+(X_{sa})$ gives a commutative diagram:
$$
\xymatrix@R=5mm{
f^* g_{!!} \shn  \ar[r]  &  g'_{!!} f'^* \shn\\
f^{-1} g_{!!} \shn  \ar[u]  \ar@{-}[r]^\sim
& g'_{!!} f'^{-1} \shn ,\ar[u] 
}
$$
where the bottom arrow is the usual base change isomorphism.

Moreover, if $f$ is an immersion and $g$ is smooth,
then~\eqref{eq:base_change_sha} is an isomorphism.
\end{lemma}
\begin{proof}
  The morphism is defined by the following composition:
  \begin{multline*}
  f^* g_{!!} \shn 
\; \simeq  \;
\sha_X \otimes_{f^{-1}\sha_Y} g'_{!!} f'^{-1} \shn 
\; \simeq \;
g'_{!!} ( g'^{-1} \sha_X \otimes_{g'^{-1}f^{-1}\sha_Y} f'^{-1} \shn) \\
\xto{\varphi}
g'_{!!} (  \sha_{X'} \otimes_{f'^{-1} \sha_{Y'}} f'^{-1} \shn) 
\; = \;
g'_{!!} f'^* \shn,
  \end{multline*}
  where the first isomorphism uses the classical base change formula
  (for complexes), and the second one the classical projection
  formula. Morphism $\varphi$ is induced by $g'^\sharp$.
  
\medskip

Now we show that $\varphi$ is an isomorphism when $f$ is an immersion
and $g$ is smooth. It is enough to show that
\begin{equation}
  \label{eq:egalite_ch_bs}
g'^{-1} \sha_X \otimes_{g'^{-1}f^{-1}\sha_Y} f'^{-1} \shn
\simeq \sha_{X'} \otimes_{f'^{-1} \sha_{Y'}} f'^{-1} \shn.
\end{equation}
This is a local statement on $X'$ so that we may as well assume that
$f$ is an embedding and $X'= X\times Z$, $Y'=Y\times Z$ for some
manifold $Z$.  We may also assume that $X$ is given by equations
$y_i=0$, $i=1\ldots,d$ in $Y$.  Then $\sha_X$ is the quotient of
$f^{-1}\sha_Y$ by the ideal generated by $y_i, dy_i$, $i=1\ldots,d$.
The same holds for $X'$ and we have the presentations:
\begin{gather*}
  f^{-1}(\sha_Y)^{2d} \xto{(y_1,\ldots,dy_d)}
 f^{-1}(\sha_Y) \to \sha_X \to 0,  \\
f'^{-1}(\sha_{Y'})^{2d} \xto{(y_1,\ldots,dy_d)}
f'^{-1}(\sha_{Y'}) \to \sha_{X'} \to 0.
\end{gather*}
Since the tensor product is right exact, the images of these exact
sequences by $g'^{-1} (\cdot) \otimes_{g'^{-1}f^{-1}\sha_Y} f'^{-1}
\shn$ and $(\cdot) \otimes_{f'^{-1} \sha_{Y'}} f'^{-1} \shn$ give the
same presentations of both sides of~\eqref{eq:egalite_ch_bs}, which
shows that they are isomorphic.
\end{proof}

\subsection{Complex manifolds}
Now we assume that $X$ is a complex analytic manifold, of dimension
$d_X^c$ over $\C$; we denote by $\overline X$ the complex conjugate
manifold and $X_\R$ the underlying real analytic manifold.  We recall
that $t$ is the coordinate on $\widehat X_\R$ given by the projection
$\tau_{X_\R,2} \cl \widehat X_\R \to \R$, and that we have the
decomposition $d = d_1 +d_2$ of the differential of $\sha_{X_\R}$
($d_2(\omega) = \partial \omega / \partial t \, dt$).  We consider the
complex of ``tempered holomorphic functions'', $\O_X^t \in
\BDC(\C_{X_{sa}})$, defined as the Dolbeault complex with tempered
coefficients:
\begin{equation}
  \label{eq:def_Ot}
\O_X^t = 
\Rhom_{\rho_! \D_{\overline X}}(\rho_! \O_{\overline X},\cinft_{X_\R})
\quad \simeq \quad
0 \to \Omegat^{t,0,0}_{X_\R} \xto{\bar \partial}  \Omegat^{t,0,1}_{X_\R} 
\xto{\bar \partial}  \cdots 
\xto{\bar \partial} \Omegat^{t,0,d_X^c}_{X_\R},
\end{equation}
where $\Omegat_{X_\R}^{t,i,j}$ denotes as usual the forms of type
$(i,j)$. The product of forms induces a morphism $\O_X^t \otimes
\O_X^t \to \O_X^t$ in $\BDC(\C_{X_{sa}})$.  In degree $0$,
$H^0(\O_X^t)$ is a subalgebra of $\rho_* \O_X$.

\begin{definition}
  \label{def:resolution_Ot}
  We let $\Omegat_{\widehat{X_\R}}^{t,i,j} = \cinft_{\widehat{X_\R}}
  \tau_{X_\R,1}^* ( \Omegat^{t,i,j}_{X_\R} )$ be the
  sub-$\cinft_{\widehat{X_\R}}$-module of
  $\Omegat^{t,i+j}_{\widehat{X_\R}}$ generated by the forms of type
  $(i,j)$ coming from $X_\R$.
  
  We define $\sha_{X_\R}^{i,j} = i_{X_\R}^{-1} \sect_{X_\R^+}
  \Omegat_{\widehat{X_\R}}^{t,i,j}$. This is a
  sub-$\sha_{X_\R}^0$-module of $\sha_{X_\R}^{i+j}$ and we have the
  decomposition $\sha_{X_\R}^k = \bigoplus_{i+j=k} \sha_{X_\R}^{i,j}
  \oplus \bigoplus_{i+j=k-1} \sha_{X_\R}^{i,j} dt$.  The operators
  $\partial, \bar \partial$ on $\Omegat_{X_\R}^t$ induce a
  decomposition of the differential of $\sha_{X_\R}$, $d = \partial +
  \bar \partial +d_2$.
  
  We let $J_X < \sha_{X_\R}$ be the differential ideal generated by
  $\sha_{X_\R}^{1,0}$ and introduce the dg-$\sha_{X_\R}$-module
  $\oot_X = \sha_{X_\R} / J_X$.  As a quotient by a differential
  ideal, $\oot_X$ inherits a structure of dg-algebra. We note the
  obvious inclusions $\rho_! \O_X \subset \rho_!  \cinft_X \subset
  \sha_{X_\R}^0$ and we define, for two complex analytic manifolds,
  $X$, $Y$:
$$
\oot^{(i)}_X = \oot_X \otimes_{\rho_! \O_X} \rho_! \O^{(i)}_X,
\qquad
\oot^{(p,q)}_{X\times Y} = \oot_{X\times Y} \otimes_{\rho_!
  \O_{X\times Y}} \rho_! \O^{(p,q)}_{X\times Y},
$$
where $\O^{(i)}_X$ denotes the holomorphic $i$-forms on $X$ and
$\O^{(p,q)}_{X\times Y} = \O_{X\times Y} \otimes_{(\O_X \boxtimes
  \O_Y)} (\O^{(p)}_X \boxtimes \O^{(q)}_Y)$.
\end{definition}
\begin{proposition}
  \label{prop:resolution_Ot}
  (i) We have an isomorphism of complexes between $\oot_X$ and
  $$
  \sha_{X_\R}^{0,0} \to \big( \sha_{X_\R}^{0,1} \oplus
  \sha_{X_\R}^{0,0} dt \big) \to \big( \sha_{X_\R}^{0,2} \oplus
  \sha_{X_\R}^{0,1} dt \big) \to \cdots \to \sha_{X_\R}^{0,d^c_X} dt,
  $$
  with differential $\bar \partial + d_2$.
  
  (ii) $\oot^{(d^c_X)}_X [-d^c_X]$ is isomorphic to the differential
  ideal of $\sha_{X_\R}$:
  $$
  \sha_{X_\R}^{d^c_X,0} \to \big( \sha_{X_\R}^{d^c_X,1} \oplus
  \sha_{X_\R}^{d^c_X,0} dt \big) \to \cdots \to
  \sha_{X_\R}^{d^c_X,d^c_X} dt.
  $$
  Moreover we have a decomposition $\sha_{X_\R} \simeq
  \oot_X^{(d^c_X)}[-d^c_X] \oplus M_X$ in free $\sha^0_{X_\R}$-modules.
  
  (iii) There exist a natural isomorphism $\O_X^t \simeq \oot_X$, in
  $\BDC(\C_{(X_\R)_{sa}})$, which commutes with the products $\O_X^t
  \otimes \O_X^t \to \O_X^t$ and $\oot_X \otimes \oot_X \to \oot_X$.
  We also have $\O_{X\times Y}^{t(p,q)} \simeq \oot^{(p,q)}_{X\times
    Y}$, in $\BDC(\C_{(X_\R\times Y_\R)_{sa}})$.  
\end{proposition}
\begin{proof}
  (i), (ii) The decomposition of $\sha_{X_\R}^k$ given in
  Definition~\ref{def:resolution_Ot} yields projections $\sha_{X_\R}^k
  \to \sha_{X_\R}^{0,k} \oplus \sha_{X_\R}^{0,k-1} dt$.  The sum of
  these projections is a surjective morphism from $\sha_{X_\R}$ to the
  complex of the proposition and we see that its kernel is $J_X$.
  Assertion (ii) follows from (i).

(iii) We use the isomorphism $\O_X^t \simeq 0 \to \Omegat^{t,0,0}_{X_\R}
\xto{\bar \partial} \cdots \xto{\bar \partial} \Omegat^{t,0,d^c_X}_{X_\R}
\to 0$. The exact sequences
$$
0 \to \Omegat^{t,0,j}_{X_\R}
\to \sha_{X_\R}^{0,j}  
\xto{\alpha \mapsto (\partial \alpha / \partial t) dt}
\sha_{X_\R}^{0,j} dt \to 0,
$$
combine into an isomorphism between $\O_X^t$ and the complex given
in (i). This proves the first isomorphism.  The second one follows
from the first and the definitions.
\end{proof}
For a morphism of complex analytic manifolds $f\cl X \to Y$, we have
an integration morphism in the derived category $Rf_!
\O_X^{(d^c_X)}[d^c_X] \to \O_Y^{(d^c_Y)}[d^c_Y]$ and its tempered
version $Rf_{!!}  \O_X^{t(d^c_X)}[d^c_X] \to \O_Y^{t(d^c_Y)}[d^c_Y]$.
By adjunction between $Rf_{!!}$ and $f^!$ we obtain
$\O_X^{t(d^c_X)}[d^c_X] \to f^!\O_Y^{t(d^c_Y)}[d^c_Y]$.

When $f$ is a submersion we have $f^! \simeq f^{-1}[2(d^c_X - d^c_Y)]$
(note that the manifolds are complex, hence oriented) and our last
morphism becomes:
\begin{equation}
  \label{eq:integration_Ot_adj}
  \O_X^{t(d^c_X)}[-d^c_X] \to f^{-1}\O_Y^{t(d^c_Y)}[-d^c_Y].
\end{equation}

\begin{proposition}
  \label{prop:integration_Ot}
  For a submersion of complex analytic manifolds $f\cl X \to Y$, the
  embeddings, for $Z = X,Y$, $\oot^{(d^c_Z)}_Z[-d^c_Z] \subset
  \sha_{Z_\R}$ of Proposition~\ref{prop:resolution_Ot} (ii) induce a
  morphism of dg-$\sha_{X_\R}$-modules
  \begin{equation}
    \label{eq:integration_Ct_adj}
\oot_X^{(d^c_X)}[-d^c_X] \to f^* \oot_Y^{(d^c_Y)}[-d^c_Y],
  \end{equation}
which represents~\eqref{eq:integration_Ot_adj} through the
isomorphism of Proposition~\ref{prop:resolution_Ot} (iii).
\end{proposition}
\begin{proof}
  By Proposition~\ref{prop:resolution_Ot} we have a decomposition
  $\sha_{Y_\R} \simeq \oot_Y^{(d^c_Y)}[-d^c_Y] \oplus M_Y$ in free
  $\sha^0_{Y_\R}$-modules; hence the quotient $\sha_{Y_\R} /
  \oot_Y^{(d^c_Y)}[-d^c_Y]$ is free over $\sha^0_{Y_\R}$ and
  Proposition~\ref{prop:im_inv_sha_subm} implies that the morphism
  $f^* \oot_Y^{(d^c_Y)}[-d^c_Y] \to f^* \sha_{Y_\R} \simeq
  \sha_{X_\R}$ is injective. Hence we just have to check the inclusion
  of ideals of $\sha_{X_\R}$: $\oot_X^{(d^c_X)}[-d^c_X] \subset f^*
  \oot_Y^{(d^c_Y)}[-d^c_Y]$.
  
  This is a local problem on $X$ and we may assume $X = Y \times Z$.
  As an $\sha_{X_\R}^0$-module, $\oot_X^{(d^c_X)}[-d^c_X]$ decomposes
  into summands $\sha_{X_\R}^{d^c_X,i}$ and $\sha_{X_\R}^{d^c_X,i}
  dt$. Now any form of type $(d^c_X,i)$ on $X = Y \times Z$ is a sum
  of products of forms of types $(d^c_Y,j)$ and $(d^c_Z,k)$, with
  $j+k=i$. In particular $\oot_X^{(d^c_X)}[-d^c_X]$ is in the image of
  $(\sha_{X_\R} \otimes f^{-1}\oot_Y^{(d^c_Y)}[-d^c_Y]) \to
  \sha_{X_\R}$.
\end{proof}
\begin{corollary}
  With the hypothesis of Proposition~\ref{prop:integration_Ot}, the
  integration morphism of Definition~\ref{def:integration_sha} induces
  a morphism of dg-$\sha_{Y_\R}$-modules
\begin{equation}
  \label{eq:integration_Ot}
f_{!!} \oot_X^{(d^c_X)}[d^c_X] \to \oot_Y^{(d^c_Y)}[d^c_Y],
\end{equation}
which represents the integration morphism $Rf_{!!}
\O_X^{t(d^c_X)}[d^c_X] \to \O_Y^{t(d^c_Y)}[d^c_Y]$.
\end{corollary}
\begin{proof}
  Morphism~\eqref{eq:integration_Ct_adj}, the projection formula and
  the integration morphism give:
  \begin{align*}
f_{!!} \oot_X^{(d^c_X)}[d^c_X]
&\to
f_{!!} f^*\oot_Y^{(d^c_Y)}[2d^c_X -d^c_Y]  \\
&\simeq
\oot_Y^{(d^c_Y)}[d^c_Y] \otimes_{ \sha_{Y_\R}}
f_{!!} \sha_{X_\R}[2(d^c_X -d^c_Y)]  \\
&\to
\oot_Y^{(d^c_Y)}[d^c_Y].
  \end{align*}
  We define \eqref{eq:integration_Ot} as the composition of these
  arrows. The integration morphism for $\O_X^t$ is also defined by
  integration of forms using the Dolbeault complex. It is nothing but
  the restriction of the integration morphism for $\sha_{X_\R}$ to a
  subcomplex, so that it coincides with~\eqref{eq:integration_Ot}.
\end{proof}
In section~\ref{sec:con_kern} we need the following composition of
kernels. Let $X,Y,Z$ be three complex analytic manifolds and $q_{ij}$
the projection from their product to the $i^{th} \times j^{th}$
factors. The product of $\O_Y$ and the integration morphism give a
convolution product: $Rq_{13!} ( q_{12}^{-1}\O_{X \times
  Y}^{(0,d^c_Y)}[d^c_Y] \otimes q_{23}^{-1}\O_{Y \times
  Z}^{(0,d^c_Z)}[d^c_Z] ) \to \O_{X \times Z}^{(0,d^c_Z)}[d^c_Z]$.  We
can also define a tempered version of this convolution, and in fact we
can even realize this tempered convolution product at the level of
complexes, using the above sheaf $\oot_{X \times Y}^{(0,d^c_Y)}$.  As
in Proposition~\ref{prop:integration_Ot}, we rather define its
``adjoint'' morphism as the following composition:
\begin{equation}
\label{eq:convol_shc}
\begin{split}
  q_{12}^*\oot_{X \times Y}^{(0,d^c_Y)}[-d^c_Y] 
\otimes_{\sha} q_{23}^*\oot_{Y \times Z}^{(0,d^c_Z)}[-d^c_Z]  
& \to 
\oot_{X \times Y \times Z}^{(0,d^c_Y,d^c_Z)}[-d^c_Y-d^c_Z]  \\
& \to 
q_{13}^*\oot_{X \times Z}^{(0,d^c_Z)}[-d^c_Z],
\end{split}
\end{equation}
where the first morphism is induced by the product $\oot_Y \otimes
\oot_Y \otimes_{\rho_! \O_Y} \rho_! \O_Y^{(d^c_Y)} \to \oot_Y
\otimes_{\rho_! \O_Y} \rho_! \O_Y^{(d^c_Y)}$ and the second morphism
is induced by morphism~\eqref{eq:integration_Ct_adj}.

\section{Microlocalization functor}
\label{sec:microl_funct}
In this section we recall the definition of the microlocalization
functor $\mu$ introduced in~\cite{KSIW06}. For a manifold $X$ this is
a functor, $\mu_X$, from $\BDC(\I(\C_{X}))$ to $\BDC(\I(\C_{T^*X}))$
given by a kernel $L_X \in \BDC(\I(\C_{X\times T^*X}))$.

We define analogs of this kernel and of the microlocalization functor
in the framework of $\sha$-modules. We check that, in the case we are
interested in, this gives a resolution of $\mu_X F$, and that it has a
functorial behavior with respect to the usual operations.

In fact, with the definition of~\cite{KSIW06}, the construction of the
external tensor product is not so straightforward. For this reason we
define another kernel for which the tensor product is easy and which
coincides with the kernel of~\cite{KSIW06} outside the zero section.

\subsection{Microlocalization functor in the derived category}
\label{sec:microl_funct_catder}
In~\cite{KSIW06} the authors define a kernel associated to the
following data: let $X$ be a manifold, $Z\subset X$ a closed
submanifold and $\sigma$ a $1$-form defined on $Z$, i.e. $\sigma$ is a
section of the bundle $Z \times_X T^*X \to Z$. To simplify the
exposition, we make the following assumption which will be satisfied
in our case:
\begin{equation}
   \label{eq:hyp_non_nul}
     \text{$\forall z\in Z$, $\sigma_z$ vanishes on $T_zZ$.}
\end{equation}
Hence $\sigma$ induces a section of $Z \times_X T^*_ZX \to Z$ and we
may define:
$$
P_\sigma = \{ (x,v) \in T_ZX;\; \langle v, \sigma(x) \rangle \geq 0
\}.
$$
Hence $P_\sigma$ is a subset of $T_ZX$, viewed itself as a subset
of the normal deformation of $Z$ in $X$, $\tilde X_Z$.  We recall that
$\tilde X_Z$ and the projection $p\cl \tilde X_Z \to X$ are given in
local coordinates as follows.  We choose coordinates $(x_1,\ldots,
x_n)$ on $X$ such that $Z$ is given by $x_i=0$, $i=1,\ldots,d$.  This
gives coordinates $(x_i,\tau)$ on $\tilde X_Z$ and $p(x_i,\tau) =
(\tau x_1,\ldots, \tau x_d,x_{d+1}, \ldots, x_n)$.  The normal bundle
$T_ZX$ is embedded in $\tilde X_Z$ as the submanifold $\{\tau = 0\}$
and we define $\Omega = \{\tau>0\}$.
$$
\xymatrix{
P_\sigma \ar@{^{(}->}[r]   &  T_ZX  \ar[d] \ar@{^{(}->}[r]
& \tilde X_Z \ar[d]^p &  \Omega \ar@{_{(}->}[l] \ar[dl] \\
& Z \ar@{^{(}->}[r] & X }
$$
We will often restrict ourself outside the zero set of $\sigma$ and
we set $T_\sigma = \{z\in Z$; $\sigma_z$ vanishes on $T_zX\}$.
\begin{definition}
  \label{def:kernel}
  Under hypothesis~\eqref{eq:hyp_non_nul}, the kernel associated to
  these data is the object of $\BDC(\I(\C_{X}))$ (recall that, for
  $i\cl Z \hookrightarrow X$, we write $\omega_{Z|X}$ instead of $i_*
  \omega_{Z|X}$):
$$
\shl_\sigma = \shl_\sigma (Z,X)
= Rp_{!!} (\beta_{\tilde X_Z}(\C_{P_\sigma})\otimes
\C_{\overline\Omega} )
\otimes \beta_{X}( \omega_{Z|X}^{\otimes -1}).
$$
\end{definition}
We recall that $\beta_{\tilde X_Z} (\C_{P_\sigma}) = \newindlim_W
\C_{\overline W}$, $W$ running over the open neighborhoods of
$P_\sigma$ in $\tilde X_Z$. Since $\newindlim$ commutes with $\otimes$
we obtain
$$
\shl_\sigma = Rp_{!!} 
( \newindlim_W \C_{\overline W\cap \overline \Omega} )
\otimes \beta_{X}( \omega_{Z|X}^{\otimes -1}),
\qquad \text{$W$ open in $\tilde X_Z$, $P_\sigma \subset W$}.
$$
We also notice that $Rp_{!!} (\beta_{\tilde
  X_Z}(\C_{P_\sigma})\otimes \C_{\overline\Omega} )$ is supported on
$Z$ (i.e. its restriction outside $Z$ is $0$).  Hence taking the
tensor product with $\beta_{X}( \omega_{Z|X}^{\otimes -1})$ reduces
locally to a shift by the codimension of $Z$.

In Proposition~1.2.11 of~\cite{KSIW06} we also have a description of
$\shl_\sigma$ outside the zero set of $\sigma$, $T_\sigma$:
\begin{equation}
  \label{eq:kern=lim_C_U}
  Rp_{!!}(\beta_{\tilde X_Z}(\C_{P_\sigma})\otimes
\C_{\overline\Omega} ) |_{X\setminus T_\sigma}
\simeq
  Rp_{!!}(\beta_{\tilde X_Z}(\C_{P_\sigma})\otimes
\C_{\Omega} )  |_{X\setminus T_\sigma}
\simeq
\newindlim_U \C_U,
\end{equation}
where $U$ runs over the open subsets of $X \setminus T_\sigma$ such
that the cone of $U$ along $Z \setminus T_\sigma$ doesn't intersect
$P_\sigma$ outside the zero section.  In particular the complexes
in~\eqref{eq:kern=lim_C_U} are concentrated in degree $0$:
\begin{equation}
  \label{eq:kern=concent_deg}
Rp_{!!}(\beta_{\tilde X_Z}(\C_{P_\sigma})\otimes
\C_{\overline\Omega} ) |_{X\setminus T_\sigma}
\simeq
p_{!!}(\beta_{\tilde X_Z}(\C_{P_\sigma})\otimes
\C_{\overline\Omega} ) |_{X\setminus T_\sigma}.
\end{equation}
When considering resolutions of $\shl_\sigma$ by $\sha$-modules, it
will be convenient to use the following different formulation, which
is equivalent outside the zero set of $\sigma$. First, using the
embedding of categories $\I_\tau \cl \Mod(\C_{X_{sa}}) \simeq
\I_{\R-c}(\C_X) \to \I(\C_X)$ we have $\shl_\sigma \simeq
\I_\tau(\shl_\sigma^{sa})$, where $\shl_\sigma^{sa} \in
\BDC(\C_{X_{sa}})$ is given by
\begin{align}
  \shl_\sigma^{sa} 
&=  Rp_{!!} (\rho_{\tilde X_Z !}(\C_{P_\sigma})\otimes
\C_{\overline\Omega} )
\otimes \rho_{X !}(  \omega_{Z|X}^{\otimes -1})  \\
\label{eq:kernel}
&\simeq  Rp_{!!} 
(\rsect_\Omega(\rho_{\tilde X_Z !}(\C_{P_\sigma})))
\otimes \rho_{X !}(  \omega_{Z|X}^{\otimes -1})  ,
\end{align}
where the second isomorphism follows from~\eqref{eq:formulaire1} and
$\C_{\overline\Omega} \simeq \Rhom(\C_\Omega,\C_{\tilde X_Z})$.
\begin{definition}
\label{def:kernel0}
  For a real analytic manifold $Y$ and $T\subset Y$ a locally closed
  analytic subset we introduce the notation $K_T = \rho_{Y!}
  \C_{\overline T} \otimes \varinjlim_W \C_{Y \setminus \overline W}$,
  where $W$ runs over the open neighborhoods of $T$ in $Y$.  We note
  that $K_T$ has support in the boundary $\overline T \setminus T$.
  
  We let $P_\sigma^0$ be the relative interior of $P_\sigma$, i.e.
  $P_\sigma^0 = \{ (x,v) \in T_ZX;\; \langle v, \sigma(x) \rangle > 0
  \}$ and we define $\shl_\sigma^0 \in \BDC(\C_{X_{sa}})$ by:
  $$
  \shl_\sigma^0 = Rp_{!!} (\rsect_\Omega K_{P_\sigma^0}) \otimes
  \rho_{X !}( \omega_{Z|X}^{\otimes -1}).
  $$
\end{definition}
\begin{lemma}
\label{lem:kernsa_kern0}
We let $(X,Z,\sigma)$ be a kernel data satisfying
hypothesis~\eqref{eq:hyp_non_nul} and we assume that $\sigma$ doesn't
vanish.

  (i) We have $\rsect_\Omega K_{P_\sigma^0} \simeq \C_{\overline
    \Omega} \otimes K_{P_\sigma^0}$.
  
  (ii) The natural morphism $K_{P_\sigma^0} \to \rho_{\tilde X_Z
    !}(\C_{P_\sigma})$ induces an isomorphism $\shl_\sigma^0 \to
  \shl_\sigma^{sa}$ in $\BDC(\C_{X_{sa}})$.
\end{lemma}
\begin{proof}
  (i) By definition $K_{P_\sigma^0} \simeq\varinjlim_{W,W^0}
  \C_{\overline W \setminus \overline{W^0}}$, where $W$ and $W^0$ run
  over the open neighborhoods of $P_\sigma$ and $P_\sigma^0$ in
  $\tilde X_Z$. By formula~\eqref{eq:section_limit} we may commute the
  limit with $\rsect_\Omega$ so that $\rsect_\Omega K_{P_\sigma^0}
  \simeq\varinjlim_{W,W^0} \rsect_\Omega\C_{\overline W \setminus
    \overline W^0}$.  Our situation is locally isomorphic to $\tilde
  X_Z \simeq \R^n$, $\Omega \simeq \R^{n-1} \times \R_{>0}$ and
  $P_\sigma \simeq \R^{n-2} \times \R_{\geq 0} \times \{0\}$.  Hence,
  choosing for example 
$$
W = \{|x_n| < \varepsilon,\; x_{n-1} >   -\varepsilon\},
\qquad
W^0 = \{|x_n| < \varphi(x_1,\ldots,x_{n-1})\},
$$
for $\varepsilon>0$ and subanalytic continuous functions $\varphi$
on $\R^{n-2} \times \R_{>0}$, we may assume that our $W, W^0$ satisfy
$\rsect_\Omega(\C_{\overline W}) \simeq \C_{\overline W \cap \overline
  \Omega}$ (and the same with $W^0$ instead of $W$) and this gives the
desired isomorphism.

\medskip

(ii) We define $F = \varinjlim_{W^0} \C_{\overline{W^0}}$, where $W^0$
runs over the open neighborhoods of $P_\sigma^0$ in $\tilde X_Z$.
Hence we have an exact sequence $0 \to K_{P_\sigma} \to \rho_{\tilde
  X_Z !}(\C_{P_\sigma}) \to F \to 0 $ and it is enough to show that
$Rp_{!!}  (\rsect_\Omega F ) = 0$.
  
As in (i) we have $\rsect_\Omega F \simeq \varinjlim_{W^0}
\C_{\overline{W^0}\cap \overline \Omega}$.  We deduce that $Rp_{!!}
(\rsect_\Omega F ) \simeq \varinjlim_{W^0,U} Rp_* \C_{\overline{W^0}
  \cap \overline \Omega \cap U}$, where $W^0$ runs on the same set as
above and $U$ runs over the open subsets of $\tilde X_Z$ with compact
closure. Since $p_*$ commutes with $\rho_*$ we are reduced to a
computation with sheaves on topological spaces.

For $x \in X \setminus Z$, $x$ near $Z$, and $U$ big enough,
$p^{-1}(x) \cap \overline{W^0} \cap \overline \Omega \cap U$ is a
union of intervals of the line, all of them compact except at most one
which is homeomorphic to $[0,1[$. When we take the limit over $W^0$
and $U$ only the last one has a non-zero contribution in the morphisms
$\C_{p^{-1}(x) \cap \overline{W^0} \cap \overline \Omega \cap U} \to
\C_{p^{-1}(x) \cap \overline{W'^0} \cap \overline \Omega \cap U'}$.
In the same way, for $x \in Z$, since $P_\sigma \subset T_ZX$ is
locally homeomorphic to a closed half plane, we may assume that
$p^{-1}(x) \cap \overline{W^0} \cap \overline \Omega \cap U$ is
homeomorphic to an half ball $\{|x| < 1, x_1\geq 0\}$.

Since $\rsect(\R;\C_{[0,1[}) = 0$ and $\rsect(\R^{n-1}; \C_{\{|x| < 1,
  x_1\geq 0\}}) = 0$, we deduce that our direct image vanishes.
\end{proof}

Now for any manifold $X$, the cotangent bundle $T^*X$ is endowed with
a canonical $1$-form, say $\omega_X$. We set $\mf{X} = X\times T^* X$
and $\mf{Z} = X \times_X T^* X \simeq T^* X$ and consider the section
$\sigma_X\cl X \times_X T^* X \to T^*X \times T^*(T^* X)$ defined by
$\sigma_X=(-\id,\omega_X)$, i.e. in local coordinates
$$
\sigma_X(x,x,\xi) = ((x;-\xi), \omega_X(x,\xi))
= ((x;-\xi),(x,\xi;\xi,0)).
$$
Hence hypothesis~\eqref{eq:hyp_non_nul} is satisfied for the data
$(\mf{X}, \mf{Z}, \sigma_X)$.
\begin{definition}
  \label{def:mu}
  With the above notations, we set $L_X =
  \shl_{\sigma_X}(\mf{Z},\mf{X})$ so that $L_X \in \BDC(\I(\C_{X
    \times T^* X}))$. We denote by $p_1 \cl X\times T^* X \to X$, $p_2
  \cl X\times T^* X \to T^*X$ the projections.  The microlocalization
  is the functor
$$
\mu_X\cl \BDC(\I(\C_X)) \to \BDC(\I(\C_{T^*X})),
\qquad
F \mapsto  L_X \circ F = Rp_{2!!} ( L_X \otimes p_1^{-1}F).
$$
\end{definition}
We note that $\sigma_X$ doesn't vanish outside the zero section of
$T^*X$ so that we can use $\shl_{\sigma_X}^0(\mf{Z},\mf{X})$ instead
of $L_X$ when we consider $\mu_XF|_{\Tnz^*X}$.

\subsection{Microlocalization functor for $\sha$-modules}
\label{sec:microl_funct_sha}
\begin{definition}
  \label{def:kernel_sha}
  For a real analytic manifold $Y$ and $T\subset Y$ a locally closed
  subset we introduce the notation $\shb_T = \sha_Y \otimes K_T$,
  where $K_T$ is given in Definition~\ref{def:kernel0}.  Let
  $(X,Z,\sigma)$ be a kernel data satisfying
  hypothesis~\eqref{eq:hyp_non_nul}. We define $\shl^\sha_\sigma \in
  \Mod(\sha_X)$ by
$$
\shl^\sha_\sigma  =  \shl^\sha_\sigma(Z,X)
= p_{!!} ( \sect_\Omega (\shb_{P_\sigma^0} ) )
\otimes \rho_{X!}( \omega_{Z|X}^{\otimes -1}) .
$$
\end{definition}
\begin{remark}
\label{rem:section_B}
For $U\subset \tilde X_Z$ a subanalytic open subset, a section of
$\sect_\Omega \shb_{P_\sigma^0}^i$ on $U$ is given by the following
data: open neighborhoods $W$ of $P_\sigma$ and $W^0$ of $P_\sigma^0$
in $\tilde X_Z$, and a section $s \in \sha_{\tilde X_Z}^i(\Omega \cap
W \cap U)$ such that $s|_{\Omega \cap W^0 \cap U} = 0$.  Actually the
definition would require that $s$ be defined on a neighborhood of
$\overline W$ and that $(\supp s) \cap \smash{\overline{W^0}} =
\emptyset$.  But, up to shrinking $W$ and $W^0$, this amounts to the
above statement.
\end{remark}
\begin{lemma}
\label{lem:kern_kern_sha}
The complex $\shl^\sha_\sigma$ consists of quasi-injective sheaves of
$\C_{X_{sa}}$-vector spaces.  We have a natural isomorphism
$\shl_\sigma^0 \simeq \shl^\sha_\sigma$ in $\DC^+(\C_{X_{sa}})$.
Hence, if $\sigma$ doesn't vanish, $\shl_\sigma \simeq
\shl^\sha_\sigma$ in $\DC^+(\I(\C_X))$.
\end{lemma}
\begin{proof}
  We recall the definition $\shl_\sigma^0 = Rp_{!!} (\rsect_\Omega
  K_{P_\sigma^0}) \otimes \rho_{X !}(  \omega_{Z|X}^{\otimes -1})$.
  
  Since $\sha_{\tilde X_Z}$ is a quasi-injective resolution of
  $\C_{(\tilde X_Z)_{sa}}$, we have $\shb_{P_\sigma^0} \simeq
  K_{P_\sigma^0}$ in $\DC^+(\C_{X_{sa}})$.  The complex
  $\shb_{P_\sigma^0}$ consists of $\sha^0$-modules, hence soft
  sheaves. It follows from Corollary~\ref{cor:soft_acyclique} that
  $\rsect_\Omega \shb_{P_\sigma^0} \simeq \sect_\Omega
  \shb_{P_\sigma^0}$.  This last complex also is formed by
  $\sha^0$-modules, hence $p_{!!}$-acyclic sheaves, and we deduce the
  isomorphism of the lemma.
  
  Let us now check that $\sect_\Omega \shb_{P_\sigma^0}$ consists of
  quasi-injective sheaves. Let $U\subset \tilde X_Z$ be a subanalytic
  open subset; a section of $\sect_\Omega \shb_{P_\sigma^0}^i$ on $U$
  is given by $W$, $W^0$, $s \in \sha_{\tilde X_Z}^i(\Omega \cap W
  \cap U)$ as in Remark~\ref{rem:section_B}.  The condition on $s$
  says that we may extend $s$ to a section $s'$ of $\sha_{\tilde
    X_Z}^i$ on $(\Omega \cap W \cap U) \cup W^0$, with $s'|_{W^0} =0$.
  By Proposition~\ref{prop:qis-injectivite} we may extend $s'$ to
  $\tilde X_Z$, and this gives the quasi-injectivity of $\sect_\Omega
  \shb_{P_\sigma^0}^i$. Since $p_{!!}$ sends quasi-injective sheaves
  to quasi-injective sheaves, we obtain the first assertion.

The last assertion follows from Lemma~\ref{lem:kernsa_kern0}.
\end{proof}
Now we can define the microlocalization functor for $\sha$-modules. We
keep the notations introduced before Definition~\ref{def:mu}: for a
manifold $X$ we have the kernel data $(\mf{X}, \mf{Z}, \sigma_X)$.
\begin{definition}
  \label{def:mu_sha}
  With the above notations, we set $L^\sha_X = \shl^\sha_{\sigma_X}
  (\mf{Z},\mf{X})$ so that $L^\sha_X \in \Mod(\sha_{X \times T^* X})$.
  The microlocalization is the functor
$$
\mu^\sha_X\cl \Mod(\sha_{X}) \to \Mod(\sha_{T^* X}),
\qquad
F \mapsto  L^\sha_X \circ F = p_{2!!} ( L^\sha_X 
\otimes_{\sha_\mf{X}} p_1^*F).
$$
\end{definition}
For $F \in \Mod(\sha_{X})$ we have a natural morphism in
$\DC^+(\I(\C_{\Tnz^*X}))$:
\begin{equation}
  \label{eq:mu_vers_mu_sha}
  \mu_X (\For'_X(F)) \to   \For'_{T^*X} (\mu^\sha_X (F)),
\end{equation}
defined by the composition of morphisms in the derived category (we
don't write the functors $\For$) on $\Tnz^*X$:
\begin{equation}
\label{eq:mu_vers_mu_sha_detail}
  \mu_X F \simeq Rp_{2!!} ( L^\sha_X \otimes p_1^{-1}F) 
\to Rp_{2!!} ( L^\sha_X  \otimes_{\sha_\mf{X}} p_1^*F)  
\isofrom
p_{2!!} ( L^\sha_X  \otimes_{\sha_\mf{X}} p_1^*F),
\end{equation}
where the first isomorphism is given by Lemma~\ref{lem:kern_kern_sha},
the second morphism is given by the morphisms $p_1^{-1}F \to p_1^*F$
and $\otimes \to \otimes_{\sha_\mf{X}}$, and the third arrow is an
isomorphism by Proposition~\ref{prop:Amod_f_*acyclique}.
\begin{lemma}
  \label{lem:kern_prod_tens}
  Let $(X,Z,\sigma)$ be a kernel data satisfying
  hypothesis~\eqref{eq:hyp_non_nul} and consider $F \in \Mod(\sha_X)$.
  We assume that $F$ is locally free as an $\sha^0_X$-module. Then:
  
  (i) $\shl^\sha_\sigma \otimes_{\sha_X} F$ is a complex of
  quasi-injective sheaves on $X_{sa}$.
  
  (ii) Let $T_\sigma \subset Z$ be the zero set of $\sigma$. The
  natural morphism, in $\CC^+(\C_{X_{sa} \setminus T_\sigma})$,
$$
\shl^\sha_\sigma \otimes F \to \shl^\sha_\sigma \otimes_{\sha_X} F
$$
is a quasi-isomorphism.
\end{lemma}
\begin{proof}
  The proof is similar to the proof of
  Proposition~\ref{prop:im_inv_sha_subm}.  Both statements are local
  on $X$.  We choose coordinates $(x_1,\ldots, x_d,z_1,\ldots,z_m)$ on
  $X$ such that $Z$ is given by $x_i=0$, $i=1,\ldots,d$. This gives
  coordinates $(x,z,\tau)$ on $\tilde X_Z$ such that $p(x,z,\tau) =
  (\tau x,z)$. On $\Omega$ we take the coordinates $(x',z,\tau)$,
  where $x' = \tau x$, so that $p(x',z,\tau) = (x',z)$.  With these
  coordinates we argue as in the proof of
  Proposition~\ref{prop:im_inv_sha_subm} to see that $p_{!!}
  \sect_\Omega (\shb_{P_\sigma^0} ) \otimes_{\sha_X} F$ is isomorphic
  to a complex
$$
G = p_{!!}  (\sect_\Omega(\shb^0_{P_\sigma^0} ) \oplus
\sect_\Omega(\shb^0_{P_\sigma^0} )d\tau) \otimes_{\sha^0_X} F,
$$
with a differential defined as in~\eqref{eq:dble_cplx}.

(i) Since $F$ is locally free over $\sha^0_X$ and $p_{!!}
\sect_\Omega(\shb^0_{P_\sigma} )$ is quasi-injective, by
Lem\-ma~\ref{lem:kern_kern_sha}, $G$ also is quasi-injective.

(ii) We will see the exactness of the sequence:
\begin{equation}
\label{eq:kern_prod_tens1}
0 \to 
p_{!!} ( \sect_\Omega(K_{P_\sigma^0}) )  \otimes \sha_X^0 
\to
p_{!!}  \sect_\Omega(\shb^0_{P_\sigma^0} )
\xto{\frac{\partial}{\partial \tau}}
p_{!!}  \sect_\Omega(\shb^0_{P_\sigma^0} )
\to 0 .
\end{equation}
Thus $G$ is quasi-isomorphic to $p_{!!} ( \sect_\Omega(K_{P_\sigma^0})
) \otimes F$ and this implies (ii) because we already know that
$\shl^\sha_\sigma$ is quasi-isomorphic to $\shl^0_\sigma$.

Now we prove~\eqref{eq:kern_prod_tens1}.  We have the exact sequence
on $\Omega$: $0 \to K_{P_\sigma^0} \otimes p^{-1} \sha_X^0 \to
\shb^0_{P_\sigma^0} \xto{\frac{\partial}{\partial \tau}}
\shb^0_{P_\sigma^0} \to 0$. Since $\sha^0$-modules are soft this
gives~\eqref{eq:kern_prod_tens1} if we prove that
$$
u \cl
p_{!!} ( \sect_\Omega(K_{P_\sigma^0}) )  \otimes  \sha_X^0 
\to
p_{!!}\sect_\Omega (  K_{P_\sigma^0} \otimes p^{-1} \sha_X^0 )
$$ is an isomorphism.  Let $s$ be a section of $p_{!!} (
\sect_\Omega(K_{P_\sigma^0}) ) \otimes \sha_X^0$ over some open set
$U$. Up to shrinking $U$ we may assume that $s$ is of the form
$1\otimes a$ where $a \in \sha^0_X(U)$ and $1 \in \C_{\overline W
  \setminus \overline{W^0}}(\Omega \cap p^{-1}(U))$, for some open
neighborhoods $W$ and $W^0$ of $P_\sigma$ and $P_\sigma^0$ in $\tilde
X_Z$.  In the same way a section $s'$ of $p_{!!}\sect_\Omega (
K_{P_\sigma^0} \otimes p^{-1} \sha_X^0 )$ over $U$ is given by
$1\otimes b$ with $b\in \sect((\overline{W'} \setminus
\overline{W'^0}) \cap \Omega \cap p^{-1}(U); p^{-1}(\sha^0_X))$ for
some other neighborhoods of $P_\sigma$ and $P_\sigma^0$ in $\tilde
X_Z$.

In the coordinates $(x',z,\tau)$ on $\Omega$ we define $i_\varepsilon
\cl X \to \Omega$, $(x,z) \mapsto (x,z,\varepsilon)$.  Then
$i_\varepsilon^{-1}((\overline{W'} \setminus \overline{W'^0}) \cap
\Omega \cap p^{-1}(U))$ is a neighborhood of $Z$ in $U$ for
$\varepsilon$ small enough.  The inverse to morphism $u$ is then given
by $a = i_\varepsilon^*(b)$.
\end{proof}

\begin{proposition}
\label{prop:mu_sha_qinj}
  We consider $F \in \Mod(\sha_X)$ and we assume that it is locally
  free as an $\sha^0_X$-module. Then:
  
  (i) $\mu^\sha_X(F)$ is a complex of quasi-injective sheaves on
  $T^*X_{sa}$.
  
  (ii) The natural morphism~\eqref{eq:mu_vers_mu_sha} in
  $\DC^+(\I(\C_{\Tnz^* X}))$, $\mu_X(F) \to \mu^\sha_X(F)$, is an
  isomorphism.
\end{proposition}
\begin{proof}
  (i) Since $p_{2!!}$ sends quasi-injective sheaves to quasi-injective
  sheaves, it is enough to prove that $L^\sha_X \otimes_{\sha_\mf{X}}
  p_1^*F$ is quasi-injective. By
  Proposition~\ref{prop:im_inv_sha_subm} $p_1^*F$ is locally free over
  $\sha^0_\mf{X}$ and we conclude by Lemma~\ref{lem:kern_prod_tens}
  (i).
  
  (ii) We have to prove that the second arrow
  in~\eqref{eq:mu_vers_mu_sha_detail} is an isomorphism over $\mf{X}
  \setminus (X\times T^*_XX)$.  By
  Proposition~\ref{prop:im_inv_sha_subm} again, $p_1^{-1}F \isoto
  p_1^*F$ in $\DC^+(\I(\C_{\mf{X}}))$, and we conclude by
  Lemma~\ref{lem:kern_prod_tens} (ii).
\end{proof}

\section{Functorial behavior of the kernel}
We will use the functorial properties of $\shl_\sigma$ given in
Propositions~1.3.1, 1.3.3 and~1.3.4 of~\cite{KSIW06}, and recalled in
Proposition~\ref{prop:prop_fonct_L} below.  In fact we state these
properties on the site $X_{sa}$, using the kernel $\shl_\sigma^0 \in
\Mod(\C_{X_{sa}})$, and our formulas are equivalent to those
of~\cite{KSIW06} when $\sigma$ doesn't vanish, by
Lemma~\ref{lem:kernsa_kern0}. We give slightly different proofs than
in~\cite{KSIW06} so that we can translate them easily in the framework
of $\sha$-modules in Proposition~\ref{prop:prop_fonct_L_sha}.  In this
section $(X_1, Z_1, \sigma_1)$ and $(X_2, Z_2, \sigma_2)$ are two sets
of data as above, satisfying hypothesis~\eqref{eq:hyp_non_nul}.  We
set for short $\tilde X_i = \tilde {(X_i)}_{Z_i}$.

\subsection{Direct and inverse images}
We assume to be given a morphism $f\cl X_1 \to X_2$ is a morphism such
that $f(Z_1) \subset Z_2$ and $\sigma_1 = f^*\sigma_2$.  The morphism
$f$ induces $\tilde f\cl \tilde X_1 \to \tilde X_2$, decomposed as
$\tilde f = h \circ g$ in the following diagram, where the square is
Cartesian:
\begin{equation}
  \label{eq:diag_prop_fonct_L}
\vcenter{\xymatrix{
 \Omega_1 \ar@{^{(}->}[r]
& \tilde X_1  \ar[r]^-{g} \ar[dr]_{p_1}
& X_1 \times_{X_2} \tilde X_2  \ar[d]^q \ar[r]^-h \ar@{}[dr]|\Box
& \tilde X_2 \ar[d]_{p_2}
& \Omega_2 \ar@{_{(}->}[l] \\
P_{\sigma_1} \ar@{^{(}->}[ur] && X_1 \ar[r]^f & X_2
& P_{\sigma_2} \ar@{_{(}->}[ul]
}}
\end{equation}
We have $\Omega_1 = \tilde{f}^{-1} \Omega_2$, $T_{Z_1}X_1 =
\tilde{f}^{-1} T_{Z_2}X_2$, $P_{\sigma_1} = \tilde{f}^{-1}
P_{\sigma_2}$.  We note that $X_1 \times_{X_2} \tilde X_2$ is in
general not a manifold and may have components of different
dimensions.  When $f$ is clean with respect to $Z_2$ and $Z_1 =
f^{-1}(Z_2)$ (clean then means that $g'\cl T_{Z_1}X_1 \to X_1
\times_{X_2} T_{Z_2}X_2$ is injective), $g$ is a closed embedding.
When $f$ is transversal to $Z_2$ and $Z_1 = f^{-1}(Z_2)$, $g$ is an
isomorphism.
\begin{lemma}
  \label{lem:fonctorialite_K_T}
  Let $f\cl X \to Y$ be a morphism of real analytic manifolds,
  $T\subset Y$ a locally closed subset and $Z=f^{-1}T$.

(i) There exists a natural isomorphism $f^{-1}K_T \simeq K_Z$.

(ii) Let $V\subset Y$ be an open subset and $U = f^{-1}(V)$ and let
$G\in \CC^+(\C_{Y_{sa}})$.  We assume that the restriction $f|_U\cl
U\to V$ is smooth.  Then the integration of forms induces a morphism
of complexes:
\begin{equation}
  \label{eq:int_morph_compl}
f_{!!} \sect_U (\sha_X \otimes f^{-1}G \otimes \omega'_{X|Y})
\to \sect_V (\sha_Y \otimes G),
\end{equation}
whose image in $\BDC(\C_{X_{sa}})$ is the natural morphism $Rf_{!!}
\rsect_U (f^{-1}G \otimes\omega_{X|Y}) \to \rsect_V G$.
\end{lemma}
\begin{proof}
  (i) By definition $K_T = \varinjlim_{W_1,W_2} \C_{\overline{W_2}
    \setminus \overline{W_1}}$, where $W_1$, $W_2$ run over the open
  neighborhoods of $T$, $\overline T$ in $Y$.  For any compact
  $M\subset X$, the $f^{-1}(W_i) \cap M$ give fundamental systems of
  neighborhoods of $Z\cap M$ and $\overline Z \cap M$ in $M$.  Since
  the inductive limit commutes with $f^{-1}$ we deduce the
  isomorphism.
  
  (ii) We first reduce the statement to $G=\C_Y$.  Indeed, for $F \in
  \CC^+(\C_{X_{sa}})$ and $F' \in \CC^+(\C_{Y_{sa}})$ with a morphism
  $f_{!!}\sect_U(F) \to \sect_V F'$, we have the sequence of morphisms
  \begin{align*}
    f_{!!}\sect_U(F \otimes   f^{-1}G) 
&\isoto f_{!!}\sect_U(\sect_U(F) \otimes f^{-1}G) \\
&\isoto \sect_V (f_{!!} (\sect_U(F)\otimes f^{-1}G) )  \\
&\to \sect_V (\sect_V(F') \otimes G) \\
&\isofrom \sect_V (F' \otimes G),
  \end{align*}
  where the first one and the last one are induced by $F \to
  \sect_U(F)$ and $F' \to \sect_V(F')$ (they are isomorphisms because
  $F|_U \simeq \sect_U(F)|_U$ and $F'|_V \simeq \sect_V(F')|_V$), the
  second one is morphism~\eqref{eq:adj_im_dir_propre_bis} and the
  third one is given by the projection formula and the given morphism
  $f_{!!}\sect_U(F) \to \sect_V F'$.
  
  Hence it is enough to define $f_{!!} \sect_U (\sha_X \otimes
  \omega'_{X|Y}) \to \sect_V \sha_V$.  By definition a section of
  $f_{!!}( \sect_U \sha_X \otimes \omega'_{X|Y})$ over $W \subset Y$
  is represented by a section $\omega \in\sect( U\cap f^{-1} W; \sha_X
  \otimes \omega'_{X|Y})$ whose support has compact closure in $X$.
  Since $f$ is smooth on $U$ we may define $\int_f \omega$, and it is
  tempered on $V$, i.e. it gives an element of $\sect(V\cap W;
  \sha_Y)$.  This gives morphism~\eqref{eq:int_morph_compl}.
\end{proof}

\begin{proposition}
  \label{prop:prop_fonct_L}
  (i) There exists a natural morphism in $\BDC(\C_{(X_2)_{sa}})$:
\begin{equation}
  \label{eq:ima_dir_L}
  Rf_{!!} ( \shl_{\sigma_1}^0  \otimes 
\rho_{X_1 !}(  \omega_{Z_1 |Z_2} )) \to \shl_{\sigma_2}^0.
\end{equation}

(ii) We assume moreover that $Z_1 = f^{-1}(Z_2)$ and $f$ is clean with
respect to $Z_2$. Then there exists a natural morphism in
$\BDC(\C_{(X_1)_{sa}})$:
  \begin{equation}
    \label{eq:restr_L}
    f^{-1} \shl_{\sigma_2}^0 \to \shl_{\sigma_1}^0
\otimes \rho_{X_1 !}(  \omega_{Z_1 |Z_2} )
\otimes \omega_{X_1 |X_2}^{-1}.
  \end{equation}
  If $f$ is transversal to $Z_2$ it reduces to:
  \begin{equation}
    \label{eq:restr_L_iso}
    f^{-1} \shl_{\sigma_2}^0 \to \shl_{\sigma_1}^0.
  \end{equation}
\end{proposition}
\begin{proof} 
  (i) We note that $p_1^{-1} \omega_{X_1|X_2} \simeq \omega_{\tilde
    X_1| \tilde X_2}$. We have the morphisms:
  \begin{align*}
    Rf_{!!} ( Rp_{1!!} 
\rsect_{\Omega_1}(K_{P_{\sigma_1}^0} ) \otimes \omega_{X_1|X_2})
&\simeq
Rp_{2!!} R \tilde f_{!!} 
\rsect_{\Omega_1}(K_{P_{\sigma_1}^0} 
\otimes \omega_{\tilde X_1| \tilde X_2} )  \\
&\isoto
Rp_{2!!} \rsect_{\Omega_2} R \tilde f_{!!} 
( \tilde f^{-1} K_{P_{\sigma_2}^0} 
\otimes \omega_{\tilde X_1| \tilde X_2} )  \\
&\to 
Rp_{2!!} \rsect_{\Omega_2} K_{P_{\sigma_2}^0} ,
\end{align*}
where in the first line we use the projection formula for $p_1$ and $f
\, p_1 = p_2 \, \tilde f$ (we note that $\omega_{\tilde X_1| \tilde
  X_2}$ enters the parenthesis because it is locally constant). In the
second line we use formula~\eqref{eq:adj_im_dir_propre_bis} and
Lemma~\ref{lem:fonctorialite_K_T}, (i).  In the third line we use the
projection formula for $\tilde f$ and the integration morphism.

Now we take the tensor product with $\omega_{Z_2 |X_2}^{\otimes-1}$
and we obtain~\eqref{eq:ima_dir_L}.

(ii) Since $f$ is clean with respect to $Z_2$ and $Z_1 = f^{-1} Z_2$,
the morphism $g$ in diagram~\eqref{eq:diag_prop_fonct_L} is an
embedding.  Hence $g_* = g_{!!}$ and we have the adjunction morphism
$\id \to g_{!!} \, g^{-1}$.
We deduce a morphism of functors
\begin{equation}
  \label{eq:presque_chgt_base}
  f^{-1} Rp_{2!!} \to Rp_{1!!} \tilde f^{-1}
\end{equation}
as the composition of the base change $f^{-1} Rp_{2!!} \to Rq_{!!}
h^{-1}$ and the adjunction morphism $Rq_{!!}  h^{-1} \to Rq_{!!}
g_{!!} \, g^{-1} h^{-1} = Rp_{1!!} \tilde f^{-1}$.  Now we define
\eqref{eq:restr_L} by the sequence of morphisms:
\begin{align*}
   f^{-1} \shl_{\sigma_2}^0
&= 
f^{-1} ( Rp_{2!!} 
( \rsect_{\Omega_2} K_{P_{\sigma_2}^0} )
\otimes \rho_{X_2 !} \omega_{Z_2|X_2}^{\otimes -1} ) \\
& \to
(Rp_{1!!} \tilde f^{-1}
( \rsect_{\Omega_2} K_{P_{\sigma_2}^0} )) 
\otimes  f^{-1}\rho_{X_2 !}  \omega_{Z_2|X_2}^{\otimes -1}  \\
& \to
Rp_{1!!} ( \rsect_{\Omega_1} K_{P_{\sigma_1}^0} )
\otimes  f^{-1}\rho_{X_2 !}  \omega_{Z_2|X_2}^{\otimes -1}  \\
&=
\shl_{\sigma_1}^0 \otimes  \rho_{X_1 !}(  \omega_{Z_1|X_1} )
\otimes  f^{-1}\rho_{X_2 !} \omega_{Z_2|X_2}^{\otimes -1} ,
\end{align*}
where the second line is given by~\eqref{eq:presque_chgt_base} and in
the third line we use the morphism $\tilde f^{-1} \rsect_{\Omega_2}
K_{P_{\sigma_2}^0} \to \rsect_{\Omega_1} K_{P_{\sigma_1}^0}$, obtained
from the morphism of functor $\smash{\tilde f^{-1}} \rsect_{\Omega_2}
\to \rsect_{\Omega_1} \smash{\tilde f^{-1}}$ and
Lemma~\ref{lem:fonctorialite_K_T}, (i).

If $f$ is transversal to $Z_2$ we have $\omega_{Z_1|X_1} \simeq f^{-1}
\omega_{Z_2|X_2}$.
\end{proof}
Now we have the following analog of
Proposition~\ref{prop:prop_fonct_L} for $\sha$-modules, with the
additional hypothesis that $f$ is smooth, for the case of direct
image.
\begin{proposition}
  \label{prop:prop_fonct_L_sha}
  (i) Assume that $f$ is smooth. Then there exists a natural morphism
  of dg-$\sha_{X_2}$-modules:
\begin{equation}
  \label{eq:ima_dir_L_sha}
  f_{!!} ( \shl^\sha_{\sigma_1}  \otimes 
\rho_{X_1!}(  \omega_{Z_1 |Z_2} )) \to \shl^\sha_{\sigma_2},
\end{equation}
whose image in $\BDC(\C_{(X_2)_{sa}})$ is
morphism~\eqref{eq:ima_dir_L}.

(ii) Assume that $Z_1 = f^{-1}(Z_2)$ and $f$ is clean with respect to
$Z_2$. Then there exists a natural morphism of
dg-$\sha_{X_1}$-modules:
  \begin{equation}
    \label{eq:restr_L_sha}
    f^* \shl^\sha_{\sigma_2} \to \shl^\sha_{\sigma_1}
\otimes \rho_{X_1!}(  \omega_{Z_1 |Z_2} )
\otimes \omega_{X_1 |X_2}^{-1},
  \end{equation}
  whose image in $\BDC(\C_{(X_1)_{sa}})$ is
  morphism~\eqref{eq:restr_L}.  If $f$ is transversal to $Z_2$ it
  becomes:
  \begin{equation}
    \label{eq:restr_L_iso_sha}
    f^* \shl^\sha_{\sigma_2} \to \shl^\sha_{\sigma_1} .
  \end{equation}
\end{proposition}
\begin{proof}
  The proof is similar to the proof of
  Proposition~\ref{prop:prop_fonct_L}. We keep the same notations and
  we just point out the changes.
  
  (i) We note that $\tilde f$ is smooth on $\Omega_1$ and apply
  Lemma~\ref{lem:fonctorialite_K_T}. This gives the morphisms:
  \begin{align*}
    f_{!!} (p_{1!!} 
( \sect_{\Omega_1} \shb_{P_{\sigma_1}^0} )
\otimes \omega_{X_1|X_2})
&\simeq
p_{2!!} \tilde f_{!!} 
\sect_{\Omega_1} ( \sha_{\tilde X_1} 
\otimes \tilde f^{-1} K_{P_{\sigma_2}^0} 
\otimes \omega'_{\tilde X_1| \tilde X_2} )  \\
&\to 
p_{2!!} ( \sect_{\Omega_2} \shb_{P_{\sigma_2}^0} ),
\end{align*}
and the tensor product with $\omega_{Z_2 |X_2}^{\otimes-1}$
gives~\eqref{eq:ima_dir_L_sha}.

(ii) Morphism~\eqref{eq:presque_chgt_base} has a non derived version
$f^{-1} p_{2!!} \to p_{1!!} \tilde f^{-1}$.  Taking the tensor product
$\sha_{X_1} \otimes_{f^{-1}\sha_{X_2}} \cdot$ and using the projection
formula we obtain:
\begin{equation}
  \label{eq:presque_chgt_base_sha}
  f^* p_{2!!} \to p_{1!!} \tilde f^*
\end{equation}
Now we define \eqref{eq:restr_L_sha} by the sequence of morphisms:
\begin{align*}
  f^* \shl^\sha_{\sigma_2} &= f^* ( p_{2!!}  ( \sect_{\Omega_2}
  \shb_{P_{\sigma_2}^0} )
  \otimes \rho_{X_2!}(  \omega_{Z_2|X_2}^{\otimes -1}) ) \\
  & \to (  p_{1!!} \tilde f^* ( \sect_{\Omega_2}
  \shb_{P_{\sigma_2}^0} ) )
  \otimes  f^{-1}\rho_{X_2!}  \omega_{Z_2|X_2}^{\otimes -1}  \\
  & \to p_{1!!}  ( \sect_{\Omega_1} \shb_{P_{\sigma_1}^0} )
  \otimes  f^{-1}\rho_{X_2!}  \omega_{Z_2|X_2}^{\otimes -1}  \\
  &= \shl^\sha_{\sigma_1} \otimes \rho_{X_1!}( \omega_{Z_1|X_1} )
  \otimes f^{-1}\rho_{X_2!}  \omega_{Z_2|X_2}^{\otimes -1} ,
\end{align*}
where the second line is given by~\eqref{eq:presque_chgt_base_sha} and
the third line is the composition
\begin{multline*}
\tilde f^* ( \sect_{\Omega_2} \shb_{P_{\sigma_2}^0} ) 
=
\sha_{\tilde X_1} \otimes_{f^{-1}\sha_{\tilde X_2}}
f^{-1}\sect_{\Omega_2}(\sha_{\tilde X_2} 
\otimes K_{P_{\sigma_2}^0}) \\
\to
\sha_{\tilde X_1} \otimes_{f^{-1}\sha_{\tilde X_2}}
\sect_{\Omega_1}(f^{-1}\sha_{\tilde X_2} \otimes K_{P_{\sigma_1}^0})
\to \sect_{\Omega_1} \shb_{P_{\sigma_1}}
\end{multline*}
of standard morphisms of sheaves and the isomorphism of
Lemma~\ref{lem:fonctorialite_K_T}, (i).
\end{proof}

\subsection{External tensor product}
The external tensor product is a consequence of Proposition~1.3.8
of~\cite{KSIW06}. We give a different proof here, using the kernel
$\shl_\sigma^0$ (hence our morphism coincides with the one
in~\cite{KSIW06} for a non-vanishing $\sigma$) and check that it works
for $\sha$-modules.  We still consider $(X_1, Z_1, \sigma_1)$ and
$(X_2, Z_2, \sigma_2)$ as in the beginning of this section.  We set $X
= X_1 \times X_2$, $Z= Z_1 \times Z_2$, $\sigma = \sigma_1+\sigma_2$.
Then $(X,Z,\sigma)$ also is a kernel data
satisfying~\eqref{eq:hyp_non_nul}.  We keep the notations of
diagram~\eqref{eq:diag_prop_fonct_L} and let $p\cl \tilde X_Z \to X$
be the projection.  We also have a natural embedding $k\cl \tilde X_Z
\to \tilde X_1 \times \tilde X_2$.  We set $p'= p_1\times p_2 \cl
\tilde X_1 \times \tilde X_2 \to X$.
\begin{proposition}
  \label{prop:prod_tens_ext}
  There exists a morphism $\shl_{\sigma_1}^0 \boxtimes
  \shl_{\sigma_2}^0 \to \shl_\sigma^0$ in $\BDC(\C_{(X_1 \times
    X_2)_{sa}})$.
\end{proposition}
\begin{proof}
  The kernel $\shl_{\sigma_i}^0$ is the tensor product of $Rp_{i!!} (
  \rsect_{\Omega_i} K_{P_{\sigma_i}^0} )$ and $\rho_{X_i !}( 
  \omega_{Z_i|X_i}^{\otimes -1})$.  The external product for the
  second term is straightforward:
$$
\rho_{X_1 !}(  \omega_{Z_1|X_1}^{\otimes -1})
 \boxtimes \rho_{X_2 !}(  \omega_{Z_2|X_2}^{\otimes -1})
\simeq \rho_{X!}(  \omega_{Z|X}^{\otimes -1})
$$
and now we only take care of the first term.  We have the sequence
of morphisms:
\begin{align*}
(Rp_{1!!} \rsect_{\Omega_1} K_{P_{\sigma_1}^0} )
\boxtimes 
(Rp_{2!!} \rsect_{\Omega_2} K_{P_{\sigma_2}^0})
&\to  
Rp'_{!!} \rsect_{\Omega_1 \times \Omega_2}
(K_{P_{\sigma_1}^0} \boxtimes K_{P_{\sigma_2}^0} )  \\
&\to
Rp'_{!!} k_{!!} k^{-1}  \rsect_{\Omega_1 \times \Omega_2}
(K_{P_{\sigma_1}^0} \boxtimes K_{P_{\sigma_2}^0} )  \\
&\to
Rp_{!!} \rsect_\Omega ( k^{-1}
(K_{P_{\sigma_1}^0} \boxtimes K_{P_{\sigma_2}^0} ))  \\
&\to
Rp_{!!} \rsect_\Omega ( K_{P_\sigma^0} ),
\end{align*}
where the first three arrows are standard morphisms of sheaves and the
last one is defined as follows.  We recall that $K_{P_{\sigma_i}^0}
\simeq \varinjlim_{W_i,W_i^0} \C_{\overline{W_i} \setminus
  \overline{W_i^0}}$, where $W_i$, $W_i^0$ run over the open
neighborhoods of $P_{\sigma_i}$, $P_{\sigma_i}^0$ in $\tilde X_i$.
For such $W_i$, $W_i^0$ we have
\begin{gather*}
  (\overline{W_1} \setminus \overline{W_1^0} )\times (
\overline{W_2} \setminus \overline{W_2^0})
=
(\overline{W_1}\times \overline{W_2} ) \setminus 
((\overline{W_1}\times \overline{W_2} ) \cap \overline{W^0})
= 
\overline{W}\setminus \overline{W^0},  \\
\text{where}\quad
W^0 = W_1^0 \times \tilde X_2 \cup
\tilde X_1 \times W_2^0
\quad\text{and}\quad
W = (W_1 \times W_2) \cup W^0 .
\end{gather*}
Now $W$ and $W^0$ are open neighborhoods of $P_\sigma$ and
$P_\sigma^0$ in $\tilde X_1 \times \tilde X_2$ (note that $P_\sigma
\subset T_ZX$ and $T_ZX$ can be viewed as a subset of $\tilde X_1
\times \tilde X_2$).  This defines a natural morphism
$K_{P_{\sigma_1}^0} \boxtimes K_{P_{\sigma_2}^0} \to
\varinjlim_{W,W^0} \C_{\overline{W} \setminus \overline{W^0}}$, where
$W$ and $W^0$ run over the open neighborhoods of $P_\sigma$ and
$P_\sigma^0$ in $\tilde X_1 \times \tilde X_2$.  The inverse image by
$k$ gives the required morphism $k^{-1} (K_{P_{\sigma_1}^0} \boxtimes
K_{P_{\sigma_2}^0} ) \to K_{P_\sigma^0}$.
\end{proof}
\begin{proposition}
   \label{prop:prod_tens_ext_sha}
   There exists a morphism $\shl_{\sigma_1}^\sha \underline{\boxtimes}
   \shl_{\sigma_2}^\sha \to \shl_\sigma^\sha$ in $\Mod(\sha_{X_1
     \times X_2})$, whose image in $\BDC(\C_{(X_1 \times X_2)_{sa}})$
   is the morphism of Proposition~\ref{prop:prod_tens_ext}.
\end{proposition}
\begin{proof}
  The proof of the previous proposition adapts immediately, with the
  following modifications in the sequence of morphisms:
\begin{align*}
(p_{1!!} \sect_{\Omega_1} &(\sha_{\tilde X_1} \otimes K_{P_{\sigma_1}^0}))
\boxtimes 
(p_{2!!} \sect_{\Omega_2} (\sha_{\tilde X_2} \otimes K_{P_{\sigma_2}^0}))
\\
&\to  
p'_{!!} \sect_{\Omega_1 \times \Omega_2}
(\sha_{\tilde X_1 \times \tilde X_2} \otimes
(K_{P_{\sigma_1}^0} \boxtimes K_{P_{\sigma_2}^0} ))  \\
&\to  
p'_{!!}  k_{!!} k^{-1} \sect_{\Omega_1 \times \Omega_2}
(\sha_{\tilde X_1 \times \tilde X_2} \otimes
(K_{P_{\sigma_1}^0} \boxtimes K_{P_{\sigma_2}^0} ))  \\
&\to
p_{!!} \sect_\Omega ( \sha_{\tilde X_Z} \otimes  k^{-1}
(K_{P_{\sigma_1}^0} \boxtimes K_{P_{\sigma_2}^0} ))  \\
&\to
p_{!!} \sect_\Omega (\sha_{\tilde X_Z} \otimes  K_{P_\sigma^0} ).
\end{align*}
\end{proof}

\section{Functorial properties of microlocalization}
\label{sec:functorial_properties}
In this section $f\cl X \to Y$ is a morphism of real analytic
manifolds. We recall the functorial behavior of microlocalization
with respect to inverse image, in case $f$ is an embedding, and to
direct image. We check that the constructions make sense for
dg-$\sha$-modules (restricting to the case of a smooth map for the
direct image).

We define the submanifold $Z = X\times_Y T^*Y$ diagonally embedded in
$X \times (X\times_Y T^*Y)$. We have the morphisms of kernel data
$$
\xymatrix@C=2cm@R=5mm{
X \petittimes T^*X  
& X \petittimes (X\petittimes_Y T^*Y) \ar[l]_-{\id \times f_d} 
 \ar[r]^-{f\times f_\pi}  
& Y \petittimes T^*Y   \\
X \petittimes_X T^*X  \ar@{}[u]|\bigcup
& Z \ar@{}[u]|\bigcup  \ar[l] \ar[r]
& Y \petittimes_Y T^*Y  \ar@{}[u]|\bigcup  \\
T^*X \ar@{=}[u] 
&X\times_Y T^*Y  \ar[l]_{f_d}  \ar[r]^{f_\pi} \ar@{=}[u] 
& T^*Y,  \ar@{=}[u]
}
$$
where the $1$-form for the kernel corresponding to the middle
column is 
$$
\sigma_{Y\from X} = (\id\times f_d)^*(\sigma_X)
= (f\times f_\pi)^*(\sigma_{Y}).
$$
This equality follows from $f_d^*(\omega_X)= f_\pi^*(\omega_Y)$.
We note that $Z = (\id\times f_d)^{-1}(X\times_X T^*X)$ and $ Z
\subset (f\times f_\pi)^{-1} (Y\times_Y T^*Y)$, with equality if $f$
is an embedding. This implies that hypothesis~\eqref{eq:hyp_non_nul}
is satisfied for $(X \times (X\times_Y T^*Y), Z, \sigma_{Y\from X} )$.
We denote the corresponding kernel by $L_{Y\from X} =
\shl_{\sigma_{Y\from X}}$.

\subsection{Microlocalization and inverse image}
For the next two propositions we assume that $f\cl X \to Y$ is an
embedding. For $G \in \DC^+(\I(\C_Y))$ we have a morphism $R f_{d!!}
f_\pi^{-1} \mu_Y(G) \to \mu_X(f^{-1}G)$, defined in Theorem~2.4.4
of~\cite{KSIW06}.  We recall its construction below. The notations are
introduced in the diagram:
\begin{equation}
  \label{eq:diag_kern_data}
\vcenter{
\xymatrix@C=20mm{
X \ar@{=}[r]& X \ar[r]_f & Y  \\
X \petittimes T^*X  \ar[d]_{p_2} \ar[u]^{p_1}
& X \petittimes (X\petittimes_Y T^*Y) \ar[l]_-{\id \times f_d} 
\ar[r]^-{f\times f_\pi}  \ar[u]^{p} \ar[d]_r
& Y \petittimes T^*Y  \ar[d]^{q_2} \ar[u]_{q_1} \\
T^*X & X\times_Y T^*Y  \ar[l]_{f_d}  \ar[r]^{f_\pi}
& T^*Y }}
\end{equation}
\begin{proposition}[\cite{KSIW06}, Theorem~2.4.4]
\label{prop:micro_im_inv}
We have a natural morphism, for an embedding $f\cl X \to Y$ and $G \in
\DC^+(\I(\C_Y))$:
\begin{equation}
  \label{eq:micro_im_inv}
  R f_{d!!} f_\pi^{-1} \mu_Y(G) \to \mu_X(f^{-1}G).
\end{equation}
\end{proposition}
\begin{proof}
  We first note the morphism of functors $f_\pi^{-1} Rq_{2!!} \to
  Rr_{!!} (f\times f_\pi)^{-1}$. It is obtained by the following
  composition of adjunction morphisms, where we use the fact that $f$,
  hence $f_\pi$ and $f\times f_\pi$, are embeddings, so that direct
  and proper direct images coincide:
\begin{equation}
  \label{eq:mii_f1}
  \begin{split}
f_\pi^{-1} Rq_{2!!}
\to
f_\pi^{-1} Rq_{2!!} 
& (f\times f_\pi)_* (f\times f_\pi)^{-1}  \\
&\simeq
f_\pi^{-1} R f_{\pi *} Rr_{!!} (f\times f_\pi)^{-1}
\to
Rr_{!!} (f\times f_\pi)^{-1}.
\end{split}
\end{equation}
We also note the morphisms of kernels:
\begin{equation}
  \label{eq:mii_mk}
  (f \times f_\pi)^{-1} L_Y  \to L_{Y\from X} \otimes \omega_{X|Y}^{-1}
\qquad
R(\id \times f_d)_{!!} (L_{Y\from X}\otimes \omega_{X|Y}^{-1}) \to L_X.
\end{equation}
The first one is morphism~\eqref{eq:restr_L} of
Proposition~\ref{prop:prop_fonct_L} (for $\shl_\sigma$ instead of
$\shl_\sigma^0$), applied to $f\times f_\pi$: we note that $f\times
f_\pi$ is clean with respect to $Y\times_Y T^*Y$ and $X\times_Y T^*Y =
(f\times f_\pi)^{-1} (Y\times_Y T^*Y)$.  The second one is
morphism~\eqref{eq:ima_dir_L} (for $\shl_\sigma$ instead of
$\shl_\sigma^0$), applied to $\id \times f_d$.

Now the morphism of the lemma is defined by the succession of
morphisms:
\begin{align}
  Rf_{d!!} f_\pi^{-1} \mu_Y(G)
&=
Rf_{d!!} f_\pi^{-1} Rq_{2!!} ( L_Y \otimes q_1^{-1} G)  \\ 
\label{eq:mii_l3}
& \to
Rf_{d!!} Rr_{!!} ( (f \times f_\pi)^{-1}L_Y \otimes p^{-1} f^{-1}G) \\
\label{eq:mii_l4}
& \isofrom
Rp_{2!!}  ( R(\id \times f_d)_{!!}
 (f \times f_\pi)^{-1} L_Y \otimes p_1^{-1} f^{-1} G) \\
\label{eq:mii_l5}
& \to
Rp_{2!!} ( L_X \otimes p_1^{-1} f^{-1} G),
\end{align}
where in line~\eqref{eq:mii_l3} we used morphism~\eqref{eq:mii_f1} and
the commutativity of inverse image and tensor product, and in
line~\eqref{eq:mii_l4} the identities $f_d r = p_2(\id\times f_d)$, $p
= p_1 (\id\times f_d)$ and the projection formula for $(\id\times
f_d)$.  The last morphism is the composition of the morphisms
in~\eqref{eq:mii_mk}.
\end{proof}

\begin{proposition}
\label{prop:micro_im_inv_sha}
For an embedding $f\cl X \to Y$ and $G \in \Mod(\sha_{Y})$, we have a
morphism of $\sha_{T^*X}$-modules:
\begin{equation}
  \label{eq:micro_im_inv_sha}
f_{d!!} f_\pi^* \mu^\sha_Y(G) \to \mu^\sha_X(f^*G),
\end{equation}
which makes a commutative diagram in $\DC^+(\I(\C_{\Tnz^*X}))$ with
morphism~\eqref{eq:micro_im_inv}:
$$
\xymatrix@C=15mm{
R f_{d!!} f_\pi^{-1} \mu_Y(G) \ar[d] \ar[r]
& \mu_X(f^{-1}G) \ar[d]  \\
R f_{d!!} f_\pi^{*} \mu^\sha_Y(G)
\isofrom f_{d!!} f_\pi^{*} \mu^\sha_Y(G) \ar[r]
& \mu^\sha_X(f^*G).
}
$$
\end{proposition}
\begin{proof}
  We follow the construction of morphism~\eqref{eq:micro_im_inv},
  replacing each morphism by its analog for $\sha$-modules. We have
  the analogs of morphisms~\eqref{eq:mii_f1} and~\eqref{eq:mii_mk}:
\begin{gather}
  \label{eq:mii_sha_f1}
  \begin{split}
f_\pi^* q_{2!!}
\to
f_\pi^* q_{2!!} 
& (f\times f_\pi)_* (f\times f_\pi)^*  \\
&\simeq
f_\pi^*  f_{\pi *} r_{!!} (f\times f_\pi)^*
\to
r_{!!} (f\times f_\pi)^*,
\end{split}  \\
  \label{eq:mii_sha_mk}
  (f \times f_\pi)^* L^\sha_Y  
\to L^\sha_{Y\from X} \otimes \omega_{X|Y}^{-1},
\qquad
(\id \times f_d)_{!!} (L^\sha_{Y\from X}\otimes \omega_{X|Y}^{-1}) 
\to L^\sha_X.
\end{gather}
Morphism~\eqref{eq:mii_sha_f1} is defined with the same adjunction
properties as morphism~\eqref{eq:mii_f1}. The morphisms in
line~\eqref{eq:mii_sha_mk} are defined the same way
as~\eqref{eq:mii_mk}, using Proposition~\ref{prop:prop_fonct_L_sha}
instead of Proposition~\ref{prop:prop_fonct_L}. We deduce the
succession of morphisms:
\begin{align}
  f_{d!!} f_\pi^* \mu_Y^\sha(G)
&=
f_{d!!} f_\pi^* q_{2!!} ( L^\sha_Y \otimes_\sha q_1^* G)  \\ 
\label{eq:mii_sha_l3}
& \to
f_{d!!} r_{!!} ( (f \times f_\pi)^* L^\sha_Y \otimes_\sha p^* f^*G) \\
\label{eq:mii_sha_l4}
& \isofrom
p_{2!!}  ( (\id \times f_d)_{!!}
 (f \times f_\pi)^* L^\sha_Y \otimes_\sha p_1^* f^* G) \\
\label{eq:mii_sha_l5}
& \to
p_{2!!} ( L^\sha_X \otimes_\sha p_1^* f^* G),
\end{align}
where in line~\eqref{eq:mii_sha_l3} we used
morphism~\eqref{eq:mii_sha_f1} and the commutativity of inverse image
and tensor product, and in line~\eqref{eq:mii_sha_l4} the identities
$f_d r = p_2(\id\times f_d)$, $p = p_1 (\id\times f_d)$ and the
projection formula for $(\id\times f_d)$ (Lemma~\ref{lem:proj_form}).
The last morphism is the composition of the morphisms
in~\eqref{eq:mii_sha_mk}.

The vertical arrows in the diagram are the compositions of $g^{-1} \to
g^*$ respectively for $g= f_\pi$, $g = f$, and $\mu_Z \to \mu_Z^\sha$,
respectively for $Z=Y$, $Z=X$.  This last morphism is defined only on
$\Tnz^*Z$. The diagram commutes because it is obtained by morphisms of
functors.  The isomorphism between the direct image and the derived
direct image by $f_d$ follows from the softness of $\sha$-modules
(Proposition~\ref{prop:Amod_f_*acyclique}).
\end{proof}

\subsection{Microlocalization and direct image}
In Proposition~\ref{prop:micro_im_dir} below we recall a weak version
of the direct image morphism, defined in Theorem~2.4.2
of~\cite{KSIW06}. This theorem gives a morphism, for $F \in
\BDC(\I(\C_X))$, $R f_{\pi !!}  f_d^{-1} \mu_X(F) \to \mu_Y(f_{!!}F)$.
We consider the case where $F = f^{-1} G \otimes \omega_{X|Y}$ which
is sufficient for our purpose, and we give an easier proof in this
case.  This proof also works for the resolutions by $\sha$-modules,
assuming moreover that $f$ is smooth (see
Proposition~\ref{prop:micro_im_dir_sha}).  We use the notations of
diagram~\eqref{eq:diag_kern_data}.
\begin{proposition}[special case of \cite{KSIW06}, Theorem~2.4.2]
\label{prop:micro_im_dir}
There exists a natural morphism, for $f\cl X \to Y$ and $G \in
\DC^+(\I(\C_Y))$:
\begin{equation}
  \label{eq:micro_im_dir}
  R f_{\pi !!} f_d^{-1} \mu_X(f^{-1}G \otimes \omega_{X|Y})
\to \mu_Y(G).
\end{equation}
\end{proposition}
\begin{proof}
  We set $F = f^{-1}G \otimes \omega_{X|Y}$ and obtain the sequence of
  morphisms:
\begin{align}
R f_{\pi !!} &f_d^{-1} \mu_X(F)  \\
&=
R f_{\pi !!} f_d^{-1} Rp_{2!!} ( L_X \otimes p_1^{-1}  F) \\
\label{eq:mid_l2}
&\isoto
R f_{\pi !!} Rr_{!!} ( (\id \times f_d)^{-1} L_X
\otimes p^{-1}  F) \\
\label{eq:mid_l3}
&\simeq
Rq_{2!!} R(f\times f_\pi)_{!!} ( (\id \times f_d)^{-1} L_X
\otimes p^{-1}\omega_{X|Y}
\otimes (f\times f_\pi)^{-1} q_1^{-1}  G  ) \\
\label{eq:mid_l4}
&\isofrom
Rq_{2!!}  ( R(f\times f_\pi)_{!!} ((\id \times f_d)^{-1} L_X
\otimes p^{-1}\omega_{X|Y} )
\otimes  q_1^{-1}  G  ) \\
\label{eq:mid_l5}
& \to
Rq_{2!!}( L_Y \otimes q_1^{-1}  G),
\end{align}
where in line~\eqref{eq:mid_l2} we used the base change formula
$f_d^{-1} Rp_{2!!} \isoto Rr_{!!} (\id \times f_d)^{-1}$ and the
identity $p = p_1 (\id \times f_d)$, in line~\eqref{eq:mid_l3} the
identities $f_{\pi} r = q_2 (f\times f_\pi)$ and $f p = q_1 (f\times
f_\pi)$, and in line~\eqref{eq:mid_l4} the projection formula for
$(f\times f_\pi)$.  The last line is given by the composition of
$$
(\id \times f_d)^{-1} L_X \to L_{Y\from X}
\quad \text{and} \quad
R(f\times f_\pi)_{!!} ( L_{Y\from X}\otimes p^{-1}\omega_{X|Y} )
\to L_Y,
$$
which are respectively given by (ii) and (i) of
Proposition~\ref{prop:prop_fonct_L} (for the first morphism we note
that $(\id \times f_d)$ is transversal to $X\times_X T^*X$ and for the
second one we note that the restriction of $p^{-1}\omega_{X|Y}$ to
$X\times_Y T^*Y$ is isomorphic to $\omega_{X\times_Y T^*Y | Y\times_Y
  T^*Y}$).
\end{proof}

The following proposition gives a realization of
morphism~\eqref{eq:micro_im_dir} by $\sha$-modules. We restrict to the
case where $f$ is a submersion because we only have an integration
morphism in this case.
\begin{proposition}
\label{prop:micro_im_dir_sha}
There exists a natural morphism of $\sha_{T^*Y}$-modules, for a
submersion $f\cl X \to Y$ and for $G \in \Mod(\sha_{Y})$:
\begin{equation}
  \label{eq:micro_im_dir_sha}
  f_{\pi !!} f_d^* \mu^\sha_X(f^* G \otimes \omega'_{X|Y})
\to \mu^\sha_Y(G),
\end{equation}
which makes a commutative diagram in $\DC^+(\I(\C_{\Tnz^*Y}))$ with
morphism~\eqref{eq:micro_im_dir}:
$$
\xymatrix@C=15mm{
R f_{\pi !!} f_d^{-1} \mu_X(f^{-1}G \otimes \omega_{X|Y})
\ar[r] \ar[d]
& \mu_Y(G) \ar[d]  \\
R f_{\pi !!} f_d^{*} \mu^\sha_X(f^{*}G \otimes \omega'_{X|Y}) 
\isofrom
f_{\pi !!} f_d^{*} \mu^\sha_X(f^{*}G \otimes \omega'_{X|Y}) \ar[r]
& \mu^\sha_Y(G) .
}
$$
\end{proposition}
\begin{proof}
  We follow the proof of Proposition~\ref{prop:micro_im_dir}, but now
  we consider morphisms of $\sha$-modules.  We set $F = f^*G \otimes
  \omega'_{X|Y}$ and obtain the sequence of morphisms:
\begin{align}
f_{\pi !!} &f_d^* \mu_X^\sha(F)  \\
&=
f_{\pi !!} f_d^* p_{2!!} ( L^\sha_X \otimes_{\sha} p_1^*  F) \\
\label{eq:mid_sha_l2}
&\to
f_{\pi !!} r_{!!} ( (\id \times f_d)^* L^\sha_X
\otimes p^*  F) \\
\label{eq:mid_sha_l3}
&\simeq
q_{2!!} (f\times f_\pi)_{!!} ( (\id \times f_d)^* L^\sha_X
\otimes p^{-1}\omega'_{X|Y}
\otimes_{\sha} (f\times f_\pi)^* q_1^*  G  ) \\
\label{eq:mid_sha_l4}
&\isofrom
q_{2!!}  ( (f\times f_\pi)_{!!} ((\id \times f_d)^* L^\sha_X
\otimes p^{-1}\omega'_{X|Y} )
\otimes_{\sha}  q_1^*  G  ) \\
\label{eq:mid_sha_l5}
& \to
q_{2!!}( L^\sha_Y \otimes_\sha q_1^*  G),
\end{align}
where in line~\eqref{eq:mid_sha_l2} we used the base change formula
$f_d^* p_{2!!} \to r_{!!} (\id \times f_d)^*$ and the identity $p =
p_1 (\id \times f_d)$, in line~\eqref{eq:mid_sha_l3} the identities
$f_{\pi} r = q_2 (f\times f_\pi)$ and $f p = q_1 (f\times f_\pi)$, and
in line~\eqref{eq:mid_sha_l4} the projection formula for $(f\times
f_\pi)$.  The last line is given by the composition of
$$
(\id \times f_d)^* L^\sha_X \to L^\sha_{Y\from X}
\quad \text{and} \quad
(f\times f_\pi)_{!!} ( L^\sha_{Y\from X}\otimes p^{-1}\omega'_{X|Y} )
\to L^\sha_Y,
$$
which are given by (ii) and (i) of
Proposition~\ref{prop:prop_fonct_L_sha}.

The diagram is defined as in Proposition~\ref{prop:micro_im_inv_sha}.
\end{proof}

\subsection{External tensor product}
We consider $X,Y$ as above and $F \in \DC^+(\I(\C_X))$, $G \in
\DC^+(\I(\C_Y))$. Proposition~2.1.14 of~\cite{KSIW06} implies the
existence of a natural morphism:
\begin{equation}
  \label{eq:prod_ext_mu}
  \mu_X F \boxtimes \mu_Y G \to \mu_{X\times Y}(F \boxtimes G).
\end{equation}
\begin{proposition}
\label{prop:prod_ext_mu_sha}
  For $F \in \Mod(\sha_{X})$ and $G \in \Mod(\sha_{Y})$
there exists a natural morphism
$$
\mu^\sha_X F \underline{\boxtimes} \mu^\sha_Y G
\to \mu^\sha_{X\times Y}(F \boxtimes G).
$$
Its restriction to $\Tnz^*X \times \Tnz^*Y$ makes a commutative
diagram with morphism~\eqref{eq:prod_ext_mu} in $\DC^+(\I(\C_{\Tnz^*X
  \times \Tnz^*Y}))$:
$$
\xymatrix{
\mu_X F \boxtimes \mu_Y G  \ar[r]\ar[d]
& \mu_{X\times Y}(F \boxtimes G)\ar[d] \\
\mu^\sha_X F \boxtimes \mu^\sha_Y G  \ar[r]
& \mu^\sha_{X\times Y}(F \boxtimes G) .
}
$$
\end{proposition}
\begin{proof}
  The existence of the morphisms follows from the Künneth formula and
  Proposition~\ref{prop:prod_tens_ext_sha}.  It coincides with the
  already known construction outside the zero section by
  Proposition~\ref{prop:mu_sha_qinj}.
\end{proof}

\section{Composition of kernels}
\label{sec:con_kern}
We recall the microlocal composition of kernels defined
in~\cite{KSIW06}, Theorem~2.5.1, and we check that a similar
construction also works for $\sha$-modules. This construction is a
composition of the operations recalled in
section~\ref{sec:functorial_properties}, and we just have to check
that the restrictive hypothesis assumed in the case of $\sha$-modules
are satisfied.

We first recall some standard notations and definitions.  We consider
three analytic manifolds $X$, $Y$, $Z$ and we let $q_{ij}$ be the
$(i,j)$-th projection from $X \times Y \times Z$ and $p_{ij}$ the
$(i,j)$-th projection from $T^*X\times T^*Y\times T^*Z$.  We also
denote by $a\cl T^*Y \to T^*Y$ the antipodal map and we set $p_{12}^a
= (\id \times a) \circ p_{12}$.  For $F \in \DC^+(\I(\C_{X\times
  Y}))$, $G \in \DC^+(\I(\C_{Y\times Z}))$ and
$\me{F}\in\DC^+(\I(\C_{T^*X\times T^*Y}))$,
$\me{G}\in\DC^+(\I(\C_{T^*Y\times T^*Z}))$ we define:
\begin{equation}
  \label{eq:def_convol}
F \circ G =    Rq_{13!!} (q_{12}^{-1} F
\otimes q_{23}^{-1} G ),
\qquad
\me{F}\overset{a}{\circ} \me{G}= Rp_{13!!} (p_{12}^{a-1}\me{F}
\otimes p_{23}^{-1}\me{G} ).
\end{equation}
We set for short $M = X\times Y\times Y\times Z$,
$N = X\times Y\times Z$ and let $j\cl N \to M$ be the diagonal embedding.
We define the maps:
\begin{alignat*}{2}
  k  &\cl T^*N \hookrightarrow N \times_M  T^*M,
&\quad
(x,y,z;\xi,\eta,\zeta) &\mapsto (x,y,y,z;\xi,-\eta,\eta,\zeta)  \\
\tau  &\cl T^*N \to N \times_{X\times Z} T^*(X\times Z),
&\quad
(x,y,z;\xi,\eta,\zeta) &\mapsto (x,y,z;\xi,\zeta)  \\
p &= j_\pi \circ k
\end{alignat*}
and obtain the following commutative diagram, with a Cartesian square:
\begin{equation}
  \label{eq:diag_noyaux}
\vcenter{
\xymatrix@C=1.5cm{
    {T^*N} \ar[d]^{\tau} \ar@<-1pt>@{^{(}->}[r]^{k}
\ar@/^25pt/[rr]_{p}  \ar@/_60pt/[dd]_{p_{13}}    \ar@{}[rd]|{\square} 
& {N\times_M T^*M}
     \ar[d]_(.6){j_d}  \ar@<-1pt>@{^{(}->}[r]^(.6){j_\pi} 
&  {T^*M} \\
    {N \times_{X\times Z} T^*(X\times Z)} 
\ar@<-1pt>@{^{(}->}[r]_(.6){q_{13}{}_d}
 \ar[d]^(.6){q_{13\pi}} & 
    {T^*N} & \\
  {T^*(X\times Z)} & &
     } }
\end{equation}
We note that $\me{F}\overset{a}{\circ} \me{G} \simeq Rp_{13!!} \,
p^{-1} (\me{F} \boxtimes \me{G})$. Theorem~2.5.1 of~\cite{KSIW06}
gives a natural morphism, the composition of kernels:
\begin{equation}
  \label{eq:comp_kernels_original}
\mu_{X\times Y} K_1 \overset{a}{\circ}
\mu_{Y\times Z} K_2 \to \mu_{X\times Z} (K_1 \circ K_2),
\end{equation}
for $K_1 \in \DC^+(\I(\C_{X\times Y}))$, $K_2 \in \DC^+(\I(\C_{Y\times
  Z}))$.  Since the commutation of microlocalization and direct image
has a weaker statement in the case of $\sha$-modules than in the case
of ind-sheaves of vector spaces, we also give a weaker statement
than~\eqref{eq:comp_kernels_original} for the composition of kernels.

In fact, for ind-sheaves, morphism~\eqref{eq:micro_convol} below is
equivalent to~\eqref{eq:comp_kernels_original}: indeed using the
adjunction between $Rq_{13 !!}$ and $q_{13}^!$ we may
apply~\eqref{eq:micro_convol} to $K_3 = K_1 \circ K_2$ and
recover~\eqref{eq:comp_kernels_original}.  But for $\sha$-modules we
don't have this adjunction and the statement of
Proposition~\ref{prop:convol_sha} is actually weaker than an
$\sha$-module analog of~\eqref{eq:comp_kernels_original}.
\begin{proposition}
\label{prop:convol}
For complexes $K_1 \in \DC^+(\I(\C_{X\times Y}))$, $K_2 \in
\DC^+(\I(\C_{Y\times Z}))$ and $K_3 \in \DC^+(\I(\C_{X\times Z}))$,
with a morphism $q_{12}^{-1} K_1 \otimes q_{23}^{-1} K_2 \to
q_{13}^{-1} K_3 \otimes \omega_Y$, there exists a natural morphism
  \begin{equation}
    \label{eq:micro_convol}
    \mu_{X\times Y} K_1 \overset{a}{\circ} \mu_{Y\times Z} K_2
\to \mu_{X\times Z} K_3.
  \end{equation}
\end{proposition}
\begin{proof}
  By definition $\mu_{X\times Y} K_1 \overset{a}{\circ} \mu_{Y\times
    Z} K_2 = Rp_{13!!} p^{-1} ( \mu_{X\times Y} K_1 \boxtimes
  \mu_{Y\times Z} K_2)$. The external tensor
  product~\eqref{eq:prod_ext_mu} gives $\mu_{X\times Y} K_1 \boxtimes
  \mu_{Y\times Z} K_2 \to \mu_M(K_1 \boxtimes K_2)$ and the base
  change formula gives $Rp_{13!!} p^{-1} = Rq_{13\pi !!}  R\tau_{!!}
  k^{-1} \, j_\pi^{-1} \simeq Rq_{13\pi !!} q_{13d}^{-1} \, Rj_{d!!}
  j_\pi^{-1}$. We obtain the morphisms
\begin{align*}
  \mu_{X\times Y} K_1 \overset{a}{\circ} \mu_{Y\times Z} K_2
& \to
Rq_{13\pi !!} q_{13d}^{-1} \, Rj_{d!!} j_\pi^{-1}
(\mu_M(K_1 \boxtimes K_2))  \\
& \to
Rq_{13\pi !!}q_{13d}^{-1} \,
 \, \mu_N  \, j^{-1} (K_1 \boxtimes K_2) \\
&\to
Rq_{13\pi !!}q_{13d}^{-1} \,
\mu_N (q_{13}^{-1} K_3 \otimes \omega_Y)   \\
&\to
\mu_{X\times Z} K_3,
\end{align*}
where in the second line we have applied
Proposition~\ref{prop:micro_im_inv}, in the third the hypothesis and
in the fourth Proposition~\ref{prop:micro_im_dir}.
\end{proof}
Now we give the $\sha$-module analog of the above result. For
$\me{F}\in\Mod(\sha_{T^*X\times T^*Y})$ and
$\me{G}\in\Mod(\sha_{T^*Y\times T^*Z})$ we set
$$
\me{F}\overset{a \sha}{\circ} \me{G}
= p_{13!!} (p_{12}^{a*}\me{F} \otimes_{\sha_N} p_{23}^*\me{G} )
\simeq p_{13!!} p^* ( \me{F} \underline{\boxtimes} \me{G} ).
$$
We note the morphisms in $\DC^+(\I(\C_{T^*X\times T^*Z}))$:
\begin{equation}
\label{eq:circa_circasha}
  \me{F}\overset{a}{\circ} \me{G}
\to Rp_{13!!} (p_{12}^{a*}\me{F} \otimes_{\sha_N} p_{23}^*\me{G} )
\isofrom \me{F}\overset{a \sha}{\circ} \me{G},
\end{equation}
where the second arrow is an isomorphism by
Proposition~\ref{prop:Amod_f_*acyclique}.
\begin{proposition}
  \label{prop:convol_sha}
  For $\sha$-modules $K_1 \in \Mod(\sha_{X\times Y})$, $K_2 \in
  \Mod(\sha_{Y\times Z})$ and $K_3 \in \Mod(\sha_{X\times Z})$ with a
  morphism $q_{12}^* K_1 \otimes q_{23}^* K_2 \to q_{13}^* K_3 \otimes
  \omega'_Y$ there exists a natural morphism
\begin{equation}
\label{eq:micro_convol_sha}
    \mu^\sha_{X\times Y} K_1 
\overset{a \sha}{\circ} \mu^\sha_{Y\times Z} K_2
\to \mu^\sha_{X\times Z} K_3,
\end{equation}
with the following property. Setting $U = \Tnz^*X \times \Tnz^* Y$, $V
= \Tnz^*Y \times \Tnz^*Z$, the restrictions of
morphisms~\eqref{eq:micro_convol} and~\eqref{eq:micro_convol_sha}
outside the zero section make a commutative diagram in
$\DC^+(\I(\C_{\Tnz^*X \times \Tnz^*Z)}))$:
$$
\xymatrix@R=5mm{
 (\mu_{X\times Y} K_1)_U \overset{a}{\circ} (\mu_{Y\times Z} K_2)_V
\ar[r] \ar[d]
& \mu_{X\times Z} K_3  \ar[d] \\
\mu^\sha_{X\times Y} K_1  
\overset{a \sha}{\circ} \mu^\sha_{Y\times Z} K_2
\ar[r] 
&\mu^\sha_{X\times Z} K_3 .
}
$$
\end{proposition}
\begin{proof}
  The proof is similar to the proof of Proposition~\ref{prop:convol},
  replacing operations in $\DC^+(\I(\C_\cdot))$ by the same operations
  in $\Mod(\sha_\cdot)$.  In particular
  Proposition~\ref{lem:base_change} gives the base change $p_{13!!}
  p^* = q_{13\pi !!}  \tau_{!!} k^* \, j_\pi^* \isofrom q_{13\pi !!}
  q_{13d}^* \, j_{d!!} j_\pi^*$, which is an isomorphism because
  $q_{13d}$ is an embedding and $j_d$ is smooth.  Then we use
  Propositions~\ref{prop:micro_im_inv_sha}
  and~\ref{prop:micro_im_dir_sha} instead of
  Propositions~\ref{prop:micro_im_inv} and~\ref{prop:micro_im_dir}.
  
  By Proposition~\ref{prop:im_inv_sha_subm} we have $q_{12}^* K_1
  \otimes q_{23}^* K_2 \simeq q_{12}^{-1} K_1 \otimes q_{23}^{-1} K_2$
  and $q_{13}^* K_3 \simeq q_{13}^{-1} K_3$ in $\DC^+(\I(\C_{X\times
    Y\times Z}))$. Hence the morphism in the hypothesis of the
  proposition yields a morphism $q_{12}^{-1} K_1 \otimes q_{23}^{-1}
  K_2 \to q_{13}^{-1} K_3 \otimes \omega_Y$ in $\DC^+(\I(\C_{X\times
    Y\times Z}))$ and we may apply Proposition~\ref{prop:convol}.  The
  vertical arrows in the diagram are given by the morphisms of
  functors $\mu \to \mu^\sha$ and~\eqref{eq:circa_circasha}.
\end{proof}
We are in fact only interested in the following example.  We assume
now that $X,Y,Z$ are complex analytic manifolds.  We use the
$\sha$-module $\oot_X$ and its variants introduced in
Definition~\ref{def:resolution_Ot}.  We set $K_1 = \oot_{X\times
  Y}^{(0,d^c_Y)}[d^c_Y]$, which gives a resolution of $\O_{X\times
  Y}^{t(0,d^c_Y)}[d^c_Y]$, $K_2 = \oot_{Y\times
  Z}^{(0,d^c_Z)}[d^c_Z]$, $K_3 = \oot_{X\times Z}^{(0,d^c_Z)}[d^c_Z]$.
With these notations morphism~\eqref{eq:convol_shc} yields a morphism
$q_{12}^* K_1 \otimes q_{23}^* K_2 \to q_{13}^* K_3 \otimes \omega'_Y$
and Proposition~\ref{prop:convol_sha} gives the microlocal
convolution:
\begin{equation}
  \label{eq:produit_micro_shc}
\mu^\sha_{X\times Y} \oot_{X\times Y}^{(0,d^c_Y)}[d^c_Y]
\overset{a \sha}{\circ}
\mu^\sha_{Y\times Z} \oot_{Y\times Z}^{(0,d^c_Z)}[d^c_Z]
\to
\mu^\sha_{X\times Z} \oot_{X\times Z}^{(0,d^c_Z)}[d^c_Z].
\end{equation}
This convolution product is associative, because the composition of
kernels $\overset{a \sha}{\circ}$ is associative, as well as the
integration morphism, by Fubini.

\section{Sheaves of morphisms}
We will in fact use the morphisms of the previous section in a
slightly more general situation, namely for complexes of the type
$\hom(\pi^{-1}F, \mu G)$, rather than $\mu G$.  For this we use the
following proposition.  Once again we recall the convolution for
sheaves and then build it for $\sha$-modules. To compare them we use
the convolution products for complexes $F,G$, $\me{F}, \me{G}$:
$$
F \overset{0}{\circ} G
 = q_{13!!} (q_{12}^{-1} F \otimes q_{23}^{-1} G),
\qquad
\me{F}\overset{a0}{\circ} \me{G}= p_{13!!} (p_{12}^{a-1}\me{F}
\otimes p_{23}^{-1}\me{G} ).
$$
\begin{proposition}
\label{prop:conv_morph}
We consider $F \in \CC^+(\I(\C_{X\times Y}))$, $G \in
\CC^+(\I(\C_{Y\times Z}))$, $\me{F} \in \CC^+(\I(\C_{T^*(X\times
  Y)}))$ and $\me{G} \in \CC^+(\I(\C_{T^*(Y\times Z)}))$, there exists
natural morphisms, respectively in $ \DC^+(\I(\C_{T^*(X\times Z)}))$
and $ \CC^+(\I(\C_{T^*(X\times Z)}))$:
\begin{align}
  \label{eq:micro_convol_hom1}
    \Rhom(\pi^{-1}_{X\times Y} F, \me{F})
\overset{a}{\circ}
 \Rhom(\pi^{-1}_{Y\times Z} G, \me{G})  
&\to 
\Rhom(\pi^{-1}_{X\times Z}(F \circ G) , 
\me{F} \overset{a}{\circ} \me{G}),
 \\
\label{eq:micro_convol_hom2}
  \hom(\pi^{-1}_{X\times Y} F, \me{F})
\overset{a0}{\circ}
 \hom(\pi^{-1}_{Y\times Z} G, \me{G})  
&\to 
\hom(\pi^{-1}_{X\times Z}(F \overset{0}{\circ} G) , 
\me{F} \overset{a0}{\circ} \me{G}).
\end{align}

For $F \in \CC^+(\C_{(X\times Y)_{sa}})$, $G \in \CC^+(\C_{(Y\times
  Z)_{sa}})$, $\me{F} \in \Mod(\sha_{T^*(X\times Y)})$ and $\me{G} \in
\Mod(\sha_{T^*(Y\times Z)})$, we also have the natural morphism
  \begin{equation}
    \label{eq:micro_convol_hom_sha}
    \hom(\pi^{-1}_{X\times Y} F, \me{F})
\overset{a \sha}{\circ}
 \hom(\pi^{-1}_{Y\times Z} G, \me{G})  
\to 
\hom(\pi^{-1}_{X\times Z}(F \overset{0}{\circ} G) , 
\me{F} \overset{a \sha}{\circ} \me{G}).
  \end{equation}
  These morphisms fit into the commutative diagram:
\begin{equation}
  \label{eq:diag_conv_morph}
\vcenter{
\xymatrix@R=5mm{
    \Rhom(\pi^{-1}_{X\times Y} F, \me{F})
\overset{a}{\circ}
 \Rhom(\pi^{-1}_{Y\times Z} G, \me{G})  
\ar[r] 
& \Rhom(\pi^{-1}_{X\times Z}(F \circ G) , 
\me{F} \overset{a}{\circ} \me{G})   \ar[d] \\
& \hom(\pi^{-1}_{X\times Z}(F \overset{0}{\circ} G) , 
\me{F} \overset{a}{\circ} \me{G})   \\
\hom(\pi^{-1}_{X\times Y} F, \me{F})
\overset{a 0}{\circ}
 \hom(\pi^{-1}_{Y\times Z} G, \me{G})  
\ar[uu]\ar[d] \ar[r] 
& \hom(\pi^{-1}_{X\times Z}(F \overset{0}{\circ} G) , 
\me{F} \overset{a0}{\circ} \me{G})   \ar[u] \ar[d]  \\
\hom(\pi^{-1}_{X\times Y} F, \me{F})
\overset{a \sha}{\circ}
 \hom(\pi^{-1}_{Y\times Z} G, \me{G})  
\ar[r]
& \hom(\pi^{-1}_{X\times Z}(F \overset{0}{\circ} G) , 
\me{F} \overset{a \sha}{\circ} \me{G})   .
} }
\end{equation}
\end{proposition}
\begin{proof}
  We first build morphism~\eqref{eq:micro_convol_hom1}, in the derived
  category.  We keep the notations of section~\ref{sec:con_kern}, in
  particular diagram~\eqref{eq:diag_noyaux}.  To simplify the
  notations we suppress some subscripts on $\pi^{-1}$. Let us denote
  by $LHS$ the left hand side of~\eqref{eq:micro_convol_hom1}. We have
  \begin{align*}
LHS &=      Rp_{13!!} p^{-1} (
  \Rhom(\pi^{-1} F, \me{F}) \boxtimes \Rhom(\pi^{-1} G, \me{G})
  )  \\
&\simeq
Rp_{13!!} p^{-1} 
\Rhom(\pi^{-1} (F\boxtimes G), \me{F}\boxtimes \me{G}) .
  \end{align*}
  We can enter the functor $p^{-1}$ inside the $\Rhom$, and use the
  morphism of functors $Rp_{13!!} \Rhom(\cdot,\cdot) \to
  \Rhom(Rp_{13*}(\cdot),Rp_{13!!}(\cdot))$.  Thus we obtain a morphism:
  \begin{equation}
    \label{eq:conv_morph1}
  LHS \to \Rhom(Rp_{13*}p^{-1}\pi^{-1} ( F \boxtimes G), 
\me{F} \overset{a}{\circ} \me{G}).
  \end{equation}
  We let $\sigma\cl N \times_{X\times Z}T^*(X\times Z) \to T^*N$ be
  induced by the inclusion of the zero section of $Y$ and we let
  $\pi'_N \cl N \times_{X\times Z}T^*(X\times Z) \to N$ be the
  projection.  Then $\pi_M \circ p \circ \sigma = j \circ \pi'_N$.
  Moreover, since we deal with conic sheaves, we have the isomorphism
  of functors $R\tau_* \simeq \sigma^{-1}$.  We also have a morphism
  $Rq_{13\pi !!} \to Rq_{13\pi *}$. We deduce the sequence of
  morphisms:
  \begin{equation}
    \label{eq:conv_morph2}
  \begin{split}
    Rp_{13*} p^{-1} \pi^{-1}_M
&\simeq
Rq_{13\pi *} R\tau_* p^{-1} \pi^{-1}_M \\
& \from
Rq_{13\pi !!} \sigma^{-1} p^{-1} \pi^{-1}_M \\
& \simeq
Rq_{13\pi !!} {\pi'}_N^{-1} j^{-1}  \\
& \simeq 
\pi^{-1}_{X\times Z} Rq_{13!!} j^{-1},
  \end{split}
  \end{equation}
  where the last isomorphism is a base change. So we obtain
  $\pi^{-1}_{X\times Z} Rq_{13!!} j^{-1} \to Rp_{13*} p^{-1}
  \pi^{-1}_M$ and composing this morphism with~\eqref{eq:conv_morph1}
  we deduce~\eqref{eq:micro_convol_hom1}.

\medskip

Morphism~\eqref{eq:micro_convol_hom2} in the category of complexes is
obtained in the same way.  In particular the analog of
morphism~\eqref{eq:conv_morph1} is obtained from the morphism of
complexes $p_{13!!} \hom(\cdot,\cdot) \to
\hom(p_{13*}(\cdot),p_{13!!}(\cdot))$.  Moreover, $\tau_*$ is exact on
conic sheaves, so that $\tau_* \simeq R\tau_* \simeq \sigma^{-1}$ and
we have a sequence of morphisms in the category of complexes analog
to~\eqref{eq:conv_morph2} (note that the base change formula is true
for complexes).

The top part of diagram~\eqref{eq:diag_conv_morph} is given by the
natural morphisms between functors and their derived functors.

\medskip

The difference between morphism~\eqref{eq:micro_convol_hom_sha} and
morphism~\eqref{eq:micro_convol_hom2} only concerns the right hand
side of the $\hom$ functors.  Namely we replace the functor
$\overset{a0}{\circ}$ by $\overset{a \sha}{\circ}$ and obtain the same
proof. The bottom part of diagram~\eqref{eq:diag_conv_morph} is then
given by the morphism of functors~\eqref{eq:circa_circasha}.
\end{proof}

\section{$\E$-modules}
In this section $X$ is a complex analytic manifold of complex
dimension $n = d^c_X$ and $\Delta$ denotes the diagonal of $X\times
X$.  We identify $T^*X$ and $T^*_\Delta (X\times X)$ by the first
projection.  We denote by $\E_X$ the sheaf of microdifferential
operators of finite order. This is a sheaf on $T^*X$ and its
restriction to $\Tnz^*X$ was interpreted using the tempered
microlocalization in~\cite{A94} (see also~\cite{P07b}), as follows. We
let $\gamma \cl \Tnz^*X \to \CP(X)$ be the projection to the complex
projective bundle associated to $T^*X$.  Then $\E_X \simeq \gamma^{-1}
\gamma_*(\ERf_X)$, where $\ERf_X$ is the sheaf on $T^*X \simeq
T^*_\Delta (X\times X)$:
$$
\ERf_X = \tmu hom(\C_\Delta, \O_{X\times X}^{(0,n)}[n]) .
$$
The product of $\ERf_X$ is defined in~\cite{A94} by the convolution
product for tempered microlocalization.  This can be defined also in
the language of ind-sheaves, following~\cite{KSIW06}.  We first define
$\Eindd_X \in \BDC(\I(\C_{T^*(X\times X)}))$ by
$$
\Eindd_X = \Rihom( \pi^{-1}\C_\Delta,
\mu_{X\times X} \O_{X\times X}^{t(0,n)}[n]),
$$
where $\O_{X\times X}^{t(0,n)}$, defined in~\eqref{eq:def_Ot} as an
object of $\BDC(\C_{(X\times X)_{sa}})$, is now considered in
$\BDC(\I(\C_{X\times X})$ using the functor $I_\tau$.  Thus $\Eindd_X$
has support on $T^*X \simeq T^*_\Delta (X\times X)$ but this doesn't
imply that it is the image of an ind-sheaf on $T^*X$.  We recall the
notations $p_1,p_2 \cl T^*(X\times X) \to T^*X$ for the projections,
$a\cl T^*X \to T^*X$ for the antipodal map and we define the embedding
$$
\delta' \cl T^*X \simeq T^*_\Delta (X\times X) \to
T^*(X\times X),
\quad
(x,\xi) \mapsto (x,x,\xi,-\xi).
$$
Since $\supp \Eindd_X = T^*_\Delta (X\times X)$ the morphisms of
functors $p_{1*} \to p_{1*} \delta'_* \delta'^{-1} = \delta'^{-1}$ and
$p_{2*} \to a^{-1} \delta'^{-1}$ induce isomorphisms:
\begin{equation}
  \label{invdiag=dirproj}
\delta'^{-1} \Eindd_X \simeq  p_{1*} \Eindd_X 
\simeq  a^{-1} p_{2*} \Eindd_X .
\end{equation}
We could write the same isomorphisms with $p_{i!!}$ instead of
$p_{i*}$ or their derived functors.
\begin{definition}
  We let $\Eind_X \in \BDC(\I(\C_{T^*X}))$ be the ind-sheaf on $T^*X$
  defined by~\eqref{invdiag=dirproj}.
\end{definition}
Since the functor $\alpha$ from ind-sheaves to sheaves commute with
direct image (or inverse image) we have $\ERf_X \simeq \alpha_{T^*X}
(\Eind_X)$.

The complex $\Eind_X$ comes with a product in the sense of
Definition~\ref{def:anneau_module}, defined as follows:
\begin{itemize}
\item [(i)] Using~\eqref{invdiag=dirproj} we see that $\Eind_X \otimes
  \Eind_X \simeq \delta'^{-1} (\Eindd_X \overset{a}{\circ} \Eindd_X)$.
  
\item [(ii)] We have $\C_\Delta \circ \C_\Delta = \C_\Delta$ and
  morphism~\eqref{eq:micro_convol_hom1}, with $X=Y=Z$, gives a
  morphism $\Eind_X \otimes \Eind_X \to \delta'^{-1} \Rihom(
  \pi^{-1}\C_\Delta, \mu \O_{X\times X}^{t(0,n)}\overset{a}{\circ} \mu
  \O_{X\times X}^{t(0,n)}[2n] )$.
  
\item [(iii)] The convolution product $Rq_{13!!} ( q_{12}^{-1}\O_{X
    \times X}^{t(0,n)}[n] \otimes q_{23}^{-1}\O_{X \times
    X}^{t(0,n)}[n] ) \to \O_{X \times X}^{t(0,n)}[n]$ together with
  Proposition~\ref{prop:convol} gives a morphism $\mu \O_{X\times
    X}^{t(0,n)}\overset{a}{\circ} \mu \O_{X\times X}^{t(0,n)}[2n] \to
  \mu \O_{X\times X}^{t(0,n)}[n]$.
\end{itemize}
The composition of (i)--(iii) defines the product $\Eind_X \otimes
\Eind_X \to \Eind_X$.  In the same way
Propositions~\ref{prop:conv_morph} and~\ref{prop:convol}, applied to
$X=Y$ and $Z$ a point, give an action of $\Eind_X$ on $\mu \O^t_X$, in
the sense of Definition~\ref{def:anneau_module}.  We deduce an action
of $\Eind_X$ on $\Rihom( \pi^{-1}F,\mu \O^t_X)$, for any $F\in
\BDC(\I(\C_X))$.

This product and this action are just morphisms in the derived
category and do not endow the complex $\Eind_X$ with a structure of
algebra. However, when we go back to the derived category of sheaves
with the functor $\alpha_{T^*X}$, the product gives a morphism $\ERf_X
\otimes \ERf_X \to \ERf_X$.  But $\ERf_X$ is a sheaf (i.e. concentrated
in degree $0$) and this morphism really endows $\ERf_X$ with a
structure of sheaf of algebras.  But this is not enough to define a
structure of $\ERf_X$-module on $\tmu hom(F, \O_X) \simeq \alpha_{T^*X}
\Rihom( \pi^{-1}F,\mu \O^t_X)$, which is in general not concentrated in
degree.

\medskip

To solve this problem we define a dg-algebra $\Esa_X$ on the site
$X_{sa}$ (and not merely an object in the derived category) such that
$\Eind_X \simeq I_\tau (\Esa_X)$.  We also define in the same way a
dg-module over $\Esa_X$ representing $\mu \O^t_X$.  In fact our
definition is exactly the previous one, but in the categories of
$\sha$-modules instead of the derived categories.
\begin{definition}
  We define a complex of sheaves on $T^*X_{sa} = T^*_\Delta (X\times
  X)_{sa}$
$$
\Esa_X = \delta'^{-1} \hom( \pi^{-1}\C_\Delta,
\mu^\sha_{X\times X} \oot_{X\times X}^{(0,n)}[n]) ,
$$
with a product defined as follows:
\begin{itemize}
\item [(i)] as in the case of $\Eind_X$,
  morphism~\eqref{eq:micro_convol_hom_sha} gives a morphism
$$
\Esa_X \otimes \Esa_X
\to \delta'^{-1} \hom( \pi^{-1}\C_\Delta,
\mu^\sha \oot_{X\times X}^{(0,n)}
\overset{a\sha}{\circ} \mu^\sha \oot_{X\times X}^{(0,n)}[2n] ),
$$
\item [(ii)] the convolution product~\eqref{eq:convol_shc} together
  with Proposition~\ref{prop:convol_sha} gives a morphism
$$
\mu^\sha \oot_{X\times X}^{(0,n)}
\overset{a\sha}{\circ} 
\mu^\sha \oot_{X\times  X}^{(0,n)}[2n]
\to \mu^\sha \oot_{X\times X}^{(0,n)}[n].
$$
\end{itemize}
The composition of (i) and (ii) defines the product $\Esa_X \otimes
\Esa_X \to \Esa_X$.

In the same way, Propositions~\ref{prop:conv_morph}
and~\ref{prop:convol_sha}, applied to $X=Y$ and $Z$ a point, give a
morphism $\Esa_X \otimes \mu^\sha_{X} \oot_{X} \to \mu^\sha_{X}
\oot_{X}$. 
\end{definition}
\begin{proposition}
  The morphisms introduced in the previous definition give $\Esa_X$ a
  structure of dg-algebra and give $\mu^\sha_{X} \oot_{X}$ a structure
  of dg-$\Esa_X$-module.
  
  Over $\Tnz^*X$, we have isomorphisms $\Eind_X \simeq I_\tau (\Esa_X)$
  and $\mu \O^t_X \simeq I_\tau (\mu^\sha_{X} \oot_{X})$. Through these
  isomorphisms the product of $\Esa_X$ and its action on $\mu^\sha_{X}
  \oot_{X}$ coincide with the product of $\Eind_X$ and its action on
  $\mu \O^t_X$ defined above.
\end{proposition}
\begin{proof}
  The complex $\Esa_X$ is a dg-algebra and $\mu^\sha_{X} \oot_{X}$ is
  a dg-$\Esa_X$-module because the product and the action are defined
  in categories of complexes, and not merely up to homotopy.
  
  Let us check that the product of $\Esa_X$ represents the product of
  $\Eind_X$ and that their action on $\mu^\sha_{X} \oot_{X}$ and $\mu
  \O^t_X$ are the same.  This is a consequence of
  diagram~\eqref{eq:diag_conv_morph} and
  Proposition~\ref{prop:convol_sha}; but in
  diagram~\eqref{eq:diag_conv_morph} some vertical arrows go in the
  wrong direction and the commutative diagram in
  Proposition~\ref{prop:convol_sha} requires a restriction outside the
  zero section. These problems are solved as follows.
  
  In diagram~\eqref{eq:diag_conv_morph} the vertical arrows are
  isomorphisms. Indeed, we consider the cases $\me{F} = \mu^\sha_{X}
  \oot_{X\times X}^{(0,n)}[n]$ and $\me{G} = \me{F}$ or $\me{G} =
  \mu^\sha_{X} \oot_X$.  Hence, by Proposition~\ref{prop:mu_sha_qinj},
  $\me{F}$ and $\me{G}$ consist of quasi-injective sheaves (on the
  site $T^*X_{sa}$), and so are acyclic for the functors
  $\hom(H,\cdot)$, when $H$ is constructible. In our cases the
  complexes $F$, $G$ in the diagram are $\C_\Delta$ or $\C_X$, so that
  the $\hom$ sheaves are isomorphic to the $\Rhom$.  For the
  composition of kernels $\circ$ we also have to compute a direct
  image. Since we deal with $\sha$-modules,
  Proposition~\ref{prop:Amod_f_*acyclique} implies that direct images
  and derived direct images coincide.  This proves that the vertical
  arrows are isomorphisms.
  
  This diagram can be extended to the right, using
  Proposition~\ref{prop:convol_sha}. We can use the commutative
  diagram of Proposition~\ref{prop:convol_sha} because of the
  following remark: setting $U = \Tnz^*X \times \Tnz^*X$, we have, on
  $\Tnz^*X \times T^*X$, $(\Eindd_X)_U \isoto \Eindd_X$.  Then, for
  the same reason as above, the right vertical arrows in this extended
  diagram are isomorphisms.
\end{proof}
We still have to make the link between $\Esa_X$ and $\ERf_X$.  We note
that $\rho^{-1}\Esa_X$ is quasi-isomorphic to $\ERf_X$. In particular
$\rho^{-1}\Esa_X$ has its cohomology concentrated in degree $0$ and we
have isomorphisms of sheaves:
$$
\ERf_X \simeq H^0(\rho^{-1}\Esa_X)
\simeq H^0( \alpha I_\tau (\Esa_X)).
$$
Moreover the structure of dg-algebra on $\Esa_X$ gives a structure
of dg-algebra on $\rho^{-1}\Esa_X$ and a structure of algebra on
$H^0(\rho^{-1}\Esa_X)$.  The above proposition implies that this
product induced on $\ERf_X$ coincides with the one defined previously.

We also have a structure of dg-$\Esa_X$-module on $\mu^\sha_{X}
\oot_{X}$; in particular it defines an object $I_\tau(\mu^\sha_{X}
\oot_{X}) \in \DC(I_\tau(\Esa_X))$.  For any $G \in
\DC^-(\I(\C_{T^*X}))$ the complex $\Rihom(G, I_\tau(\mu^\sha_{X}
\oot_{X}) )$ is thus also naturally defined as an object of
$\DC(I_\tau(\Esa_X))$.  For $G = \pi^{-1}F$, $F\in \DC^-(\I(\C_X))$,
we deduce that
$$
\tmu hom(F, \O_X) =
\alpha \Rihom(\pi^{-1} F , I_\tau(\mu^\sha_{X} \oot_{X}) )
\quad
\in \DC(\rho^{-1} \Esa_X),
$$
and, by construction, the corresponding action in $\DC(\C_{T^*X})$
$$
\rho^{-1} \Esa_X \otimes  \tmu hom(F, \O_X)
\to \tmu hom(F, \O_X)
$$
coincides with the action of $\ERf_X$ on $\tmu hom(F, \O_X)$ defined
above. Thus we are almost done, except that $\tmu hom(F, \O_X)$ is
defined as an object of $\DC(\rho^{-1} \Esa_X)$ rather than
$\DC(\ERf_X)$. But the dg-algebra $\rho^{-1} \Esa_X$ is
quasi-isomorphic to $\ERf_X$ and it just remains to apply
Corollary~\ref{cor:chgt_dg-alg}, as follows.

We have the quasi-isomorphisms of dg-algebras on $\Tnz^*X$
$$
\rho^{-1}  \Esa_X \xfrom{\phi_{\leq 0}}
\tau_{\leq 0} \rho^{-1}  \Esa_X \xto{\phi_0} \ERf_X,
$$
and the equivalence of categories $\phi_0^* \circ \phi_{\leq 0 *}
\cl \DC(\rho^{-1} \Esa_X) \isoto \DC(\ERf_X)$. We set $\E'_X =
\beta_{T^*X} (\rho^{-1} \Esa_X)$ so that we have an adjunction
morphism $\E'_X \to I_\tau(\Esa_X)$.  This morphism induces a functor
of restriction of scalars, and $\phi_0^* \circ \phi_{\leq 0 *}$
induces an equivalence of categories:
$$
r\cl \DC(I_\tau(\Esa_X)) \to \DC(\E'_X),
\qquad
\Phi\cl \DC(\E'_X) \isoto \DC(\beta_{T^*X} (\ERf_X)).
$$
Hence we obtain an object $\O_X^\mu = \Phi(r(I_\tau(\mu^\sha_{X}
\oot_{X}))) \in \DC(\beta_{T^*X} (\ERf_X)) )$, representing $\mu
\O^t_X$ and we can state the final result:
\begin{theorem}
\label{thm:tmuhom=Emod}
The object $\O_X^\mu \in \DC(\beta_{T^*X} (\ERf_X)) )$ defined above,
over $\Tnz^*X$, is send to $\mu_X \O^t_X$ in $\DC(\I(\C_{\Tnz^*X}))$ by
the forgetful functor. It satisfies moreover: for $F\in
\DC^-(\I(\C_X))$ the complex
$$
\alpha_{T^*X} \Rihom(\pi^{-1}F, \O_X^\mu )
$$
which is naturally defined in $\DC(\ERf_X)$, over $\Tnz^*X$, is
isomorphic in $\DC(\C_{\Tnz^*X})$ to $\tmu hom(F, \O_X)$ endowed with
its action of $\ERf_X$.
\end{theorem}

\end{document}